\documentclass[a4paper, 10pt, oneside, reqno]{amsart}

\usepackage[utf8]{inputenc}
\usepackage{amsfonts,amssymb,amsthm,amsmath}
\usepackage{mathrsfs}
\usepackage{geometry}
\usepackage{graphicx,graphics}
\usepackage{color}
\usepackage{tikz}
\usepackage{relsize}
\usepackage{tikz-cd}
\usepackage{comment}
\usepackage{subfiles}
\usepackage{stmaryrd}
\usepackage[T1]{fontenc}
\usepackage{hyperref}
\usepackage{todonotes}
\usepackage{cleveref}
\usepackage{mathtools}

\hypersetup{colorlinks = true, linkbordercolor = {white}, linkcolor = {blue}, citecolor = {blue}}

\numberwithin{equation}{section}

\newcommand\bbD{\mathbb D}

\theoremstyle{plain}
\newtheorem{thm}{Theorem}[subsection]
\newtheorem{lem}[thm]{Lemma}
\newtheorem{prop}[thm]{Proposition}
\newtheorem{corol}[thm]{Corollary}
\newtheorem{conj}[thm]{Conjecture}

\theoremstyle{definition}
\newtheorem{defi}[thm]{Definition}

\theoremstyle{remark}
\newtheorem{rque}[thm]{Remark}

\newtheorem{ex}[thm]{Example}

\usepackage{comment}
\usepackage[backend=bibtex,
style=alphabetic,
bibencoding=ascii,
maxbibnames=99,
minbibnames=99,
]{biblatex}
\def \F {\mathbb{F}}

\def \Fq {\F_q}

\def \Fqb {\overline{\F}_q}
\def \Flb {\overline{\F}_{\ell}}
\def \loc {\tn{loc}}
\def \ov {\overline}

\def \Z {\mathbb{Z}}

\def \Zlb {\overline{\Z}_{\ell}}

\def \Q {\mathbb{Q}}

\def \Qlb {\overline{\mathbb{Q}}_{\ell}}

\def \Ind {\mathrm{Ind}}

\def \Gm {\mathbb{G}_m}
\def \Ga {\mathbb{G}_a}
\def \Spec {\mathrm{Spec}}

\def \SL {\mathrm{SL}}

\def \Hom {\mathrm{Hom}}
\def \D {\mathbb{D}}

\def \1 {\mathbb{1}}
\def \ev {\mathrm{ev}}

\def \QCoh {\mathrm{QCoh}}
\def \IndCoh {\mathrm{IndCoh}}
\def \End {\mathrm{End}}
\def \Mod {\mathrm{Mod}}

\def \ind {\mathrm{ind}}

\def \Cat {\mathrm{Cat}}

\def \Weil {\mathrm{W}}
\def \Gal {\mathrm{Gal}}
\def \Coh {\mathrm{Coh}}

\def \cP {\mathcal{P}}
\def \Q {\mathcal{Q}}

\def \Res {\mathrm{Res}}

\def \pt {\mathrm{pt}}
\def \spec {\mathrm{spec}}

\def \pt {\mathrm{pt}}

\def \Rep {\mathrm{Rep}}

\def \lis {\mathrm{lis}}

\def \Perf {\mathrm{Perf}}

\def \spec {\mathrm{spec}}

\def \Ccal {\mathcal{C}}

\def \GL {\mathrm{GL}}
\def \Ocal {\mathcal{O}}

\def \autom {\mathrm{autom}}

\def \FinSet {\mathrm{FinSet}}

\def \Bun {\mathrm{Bun}}

\def \Sat {\mathrm{Sat}}
\def \ad {\mathrm{ad}}

\def \alg {\mathrm{alg}}

\def \GL {\mathrm{GL}}

\def \Alg {\mathrm{Alg}}

\def \Ocal {\mathcal{O}}

\def \loc {\mathrm{loc}}

\def \Kal {\mathrm{Kal}}
\def \Par {\mathrm{Par}}

\def \oblv {\mathrm{oblv}}
\def \Div {\mathrm{Div}}
\def \Ecal {\mathcal{E}}

\def \Spd {\mathrm{Spd}}
\def \Maps {\mathrm{Maps}}
\def \Inert {\mathrm{I}}
\def \CAlg {\mathrm{CAlg}}

\def \ex {\mathrm{ex}}
\def \Res {\mathrm{Res}}
\def \mc {\mathcal}
\def \tn {\textnormal}
\def \Spec {\mathrm{Spec}}

\def \Perfd {\mathrm{Perfd}}
\def \Isoc {\mathrm{Isoc}}
\def \Kott {\mathrm{Kott}}

\def \Bdr {\mathrm{B}_{\mathrm{dR}}}
\def \Div {\mathrm{Div}}
\def \Spa {\mathrm{Spa}}
\def \Hck {\mathrm{Hck}}
\def \Wcal {\mathcal{W}}
\def \nilp {\mathrm{nilp}}
\def \Hk {\mathrm{Hck}}
\def \Ge {\mathrm{Ge}}
\def \et {\mathrm{et}}
\def \alg {\mathrm{alg}}
\def \lb {[} 
\def \rb {]}
\DeclareFieldFormat[article]{title}{\mkbibemph{#1\isdot}}
\DeclareFieldFormat
  [article,incollection]
  {title}{\mkbibemph{#1\isdot}}

\title{The extended Fargues--Scholze spectral action}
\author{Peter Dillery and Arnaud Eteve}

\begin{document}

\maketitle

\textbf{Abstract.} Let $G$ be a connected reductive group over a non-archimedean local field. The main result of Fargues and Scholze \cite{Geometrization} for the geometrization of the local Langlands correspondence is the construction of a ``spectral action'' on the category of $\ell$-adic sheaves on $\Bun_G$, the stack of $G$-torsors on the Fargues--Fontaine curve. The goal of this paper is to prove a conjecture of Fargues which says that one can extend this construction to the larger stack $\Bun_G^e$ of $G$-torsors on the Kaletha gerbe over the curve, as introduced by Fargues \cite{Fargues22}. This ``extended spectral action'' allows for a version of the categorical local Langlands conjecture for an arbitrary connected reductive group $G$, and is the first such statement for those $G$ which are not extended pure inner forms of a quasi-split group, such as non-trivial inner forms of $\mathrm{SL}_{n}$. Finally, we prove this conjecture for tori, following the original argument of Zou \cite{Zou24} and, under the same assumptions as \cite{Zou26} (including connected center), we reduce the ``extended'' version of the categorical conjecture to the one in Fargues--Scholze. 

\tableofcontents

\section{Introduction}\label{sec:intro}

\subsection{Motivation and main results}
Let $F$ be a non-archimedean local field of residue characteristic $p > 0$ and $G$ a connected reductive group over $F$. 
In their seminal work, Fargues and Scholze \cite{Geometrization} proposed a categorical version of the local Langlands correspondence in which the two sides of the correspondence appear in terms of the following categories:
\begin{enumerate}
    \item On the automorphic side, one considers $\mc{D}_{\lis}(\Bun_G)$ the category of $\Lambda \in \{\Flb, \Zlb, \Qlb\}$ étale sheaves on the stack of $G$-torsors on the Fargues--Fontaine curve. 
    \item On the Galois side, one considers the category of coherent sheaves on $\Par_G$ the stack of $L$-parameters, introduced independently in \cite{Geometrization}, \cite{ZhuCoherentSheaves} and \cite{DHKM}. 
\end{enumerate}
The main result of \cite{Geometrization} is the construction of a spectral action. 
\begin{thm}[\protect{\cite[Theorem X.0.1]{Geometrization}}]\label{thm:intro-FS-spectral-action}
    Assuming that $\ell$ does not divide the order of $\pi_{1}(\widehat{G})_{\tn{tor}}$ if $\Lambda \in \{\Flb, \Zlb\}$, there exists a canonically defined action, called the spectral action, of $\Perf(\Par_G)$ on $\mc{D}_{\lis}(\Bun_G)$. 
\end{thm}

Once this action is in place, one can formulate the following categorical local Langlands correspondence. 
\begin{conj}[\protect{\cite[Conjecture X.3.5]{Geometrization}}]\label{conj:geometrization-intro}
    Assume $G$ is quasi-split and $\ell$ as in Theorem \ref{thm:intro-FS-spectral-action}. For $(U,\psi)$ a fixed Whittaker datum for $G$, there exists a canonical $\Perf(\Par_G)$-linear and Whittaker normalized equivalence 
    \begin{equation}\label{eq:conj-geometrization-intro}
        \mc{D}_{\lis}(\Bun_G) \cong \IndCoh_{\nilp}(\Par_G).
    \end{equation}
\end{conj}

The Whittaker normalization in the statement of Conjecture \ref{conj:geometrization-intro} means that the equivalence \eqref{eq:conj-geometrization-intro} sends the Whittaker sheaf to the structure sheaf of $\Par_G$. This normalization is also the reason why it is necessary to assume that $G$ is quasi-split.

From the point of view of the theory of representations of $p$-adic groups, this categorical local Langlands conjecture misses certain inner forms of $G$. The inner forms that Conjecture \ref{conj:geometrization-intro} does encompass appear via the decomposition of $\Bun_{G}$ into geometric connected components along the fibers of the Kottwitz map $|\Bun_{G}| \xrightarrow{\kappa} \pi_{1}(G)_{\Gal_{F}}$. Sheaves supported on $\Bun_{G}^{\kappa=x}$ encode, among other valuable information, representations of $G_{x}(F)$, where $G_{x}$ is the inner form of $G$ obtained from the image of $x$ under the isomorphism $\pi_{1}(G)_{\Gal_{F}} \xrightarrow{\sim} B(G)_{\tn{basic}}$---such inner forms are called \textit{extended pure} inner forms. However, there are in general many inner forms which are not captured by $\Bun_{G}$; one critical example is the case of $\SL_n$, none of whose non-trivial inner forms are extended pure. 

To access all inner forms in a manner allowing for endoscopy, Kaletha \cite{Kal16} (building on ideas of Vogan \cite{Vog93}) introduced the notion of rigid inner forms of $G$ and subsequently proposed in \cite[Section 5.4]{Kal16} an extension of the refined local Langlands correspondence which incorporates all inner forms of $G$ and allows for a general statement of the endoscopic character identities. In more recent work, Fargues \cite{Fargues22} has proposed a geometric construction to explain how to frame rigid inner forms within the geometrization program. Specifically, one defines a certain gerbe\footnote{We use the notation $\mathfrak{X}_{S}$ for this gerbe only in the introduction for ease of exposition; in the main body of the paper, this gerbe is denoted by $[\mc{T}_{S}/\tilde{t}_{S}]$.} $\mathfrak{X}_{S}$ over the Fargues--Fontaine curve $X_S$ (for some test perfectoid space $S$) banded by $u := \varprojlim_{E/F,n} \mathrm{Res}_{E/F}(\mu_{n})$ and defines 
$$\Bun_G^e$$
the `extended stack of $G$-torsors' as the stack of $G$-torsors on the gerbe $\mathfrak{X}_{S}$ which are basic (the transition maps for $u$ are given by norms and $n/m$ power maps for $m \mid n$; see Definition \ref{defi:bandmap} for what we mean by ``basic'' here).  

\begin{thm}[\cite{Fargues22}]
    The stack $\Bun_G^e$ is a small $v$-stack and decomposes canonically as 
    \begin{equation*}
        \Bun_G^e \cong \bigsqcup_{\lambda \in \Hom_{F}(u,Z_{G})} \Bun_{G}^{e,\lambda}.
    \end{equation*}
   
    Moreover, there is a further non-canonical decomposition
    \begin{equation} \label{eq:decomp-bun-g-e-intro}
        \bigsqcup_{\lambda \in \Hom_{F}(u,Z_{G})} \Bun_{G}^{e,\lambda}  \cong  \bigsqcup_{\lambda \in \Hom_{F}(u,Z_{G})} \Bun_{G_{b_{\lambda}}},
    \end{equation}
     induced by twisting isomorphisms $\Bun_{G}^{e,\lambda} \to  \Bun_{G_{b_{\lambda}}}$ for all $\lambda$, where the collection of groups $\{G_{b_{\lambda}}\}_{\lambda}$ represents all rigid inner twists of $G$ up to extended pure inner twists. 
\end{thm}
In very much the same way that open strata of $\Bun_G$ are classifying stacks of the form $\pt/\underline{G_b(F)}$ for extended pure inner forms of $G$, open strata of $\Bun_G^e$ are classifying stacks of rigid inner twists of $G$.

On the coherent side, one introduces \cite[Section 12.1]{Fargues22} (originally appearing in \cite{Vog93} and \cite{Kaletha18}) the extended dual group $\widehat{G}^e$ defined as 
\begin{equation*}
    \widehat{G}^e = \varprojlim_{Z} \widehat{G/Z}
\end{equation*}
where the limit runs through the set of all finite central subgroups of $G$.
\begin{rque}
    If $G$ is semisimple, we have $\widehat{G}^e = \widehat{G}_{\tn{sc}}$ the simply connected cover of $\widehat{G}$.
\end{rque}
Following \cite{Kal16} and \cite{Fargues22}, one defines the extended stack of parameters. 
\begin{defi}
    Let $\Par_G^{\square}$ be the stack of framed $L$-parameters (also denoted by $Z^1(W_F, \widehat{G})$ in \cite{DHKM}), so that $\Par_G = \Par_G^{\square}/\widehat{G}$. The extended stack of parameters is 
    $$\Par_G^e = \Par_G^{\square}/\widehat{G}^e,$$
    where the $\widehat{G}^{e}$-action is inflated along $\widehat{G}^{e} \to \widehat{G}$. 
\end{defi}

It follows from Theorem \ref{thm:intro-FS-spectral-action} and the decomposition \eqref{eq:decomp-bun-g-e-intro} that the category $\mc{D}_{\lis}(\Bun_G^e)$ can be equipped with an action of $\Perf(\Par_G)$ (provided $\ell$ is not too small for modular coefficients). Our main theorem is as follows, which resolves part (1) of \cite[Conjecture 12.8]{Fargues22}: 
\begin{thm}[Theorem \ref{thm:extended-spectral-action-on-bun_G}]\label{thm:intro-extended-spectral-action}
    Assume $\ell$ is as in Theorem \ref{thm:intro-FS-spectral-action}. There is an action of $\Perf(\Par_G^e)$ on $\mc{D}_{\lis}(\Bun_G^e)$ which extends the Fargues--Scholze spectral action on each of the components coming from the decomposition \eqref{eq:decomp-bun-g-e-intro}. This action is canonical up to the choice of an isomorphism $\hat{\ov{F}} \xrightarrow{\sim} C^{\sharp}$, where $C$ is a fixed algebraically closed field of characteristic $p$.
\end{thm}

This extended spectral action is not just a result of applying the action from \cite{Geometrization} to each $\Bun_{G_{b_{\lambda}}}$ via the decomposition \eqref{eq:decomp-bun-g-e-intro}. This ``naive action'' only produces Hecke operators $T^{e}_{V}$ which are endomorphisms of each $\mc{D}_{\tn{lis}}(\Bun_{G}^{e,\lambda})$ and correspond to representations $V$ of $\widehat{G}^{e}$ which are inflated from $\widehat{G}$. Using this naive action, the chosen Whittaker sheaf for $G$, which is supported on $\Bun_{G}^{e,1} = \Bun_{G}$, does not interact at all with the nontrivial $\Bun_{G}^{e,\lambda}$. Since these $\Bun_{G}^{e,\lambda}$ are precisely the components of $\Bun_{G}^{e}$ encoding the inner forms of $G$ which are not extended pure, this approach does not suffice. 

By contrast, the Hecke operators we construct in order to prove Theorem \ref{thm:intro-extended-spectral-action} transport sheaves between distinct components $\Bun_{G}^{e,\lambda}$ for varying $\lambda$; taken together, they transitively permute the components occurring in \eqref{eq:decomp-bun-g-e-intro}. These extended Hecke operators, indexed by $V^{e} \in \Rep(\widehat{G}^{e})$, are built using a notion of modifications of torsors on $\mathfrak{X}_{S}$ that allow the inertial morphisms to change, which is the main geometric innovation of this paper.

With this action in place, we can formulate an extended version of the categorical conjecture.
\begin{conj}[\protect{\cite[Conjecture 12.8.2]{Fargues22}}, Conjecture \ref{conj:extended-categorical-equivalence}]\label{conj:intro-extended-cat-equiv}
    Assume that $\ell$ is as in Theorem \ref{thm:intro-FS-spectral-action} and $G$ is quasi-split. 
    There exists a $\Perf(\Par_G^e)$-linear and Whittaker normalized equivalence of categories
    $$\mc{D}_{\lis}(\Bun_G^e) \cong \IndCoh_{\nilp}(\Par_G^e).$$
\end{conj}

Following results of \cite{Zou24}, we prove the conjecture when $G$ is a torus. 
\begin{thm}[Theorem \ref{thm:case-of-tori}]
    If $G$ is a torus, then Conjecture \ref{conj:intro-extended-cat-equiv} holds for $G$.
\end{thm}

We also expect that it should be possible to reduce Conjecture \ref{conj:intro-extended-cat-equiv} to Conjecture \ref{conj:geometrization-intro}, which takes the form of the following statement. 
\begin{conj}[Conjecture \ref{conj:reduction-extended-to-classical}]\label{conj:intro-reduction-to-classical}
    The natural functor 
    \begin{equation}\label{eq:reduction-to-classical-intro}
        \mc{D}_{\lis}(\Bun_G) \otimes_{\Ind\Perf(\Par_G)} \Ind\Perf(\Par_G^e) \to \mc{D}_{\lis}(\Bun_G^e)
    \end{equation}
    is an equivalence. 
\end{conj}

In particular, if Conjectures \ref{conj:intro-reduction-to-classical} and \ref{conj:geometrization-intro} hold, then Conjecture \ref{conj:intro-extended-cat-equiv} holds. 

\begin{thm}[Theorem \ref{thm:reduction-connected-center-case}]\label{thm:intro-reduction-to-classical}
    If $Z_G$ is connected and $H^1(F,Z_G) = 0$ then Conjecture \ref{conj:intro-reduction-to-classical} holds. 
\end{thm}

The $H^{1}$-vanishing assumption is inherited from the same assumption in \cite{Zou26}, an assumption-free version of which has been announced that we expect transfers easily to the extended setting. There is also the recent result of Hansen--Mann:
\begin{thm}[\cite{HansenMann26}]
    Let $G = \mathrm{GL}_{n}$, $\Lambda = \ov{\mathbb{Q}_{\ell}}$, and $F$ a finite extension of $\mathbb{Q}_{p}$. Assuming \cite[Conjecture 1.6.2]{HansenMann26} (compatibility with Eisenstein functors), the spectral action induces a Whittaker-normalized natural equivalence 
    \begin{equation*}
        \mc{D}_{\tn{lis}}(\Bun_{\GL_{n}}) \xrightarrow{\sim} \IndCoh(\Par_{\GL_{n}}).
    \end{equation*}
\end{thm}

Combining the above theorem with  Theorem \ref{thm:intro-reduction-to-classical} implies that we can prove the case of $\GL_n$.

\begin{corol}
    If $F$ is a finite extension of $\mathbb{Q}_p$, $G = \GL_n$, $\Lambda = \Qlb$, and \cite[Conjecture 1.6.2]{HansenMann26} holds, then Conjecture \ref{conj:intro-extended-cat-equiv} holds. 
\end{corol}

In \cite{Kaletha18} and \cite{Dillery24} (for $F$ of characteristic zero or $p$, respectively), a general reduction of the rigid refined local Langlands correspondence is done by first reducing the correspondence to groups with connected center. The only assumption required for this classical reduction is compatibility of the rigid refined LLC along morphisms of reductive groups with central kernel and abelian cokernel. In the categorical local Langlands context, the action of the center is much more subtle, as is already seen in \cite{Zou26}. We expect, in particular, that reducing Conjecture \ref{conj:intro-extended-cat-equiv} to groups with connected center first involves reducing Conjecture \ref{conj:geometrization-intro} to groups with connected center, a problem which is presently open.  

In the geometric Langlands situation \cite{GLCV}, by contrast, such a compatibility with the center is established using $2$-categorical Fourier--Mukai methods. The authors of \emph{loc. cit.}, in particular, prove a version of the geometric Langlands equivalence for inner forms of $G$. In Section \ref{sec:reduction-perspectives}, we compare the construction provided by Theorem \ref{thm:intro-extended-spectral-action} with a potential analogous construction of \cite{GLCV} in the Fargues--Scholze context. 

\subsection{Outline of the construction}

Let us outline the key steps in the proof of Theorem \ref{thm:intro-extended-spectral-action}, assuming now that the ring of coefficients $\Lambda$ is a $\Z_{\ell}[\sqrt{q}]$-algebra. First note that the stack $\Par_G^e$ fits in the following Cartesian diagram 
\[\begin{tikzcd}
	{\Par_G^e} & {\pt/\widehat{G}^e} \\
	{\Par_G} & {\pt/\widehat{G}}
	\arrow[from=1-1, to=1-2]
	\arrow[from=1-1, to=2-1]
	\arrow[from=1-2, to=2-2]
	\arrow[from=2-1, to=2-2]
\end{tikzcd}\]

\begin{thm}[Theorem \ref{thm:relative-tensor-product}]
    Assume $\ell$ is as in Theorem \ref{thm:intro-FS-spectral-action}.
    The above diagram induces an equivalence of categories 
    $$\Ind\Perf(\Par_G^e) = \Ind\Perf(\Par_G) \otimes_{\Ind\Perf(\pt/\widehat{G})} \Ind\Perf(\pt/\widehat{G}^e).$$
\end{thm}

Consequently, the main piece of structure required to extend the action of $\Perf(\Par_G)$ is the data of an action of $\Perf(\pt/\widehat{G}^{e})$ extending the one of $\Perf(\pt/\widehat{G})$. 

\subsubsection{The Fargues--Scholze Hecke action}
Let us briefly recall the construction of the action of $\Perf(\pt/\widehat{G})$ from \cite{Geometrization}. We fix once and for all a geometric point 
$$\Spa(C) \xrightarrow{y} \Div^{1}:=\Div^1_F$$
where $C$ is an algebraically closed perfectoid field of characteristic $p > 0$. The action of $\Perf(\pt/\widehat{G})$ on $\mc{D}_{\lis}(\Bun_G)$ is given by Hecke functors, i.e. let $V \in \Rep_{\Lambda}\widehat{G}$ and let $\mc{S}_{V}$ denote the corresponding Satake sheaf on $\Hk^{\loc}_G$, the local Hecke stack of $G$ (over $\Spd(C)$). There is a diagram 
\[\begin{tikzcd}
	{\Hk^{\loc}_G} & {\Hk_G} & \\
	{\Bun_G} && {\Bun_G \times \Spd(C)}
	\arrow["\varepsilon"{description}, from=1-2, to=1-1]
	\arrow["{\overleftarrow{h}}"{description}, from=1-2, to=2-1]
	\arrow["{\overrightarrow{h}}"{description}, from=1-2, to=2-3]
\end{tikzcd}\]
where $\Hk_G$ is the moduli stack of modifications of $G$ torsors at the point $y$. The corresponding endofunctor of $\mc{D}_{\lis}(\Bun_G)$, when $\Lambda$ is a torsion ring is given by 
\begin{equation}\label{eq:hecke-operator-intro}
    A \mapsto \overrightarrow{h}_{!}(\overleftarrow{h}^*(A) \otimes \varepsilon^*(\mc{S}_V)),
\end{equation}
upon identifying $\mc{D}_{\lis}(\Bun_G) \cong \mc{D}_{\lis}(\Bun_G \times \Spd(C))$. In general, for $\ell$-adic coefficients, one needs to use the formalism of solid sheaves as in \cite[Section IX]{Geometrization}.

\subsubsection{The extended Hecke action}
Passing to the extended setup, the main geometric construction we provide is that of an extended Hecke stack $\Hk^e_G$ and its local version $\Hk^{e,\loc}_{G}$ fitting into a diagram 
\[\begin{tikzcd}
	{\Hk^{e,\loc}_G} & {\Hk_G^e} & \\
	{\Bun_G^e} && {\Bun_G^e \times \Spd(C).}
	\arrow["\varepsilon"{description}, from=1-2, to=1-1]
	\arrow["{\overleftarrow{h}}"{description}, from=1-2, to=2-1]
	\arrow["{\overrightarrow{h}}"{description}, from=1-2, to=2-3]
\end{tikzcd}\]
These stacks are defined in Section \ref{sec:extendedheckedef}. 
There are two key subtleties in the definition and construction of these Hecke stacks.

The first one is that the stack $\Hk_{G}^e$ is not naively the stack of tuples $(x, \mathcal{P}_0, \mathcal{P}_1, \phi)$ where $x : S \to \Div^1$ is a divisor on $X_S$, $\mathcal{P}_i$ are $G$-torsors on $\mathfrak{X}_{S}$ and $\phi$ is an isomorphism $\mathcal{P}_0 \cong \mathcal{P}_1$ away from the preimage of the Cartier divisor corresponding to $x$ in $\mathfrak{X}_{S}$. This definition yields a stack which we call $\Hk_{G}^{e, \tn{naive}}$ (resp. $\Hk_{G}^{e,\loc, \tn{naive}}$ for the local variant) defined in Section \ref{sec:hecke-stacks}. There is a natural inclusion 
$$\Hk_{G}^{e, \tn{naive}} \subset \Hk_{G}^{e},$$
which (after base-changing to $\Spd(C)$) is a union of connected components of $\Hk_{G}^e$. From the point of view of geometric Satake, the stack $\Hk_{G}^{e, \tn{naive}}$ only allows for modifications of torsors indexed by $\Rep(\widehat{G})$ and does not witness $\Rep(\widehat{G}^e)$. Correspondingly, the ``naive'' Hecke operators that one obtains in this way are the ones that yield the action of $\Perf(\Par_G)$ on $\mc{D}_{\lis}(\Bun_G^e)$. 

Roughly speaking, the central character of $V^{e} \in \Rep(\widehat{G}^{e})$ should capture how the corresponding Hecke operator $T^{e}_{V^{e}}$ permutes the connected components of $\Bun_{G}^{e}$. One motivation for this is \cite[Lemma 5.3.2]{Zou24}, which shows that for $\chi$-isotypic $V$, where $\chi$ is a character of $Z_{\widehat{G}}^{\Gal_{F}}$, the operator $T_{V}$ sends $\mc{D}_{\tn{lis}}(\Bun_{G}^{\kappa = x})$ to  $\mc{D}_{\tn{lis}}(\Bun_{G}^{\kappa = x-\chi})$, where we use the isomorphism $X^{*}(Z_{\widehat{G}}^{\Gal_{F}}) \xrightarrow{\sim} \pi_{1}(G)_{\Gal_{F}}$. 

In our setting, we replace $Z_{\widehat{G}}^{\Gal_{F}}$ with the covering 
\begin{equation*}
\widehat{Z} \hookrightarrow Z_{\widehat{G}}^{\Gal_{F},+} \twoheadrightarrow Z_{\widehat{G}}^{\Gal_{F}},
\end{equation*}
given by taking its preimage in $\widehat{G}^{e}$. Since $X^{*}(\widehat{Z})$ is canonically identified with $\Hom_{F}(u,Z_{G})$, one wants, for $\nu \in \Hom_{F}(u,Z_{G})$-isotypic $V^{e} \in \Rep(\widehat{G}^{e})$, a Hecke operator $T_{V^{e}}^{e}$ sending $\mc{D}_{\tn{lis}}(\Bun_{G}^{e,\lambda })$ to  $\mc{D}_{\tn{lis}}(\Bun_{G}^{e,\lambda-\nu})$. In other words, the restriction of a $\widehat{G}^{e}$-representation to $\widehat{Z} = \tn{ker}[\widehat{G}^{e} \to \widehat{G}]$ measures how the corresponding Hecke operator permutes the inertial decomposition \eqref{eq:decomp-bun-g-e-intro}; for example, the action of $\Perf(\Par_{G})$ preserves this decomposition.

It is thus critical in this extended setting to define a Hecke stack which allows the inertial morphism $\lambda$ to change. To achieve this, we introduce (Definitions \ref{defi:elochck} and \ref{defi:eglobhck}) \textit{twisted modifications} between $G$-torsors on $\mathfrak{X}_{S}$, defined by twisting by certain $Z_G$-torsors which exist only on the punctured curve (or punctured disk in the local case). 

The resulting effect of introducing these twists is that one, \emph{a priori}, loses the convolution structure of the local Hecke stacks. This is the second subtlety of the construction: We have to also redefine a convolution product to equip the category of sheaves on $\Hk_{G}^{e, \loc}$ with a monoidal structure. This is done in Section \ref{sec:Hckconv}. Doing convolution in this way depends on the isomorphism $\hat{\ov{F}} \xrightarrow{\sim} C^{\sharp}$ and thus does not glue over $\Div^1$. As such, the corresponding Hecke operator only exists over $\Spd(C)$---it turns out that this is an expected feature of the construction, see Remark \ref{rque:non-gluing-to-div-1}. 

Once the convolution is suitably constructed, we can move on to proving the relevant Geometric Satake equivalence. 
\begin{defi}[Definition \ref{defi:satake-category}]
    Let $\Sat_G^e \subset \mc{D}(\Hk^{e, \loc}_G)$ denote the full subcategory of perverse sheaves on $\Hk^{e, \loc}_G$ which are $\tn{ULA}$ over $\Spd(C)$ and $\Lambda$-flat. 
\end{defi}

\begin{thm}[Extended Geometric Satake, Theorem \ref{thm:satake}]
    There is an equivalence of monoidal categories
    $$\Sat_G^e \cong \Coh^{\heartsuit, \Lambda-\tn{flat}}(\pt/\widehat{G}^e \times_{\pt/\widehat{G}} \pt/\widehat{G}^e)$$
    where the right-hand side is equipped with the convolution monoidal structure. 
\end{thm}

Consider the relative diagonal 
$$\Delta : \pt/\widehat{G}^e \to \pt/\widehat{G}^e \times_{\pt/\widehat{G}} \pt/\widehat{G}^e.$$
The functor $\Delta_*$ provides a monoidal functor $\Rep_{\Lambda}\widehat{G}^e \to \Sat_G^e$ which, in turn, yields the desired Hecke operators using the same formula as \eqref{eq:hecke-operator-intro}. We refer to Section \ref{sec:Extended-Hecke-Action} for the details. 

\subsection{The rigid refined local Langlands correspondence}

One immediate consequence of the extended spectral action constructed in this paper is a statement of the categorical local Langlands conjecture for arbitrary connected reductive groups. It is also likely that adapting the methods of \cite{GLCV} to the Fargues--Scholze setting, as explained in Section \ref{sec:reduction-perspectives}, gives another approach for producing this maximally general formulation of the CLLC. 

Returning temporarily to the non-categorical setting, we want to emphasize that the purpose of the rigid refined local Langlands conjecture is not just to formulate a parametrization of the $L$-packets for arbitrary connected reductive groups. Indeed, the first completely general version\footnote{By ``refined local Langlands conjectures'' in this context, we just mean a parametrization of each $L$-packet $\Pi_{\varphi}$ using some datum related to the centralizer of $\varphi$ in $\widehat{G}$.} of the refined local Langlands conjectures is already given in \cite{Arthur06} by taking the preimage of the $\widehat{G}_{\tn{ad}}$--image of $S_{\varphi} := Z_{\widehat{G}}(\varphi)$ in $\widehat{G}_{\tn{sc}}$ in order to parametrize the $L$-packets for all inner forms of a fixed quasi-split group. However, in the same paper, Arthur observes that this approach cannot yield normalized absolute transfer factors in general. 

By formulating a version of the refined local Langlands correspondence parametrized by rigid inner twists, Kaletha proves that, given a tempered $L$-parameter $\varphi$ for a connected reductive group $G$ with endoscopic datum $\mathfrak{e}$, one can always construct a normalized absolute transfer factor $\Delta$ for $(G,\varphi,\mathfrak{e})$. Defining such transfer factors is the primary motivation for the aforementioned rigid formulation \cite{Kal16} of the refined local Langlands correspondence, which thus both encompasses all connected reductive groups and allows for a general statement of the endoscopic character identities.

One of the eventual aims of the categorical local Langlands project writ large is to bring the notion of endoscopy for $p$-adic groups under the umbrella of Fargues--Scholze. Since rigid inner twists are required to state the endoscopic character identities for groups which are not pure inner forms of a quasi-split group, it is reasonable to believe that one needs to use the extended version of the CLLC given in this paper in order to achieve this goal in full generality.

Furthermore, we hope that in the future one can use the extended geometrization to prove a version of the refined local Langlands correspondence (inspired by an earlier version of \cite{Fargues22}), which, among other desiderata, predicts:

\begin{conj}[\protect{\cite[Conjecture 5.1]{Dillery26}}] For a fixed $L$-parameter $\varphi$ for quasi-split $G$ with Whittaker datum $\mathfrak{w}$, there is a bijection
\begin{equation}\label{eq:rigLLC}
 \Pi_{\varphi}^{e}  \xrightarrow{\iota_{\mathfrak{w}}}  \tn{Irr}(S_{\varphi}^{+,\natural}),
\end{equation}
where $\Pi_{\varphi}^{e}$ is a $|\Bun_{G,\tn{basic}}^{e}|$-indexed union of the $L$-packets for $\varphi$ and the corresponding inner forms of $G$ and (denoting by $S_{\varphi}^{+}$ the preimage of $S_{\varphi}$ in $\widehat{G}^{e}$)
\begin{equation*}
S_{\varphi}^{+,\natural} := \frac{S_{\varphi}^{+}}{(S_{\varphi}^{+} \cap \widehat{G}_{\tn{sc}})^{\circ}}.
\end{equation*}
\end{conj}

\subsection{Global Langlands}
We conclude by explaining one potential global application. As explained in the introduction to \cite{Kaletha18b}, rigid inner twists and (a variant of) the correspondence \eqref{eq:rigLLC} are required to state the conjectural multiplicity formula for discrete automorphic representations of a connected reductive group $G$, even when $Z_{G}$ is connected (more precisely, one needs $G$ to satisfy the Hasse principle and have connected center in order to use extended pure inner forms instead for this purpose). In the setting of a global function field one has the recent work \cite{DLH26} of Li-Huerta which reformulates the global Langlands conjecture in terms of the CLLC. We hope that synthesizing this paper and its techniques with \cite{DLH26} can make progress towards proving cases of the aforementioned multiplicity formula (cf. \cite[Conjecture 5.6]{Dillery26}) over global function fields.

\subsection{Notation and conventions}

\subsubsection{Fields}

We let $F$ be a non-archimedean local field with residue field $\Fq$ of characteristic $p$. We denote by $\breve{F}, \overline{F}$, respectively, the completion of the maximal unramified extension and an algebraic closure of $F$. We let $C$ be a fixed algebraically closed perfectoid field over $\Fqb$ together with an isomorphism $C^{\sharp} \xrightarrow{\sim} \hat{\ov{F}}$. 

We denote by $\Ocal_?$ the ring of integers of $? \in \{F, \breve{F}, \overline{F}, C\}$.

We shall denote by $\Gal_F, \Inert_F, P_F$ the absolute Galois group of $F$, its inertia subgroup and wild inertia respectively. We write $\Weil_F$ for the Weil group of $F$. 

\subsubsection{Geometry}

We denote by $\Perfd$ the category of perfectoid spaces over $\Fqb$, as in \cite[Section 7.1]{SW20} and by $\Perfd_C$ the category of perfectoid spaces over $\Spa(C)$. For $S \in \Perfd$ we denote by $X_{S}$ the associated relative Fargues--Fontaine curve, as in \cite[Definition II.1.15]{Geometrization}, and denote by $\Div^{1}$ the moduli stack over $\Perfd$ parametrizing degree $1$ relative Cartier divisors on $X_{S}$ (as in \cite[Definition II.1.19]{Geometrization}). For $C$ an algebraically closed perfectoid field, denote by $\Spd(C)$ the diamond associated to $\Spa(C,C^{+})$ (as in \cite[Section 8.3]{SW20}).

The datum of the isomorphism $C^{\sharp} \xrightarrow{\sim} \hat{\ov{F}}$ gives a point 
$$y : \Spa(C,C^{+}) \to \Div^1.$$
When the context is clear, we drop the $y$ and just write $\Spa(C) \to \Div^1$. The isomorphism $C^{\sharp} \xrightarrow{\sim} \hat{\ov{F}}$ also determines points $y^{(E)} \colon \Spa(C) \to \Div^{1}_{E}$ for all finite Galois $E/F$ which are compatible via the maps $X^{(E')}_{S} \to X_{S}^{(E)}$.

\subsubsection{Categories of sheaves}

We let $\ell \neq p$ be a prime and we let $\Lambda$ be a coefficient ring. When $\Lambda$ is torsion of residue characteristic $\ell$, for a diamond or $v$-stack $X$, we denote by 
$$\mc{D}(X, \Lambda)$$ 
the category of étale sheaves of $\Lambda$-modules on $X$. We shall write $\mc{D}(X)$ if $\Lambda$ is clear from the context. For $\ell$-adic coefficients (i.e. when $\Lambda = \Zlb, \Qlb$), we shall, as in \cite{Geometrization}, use the formalism of solid and lisse sheaves, see Section \ref{sec:Extended-Hecke-Action}. 

\subsubsection{Reductive groups}

We let $G$ be a quasi-split reductive group over $F$ together with a Borel pair $(B,T)$. We denote by $(X^*(T), X_*(T), \Phi, \Phi^{\vee}, \Delta, \Delta^{\vee})$, the based root datum of $(B,T)$ and by $W$ the Weyl group of $(B,T)$. We write $Z_G$ for the center of $G$ and $\mathbf{T}$ for the abstract Cartan of $G$.
We denote by $\widehat{G}$ the Langlands dual group, a canonically defined pinned reductive group over $\Lambda$ and by ${^L}G = \widehat{G} \rtimes \Gal_F$ with Galois action normalized so that for $\alpha \in \Delta^{\vee}$ (a simple root of $\widehat{G}$), we have a $\Gal_F$-equivariant isomorphism 
$$\widehat{G}_{\alpha} \cong \Ga(-1),$$
see \cite[Section VI.11]{Geometrization} and \cite[Section 6.2.2]{IntegralSatake} for the normalization of the Tate twist.

To work in the extended setting, we denote by $\widehat{G}^e$ the extended dual group as defined in \cite[Section 12.1]{Fargues22}. This is the (pro)-reductive group over $\Lambda$ given as
$$\widehat{G}^e = \varprojlim_Z \widehat{G/Z}$$
where the limit runs through the set of all finite central subgroups $Z \subset Z_{G}$.

\subsubsection{$\infty$-categories}

We shall write $\Pr$ for the category of presentable, stable, $\Lambda$-linear categories, equipped with its usual Lurie tensor product, we write $\Alg(\Pr)$ (resp. $\CAlg(\Pr)$) to denote the category of algebras in $\Pr$, this is the category of monoidal (resp. symmetric monoidal) categories in $\Pr$ whose tensor product commutes with colimits in both variables.

\subsubsection{Miscellaneous}
For $X$ a topological space and $A$ a topological abelian group, denote by $\mathcal{C}(X,A)$ the group of continuous functions from $X$ to $A$. When $A$ is an abstract abelian group and we do not specify the topology, then it is assumed that $A$ has the discrete topology.

\subsection{Acknowledgements}
The second author thanks the Max Planck Institute for Mathematics for its hospitality during the preparation of this paper. We thank Dori Bejleri, Patrick Brosnan, Drimik Roy Chowdhury, Tom Haines, David Hansen, Tasho Kaletha, Alex Youcis, and Konrad Zou for comments and discussions during the preparation of this paper and its subsequent drafts. We thank Thibaud van den Hove for his very helpful clarifications about Hecke stacks. We also thank Yifei Zhao for explaining to us the content of \cite{GLCV} and discussions during the preparation of this paper.

\section{Rigid inner forms and $\Bun_{G}^e$}\label{sec:rigid-inner-forms}

In this section we recall some fundamental geometric objects used in the paper.

\subsection{The gerbe $\Kal$}
We first recall the reinterpretation of $\Kal$ from \cite{Fargues22} as a gerbe over the Fargues--Fontaine curve.

\begin{defi}
\begin{enumerate}
\item{Define the pro-multiplicative group scheme $\mathbb{D} := \varprojlim_{n} \mathbb{G}_{m}$.}
    \item{We denote by $\Kott_{F}$ the Kottwitz gerbe (this is sometimes denoted by $\mc{E}^{\tn{iso}}$ and called the ``isocrystal gerbe'' in the literature but we choose $\Kott_{F}$ in order to be consistent with \cite{Fargues22}), which is the gerbe over $\Spec(F)$ corresponding to the Tannakian category $\Isoc_{F}$ of $F$-isocrystals, which is banded by $\mathbb{D}$.}

\end{enumerate}
    We will frequently just write $\mathrm{Kott}$ for $\mathrm{Kott}_{F}$ and reserve the notation $\mathrm{Kott}_{E}$ for the analogous object for a finite Galois extension $E/F$. We also write $\mathrm{Kott}$ to denote the corresponding \'{e}tale gerbe over $\mathrm{Spa}(F)$.
\end{defi}

\begin{lem}\label{lem:maptoKott}
    There is a canonical morphism over $\Spa(F)$
    \begin{equation*}
        X_{S} \to \Kott
    \end{equation*}
    for any $S \in \Perfd$.
\end{lem}

\begin{proof}
For any $S \in \Perfd$ there is an exact, faithful tensor functor
\begin{equation*}
\tn{Isoc}_{F} \to \tn{VB}(X_{S}),
\end{equation*}
where $\tn{VB}(X_{S})$ denotes the tensor category of vector bundles on $X_{S}$ (by \cite[Section III.2.1]{Geometrization}). We are then done because $\Kott$ is the gerbe associated to the Tannakian category $\Isoc_{F}$. 
\end{proof}

Let $E/F$ be a finite Galois extension. As explained in \cite[Section 4.1]{Fargues22}, we can define an isocrystal of slope $-1$ over $E$ as follows. First, one chooses an arbitrary slope-one isocrystal $(D,\varphi)$ over $E$, and then defines the $\breve{E}$-vector space
\begin{equation*}
V(D,\varphi) = \{f \colon D \to B_{\tn{cris}}^{(E)} | f \circ \tn{Frob}_{E} = \tn{Frob}_{E} \circ f, \hspace{1mm} f(D) \subseteq \tn{Fil}^{1} B_{\tn{cris}}^{(E)}\}
\end{equation*}
where $B_{\tn{cris}}^{(E)}$ is as defined in \cite[Section 1.2]{FF} (along with its natural filtration and Frobenius action). The desired isocrystal is defined as
\begin{equation*}
    \mathbb{L}_{E} := (D,\varphi)^{\vee} \otimes_{\breve{E}} V(D,\varphi)^{\vee}. 
\end{equation*}
The key property of $\mathbb{L}_{E}$ for our purposes is that if one replaces $(D,\varphi)$ with another slope-one isocrystal which produces $\mathbb{L}_{E}'$ via the above recipe, then there is a canonical isomorphism from $\mathbb{L}_{E}$ to $\mathbb{L}_{E}'$ (cf. \cite[Section 4.1]{Fargues22} for more details). Further, it is shown in Lemme 4.3 loc. cit. that, for $E'/E/F$ two such extensions, there is an identification between 
\begin{equation*}
    N_{E'/E} \mathbb{L}_{E'} := \bigotimes_{\sigma \in \Gal_{E'/E}} \sigma^{*}(\mathbb{L}_{E'}) \cong \mathbb{L}_{E}.
\end{equation*}

\begin{defi}\label{def:LE/F}
The datum of $\mathbb{L}_{E}$ defines a $\mathbb{G}_{m}$-torsor on $\mathrm{Kott}_{E}$. Recalling that $X_{S}^{(E)} \to X_{S}$ denotes the corresponding finite \'{e}tale cover of $X_{S}$, we can pull this torsor back via the map from Lemma \ref{lem:maptoKott} to a $\mathbb{G}_{m}$-torsor over $X_{S}^{(E)}$, denoted by $\mc{T}_{E/F,S}$. We can also pull the line bundle $\mathbb{L}_{E}$ back to $X^{(E)}_{S}$, which we denote by $\mc{L}_{E/F,S}$. Note that $\mc{T}_{E/F,S}$ is the sheaf of trivializations of $\mc{L}_{E/F,S}$.  
\end{defi}

We define (as in \cite[Section 3.1]{Kal16}, where we are using the version without quotients, see \cite[Remark 2.12]{BMD26} for a detailed explanation of the difference)
\begin{equation*}
u = \varprojlim_{E/F,n} \Res_{E/F}(\mu_{n}),
\end{equation*}
where the transition maps are given by field norms and $n/m$-power maps (for $m \mid n$). Similarly, set $\tilde{t} = \varprojlim_{E/F,n} \Res_{E/F}(\mathbb{G}_{m})$ and $t = \varprojlim_{E/F} \Res_{E/F}(\mathbb{G}_{m})$, which fit into the short exact sequence of multiplicative group schemes over $F$
\begin{equation*}
1 \to u \to \tilde{t} \to t \to 1.
\end{equation*}

The sequence of $\mathbb{G}_{m}$-torsors on each $\mathrm{Kott}_{E}$ obtained from $\{\mathbb{L}_{E}\}_{E/F}$ defines a $t$-torsor $\mathbb{T}$ on $\mathrm{Kott}_{F} =: \mathrm{Kott}$, and we can form the $u$-gerbe 
\begin{equation*}
[\mathbb{T}/\tilde{t}] \to \mathrm{Kott},
\end{equation*}
which is also a gerbe over $\Spec(F)$.

\begin{rque}
    The above construction of the system $\{\mathbb{L}_{E}\}_{E/F}$ is a nontrivial part of \cite{Fargues22} relying on ideas from $p$-adic Hodge theory, and we do not recall any details beyond the above summary and refer the reader to \cite[Section 4.1]{Fargues22} for a detailed discussion. The only properties of the system $\{\mathbb{L}_{E}\}_{E/F}$ that we use here are that it consists of isocrystals of slope $-1$, is unique up to canonical isomorphism, and is norm-compatible. In fact, for the purpose of understanding this paper, it is essentially harmless to replace the system $\{\mathbb{L}_{E}\}_{E/F}$ with the system $\{(D_{E}, \varphi_{E})\}$ of isocrystals determined by picking a norm-compatible system of uniformizers $\{\varpi_{E}\}_{E/F}$ for all finite Galois $E/F$. Using this system instead will produce the desired spectral action---the subtle point is that the resulting action will only be canonical up to acting on $\Bun_{G}^{e}$ (cf. Definition \ref{defi:BunGe}) via pulling back by the automorphism of $[\mathbb{T}/\tilde{t}]$ determined by the differential of a $0$-cochain valued in $u(\ov{F})$.
    
    Furthermore, one could just as well take the system of isocrystals to be the one obtained by taking the fiber of each $\mc{O}_{X_{S}^{(E)}}(y^{(E)})$ at the closed point of $X_{S}^{(E)}$ corresponding to $y^{(E)}$. In this way, one obtains a norm-compatible system of isocrystals of slope $-1$ without changing the number of choices made in this paper. 
    \end{rque}

\begin{defi}\label{defi:Tdef}
We denote by $\mc{T}_{S}$ the pro-\'{e}tale $t_{S}$-torsor over $X_{S}$ corresponding to the norm-compatible system of $\mathbb{G}_{m}$-torsors $\{\mc{L}_{E/F,S}\}_{E/F}$.
\end{defi}

We then obtain a $u$-gerbe over $X_{S}$ by taking the quotient stack
\begin{equation*}
    [\mc{T}_{S}/\tilde{t}_{S}] \to X_{S}.
\end{equation*}
When $F$ has mixed characteristic, this is a pro-\'{e}tale gerbe. When $F$ has equal characteristic, it is a gerbe for the v-topology. By \cite[Corollaire 7.4]{Fargues22}, there is a canonical identification
\begin{equation*}
    [\mc{T}_{S}/\tilde{t}_{S}] \xrightarrow{\sim} X_{S} \times_{\Kott} [\mathbb{T}/\tilde{t}].
\end{equation*}

We conclude this subsection by explaining how the gerbes $[\mc{T}_{S}/\tilde{t}_{S}]$ and $[\mathbb{T}/\tilde{t}]$ are related to Kaletha's original construction in \cite{Kal16}. 

\begin{prop}\label{prop:ucohom}
   We have the following canonical identifications:
    \begin{enumerate}
       \item $H^{1}_{\tn{fppf}}(F, u) = 0$;
     \item $H^{2}_{\tn{fppf}}(F,u) = \widehat{\Z}$.
   \end{enumerate}
\end{prop}

\begin{proof}
    This is \cite[Theorem 3.1]{Kal16} in the mixed characteristic case and \cite[Theorem 3.4]{Dillery23} in the equal characteristic case.
\end{proof}

The Kaletha class is defined as the class $[\Kal] \in H^{2}_{\tn{fppf}}(F, u)$ corresponding to $-1$ under the above identification. 

\begin{prop}
\begin{enumerate}
\item{The stack 
\begin{equation*}
[\mathbb{T}/\tilde{t}] \to \mathrm{Kott} \to \Spec(F)
\end{equation*}
is a $\mathbb{D} \times u$-gerbe and represents the class $[\mathrm{Kott}] \times [\Kal] \in H_{\tn{fppf}}^{2}(F, \mathbb{D} \times u)$. In particular, the $u$-gerbe 
\begin{equation*}
[\mathbb{T}/\tilde{t}] \times^{B(\mathbb{D} \times u)} Bu\to \Spec(F)
\end{equation*}
represents $[\Kal]$; denote this gerbe by $\Kal$.}

    \item{There is a canonical isomorphism of stacks over $X_{S}$
    \begin{equation*}
        [\mc{T}_{S}/\tilde{t}_{S}] \xrightarrow{\sim} X_{S} \times_{\tn{Spa}(F)} \tn{Kal}.
    \end{equation*}}
    \end{enumerate}
\end{prop}

\begin{proof}
    This follows from \cite[Th\'{e}or\`{e}me 5.6]{Fargues22} and \cite[Corollaire 7.4]{Fargues22}.
\end{proof}

\begin{rque}
    One can also directly define a $\tilde{t}$-gerbe $\mc{E}$ over $\Spec(F)$ by considering the direct limit (over all $E/F$ finite Galois and $n \in \mathbb{N}$) of the Tannakian categories $\mc{I}_{E/F,n}$ defined in \cite[Section 2.5]{BMD26}. It has the property that $\mc{E} = \Kal \times^{Bu} B\tilde{t}$.
\end{rque}

\subsection{The stack $\tn{Bun}_{G}^{e}$}

Let $G$ be an affine group scheme over $F$; we denote by $G_{S}$ the corresponding sheaf of groups on $X_{S}$ for $S \in \Perfd$. We equip the gerbe $[\mc{T}_{S}/\tilde{t}_{S}] \to X_{S}$, which is a v-stack over $S$, with its natural étale topology, and denote by $G_{S}$ the pullback of $G_{S}$ to a group sheaf on  $[\mc{T}_{S}/\tilde{t}_{S}]$ for this topology. When we talk about $G_{S}$-torsors on $[\mc{T}_{S}/\tilde{t}_{S}]$, we always mean with this \'{e}tale topology. 

\begin{rque} Occasionally we will talk about $Z_{S}$-torsors over $[\mc{T}_{S}/\tilde{t}_{S}]$ for $Z$ a finite central subgroup of $G$, and in such cases will use the v-topology on $[\mc{T}_{S}/\tilde{t}_{S}]$ when $F$ has equal characteristic. This situation will only arise when a $Z_{S}$-torsor for the v-topology acts on an \'{e}tale $G_{S}$-torsor by taking contracted products, which yields another \'{e}tale $G_{S}$-torsor and is thus completely harmless for our purposes. 
\end{rque}

Given a $G_{S}$-torsor $\cP$ on  $[\mc{T}_{S}/\tilde{t}_{S}]$, the action of the band determines a morphism of sheaves over $[\mc{T}_{S}/\tilde{t}_{S}]$
\begin{equation}\label{eq:bandmap}
    u_{S} \to \underline{\tn{Aut}}_{[\mc{T}_{S}/\tilde{t}_{S}]}(\cP)/G_{\tn{ad},S},
\end{equation}
see \cite[Lemma 2.27]{Dillery23} for more details. 

\begin{defi}\label{defi:bandmap}
    We say that a $G_{S}$-torsor $\cP$ on $[\mc{T}_{S}/\tilde{t}_{S}]$ is \textbf{Kal-basic} if the homomorphism \eqref{eq:bandmap} is obtained by composing a (uniquely-determined) homomorphism 
    \begin{equation*}
    \lambda_{\cP} \colon u_{S} \to Z_{G,S}
    \end{equation*}
    with the action of $Z_{G,S}$ on $\cP$ induced by its structure as a $G_{S}$-torsor.

    A priori, the map $\lambda_{\cP}$ is a morphism of group schemes over $[\mc{T}_{S}/\tilde{t}_{S}]$, but by \cite[Lemma 2.42]{Dillery23}, this canonically descends to a morphism of group schemes over $X_{S}$, which we also denote by $\lambda_{\cP}$.
\end{defi}

For $S \in \Perfd$ and a $G$-torsor $\cP \in \Bun_{G}^{e}(S)$ we obtain, by construction, a morphism $\lambda_{\cP} \in \Hom_{X_{S}}(u_{S}, Z_{G,S})$, which we call the \textbf{inertial homomorphism} of $\cP$. 

\begin{defi}\label{defi:BunGe}
\begin{enumerate}
   \item{For $G$ an affine group scheme over $F$ (usually taken to be a connected reductive group), we define $\Bun_{G}^{e}$ to be the functor on $\Perfd$ sending $S$ to the groupoid of all $\Kal$-basic $G_{S}$-torsors on $[\mc{T}_{S}/\tilde{t}_{S}]$.}

\item{For $Z$ a finite central subgroup of $G$, we define $\Bun_{Z \subset G}^{e}$ to be the subfunctor of $\Bun_{G}^{e}$ of $G$-torsors such that $\lambda_{\cP}$ factors through $Z_{S}$ for $\cP \in \Bun_{G}^{e}(S)$.}
   \end{enumerate}
\end{defi}

\begin{rque}
   It is clear that $\Bun_{G}^{e} = \varinjlim_{Z} \Bun_{Z \subset G}^{e}$. 
\end{rque}

\begin{defi}[\protect{\cite[D\'{e}finition 9.1]{Fargues22}}]\label{defi:BeGdef}
    For any affine group scheme $G$ over $F$, define the extended Kottwitz set, denoted by $B_{e}(G)$, to be the pointed set of isomorphism classes of all $\Kal$-basic $G$-torsors on $\mathrm{Kott} \times_{F} \Kal$ (where $\Kal$-basic has the analogous meaning as above). For a finite central subgroup $Z \subset G$ we define $B_{e}(Z \subset G)$ analogously. Define $B_{e}(G)_{\tn{basic}}$ to be the subset of torsors whose restriction to $\mathbb{D} \times u$ factors through $Z_{G}$.
\end{defi}

We now summarize a few key results about $\Bun_{G}^{e}$:

\begin{prop}[{\cite[Proposition 12.5, Exemple 12.6]{Fargues22}}]\label{prop:BunGestructure}
Let $G$ be a connected reductive group over $F$.
    \begin{enumerate}
       \item{The functors  $\Bun_{G}^{e}$ and  $\Bun_{Z \subset G}^{e}$ are small $v$-stacks.}
       \item{There are canonical identifications
       \begin{equation*}
           |\Bun_{G}^{e}| \xrightarrow{\sim} B_{e}(G)
       \end{equation*}
       and
       \begin{equation*}
            |\Bun_{Z \subset G}^{e}| \xrightarrow{\sim} B_{e}(Z \subset G).
       \end{equation*}       
       }
       \item{For a fixed $\lambda \in \Hom_{F}(u, Z_{G})$, define the substack $\Bun_{G}^{e,\lambda}$ to consist of all $\cP$ such that $\lambda_{\cP_{\bar{s}}} = \lambda$ for all geometric points $\bar{s}$ of $S$. We have a decomposition (as stacks)
       \begin{equation*}
           \Bun_{G}^{e} = \bigsqcup_{\lambda \in \Hom_{F}(u,Z_{G})} \Bun_{G}^{e,\lambda}.
       \end{equation*}
       The same decomposition holds for $\Bun_{Z \subset G}^{e}$ if we replace $Z_{G}$ by $Z$.

       Moreover, if $b \in B_{e}(G)$ is an element whose restriction to $u$ is $\lambda$ (which always exists), then twisting by $b$ gives an isomorphism of stacks
       \begin{equation*}
           \Bun_{G}^{e,\lambda} \xrightarrow{\sim} \Bun_{G_{b}}.
       \end{equation*}
       }
\item{For a fixed $b \in B_{e}(G)_{\tn{basic}}$, define the substack $\Bun_{G}^{e,b}$ to consist of all $\cP$ such that $\cP_{\bar{s}} \cong b$ (using $\tn{(2)}$) for all geometric points $\bar{s}$ of $S$. Then the semistable locus of $\Bun_{G}^{e}$ equals (as a stack) $\bigsqcup_{b \in B_{e}(G)_{\tn{basic}}}\Bun_{G}^{e,b}$.}

    \end{enumerate}
\end{prop}

\begin{proof}
    The first two statements are \cite[Proposition 12.5]{Fargues22} and the third and fourth follow immediately from Exemple 12.6 loc. cit.
\end{proof}

\begin{lem}\label{lem:inertialhom}
    We have a canonical identification
    \begin{equation*}
        \Hom_{X_{S}}(u_{S}, Z_{G,S}) = \prod_{\pi_{0}(|X_{S}|)} \Hom_{F}(u,Z_{G}).
    \end{equation*}
\end{lem}

\begin{proof}
First, we observe that any morphism 
 \begin{equation*}
u_{S} \times_{\Spa(F)} X_{S} \to Z_{G,S} \times_{\Spa(F)} X_{S}
 \end{equation*}
 is locally constant over $|X_{S}|$, as in the proof of \cite[Proposition 12.5]{Fargues22}, which allows us to assume that $|X_{S}|$ is connected. The same proof shows that any such morphism is locally constant over $|S|$, which reduces us further to the case $S = \Spa(C,C^{+})$, where it follows from the fact that the étale fundamental group of $X_{C}$ is $\Gal_{F}$, by \cite[Th\'{e}or\`{e}me 9.5.1]{FF}.
\end{proof}

\begin{rque}\label{rque:fingerbe}
\begin{enumerate} \item{Observe that, for any $S \in \Perfd_{C}$, the fact that $G$ is finite type implies that any $G_{S}$-torsor on $[\mc{T}_{S}/\tilde{t}_{S}]$, $[\mc{T}_{S}^{\circ}/\tilde{t}_{S}]$, $[\mc{T}_{S}^{\loc}/\tilde{t}_{S}]$, or $[\mc{T}_{S}^{\loc,\circ}/\tilde{t}_{S}]$ is pulled back from a finite-level version of $[\mc{T}_{S}/\tilde{t}_{S}]$ (respectively its punctured and/or local analogues) obtained by replacing $\mc{T}_{S}$ by $\mc{T}_{E/F,S}$ (as in Definition \ref{def:LE/F}) and $\tilde{t}_{S}$ by $\tilde{t}_{E/F,n,S}$, where $\tilde{t}_{E/F,n} = \Res_{E/F}(\mathbb{G}_{m}) \to  \Res_{E/F}(\mathbb{G}_{m}) = t_{E/F}$ is the covering by the $n$-power map with kernel $u_{E/F,n}$.}

\item{It will be useful for some calculations to choose, for once and for all, an $\ov{F}$-trivialization of the gerbe  $[\mathbb{T}/\tilde{t}] \to \Kott$, which determines a \v{C}ech $2$-cocycle $\xi \in u(\ov{F}^{\bigotimes_{F}3})$ and which we can pull back to a $2$-cocycle for the gerbe $[\mc{T}_{S}/\tilde{t}_{S}]$, also denoted by $\xi$. This determines analogous $2$-cocycles $\xi_{E/F,n}$ for each finite level $u_{E/F,n}$-gerbe $[\mc{T}_{E/F,n,S}/\tilde{t}_{E/F,n,S}]$ in (1)---this cocycle is a priori valued in $u_{E/F,n}(\ov{F}^{\bigotimes_{F}3})$, but in fact takes values in $u_{E/F,n}(E')$ for finite Galois $E'/F$ trivializing $[\mc{T}_{E/F,n,S}/\tilde{t}_{E/F,n,S}]$ (for example, choose $E \subset E'$ such that $n \mid [E' \colon E]$).}

\end{enumerate}
\end{rque}

Since $\Bun_{G}^{e}$ is a v-stack, it makes sense to define, for $\Lambda$ a $\Z_{\ell}$-algebra, the category $\mc{D}_{\tn{lis}}(\Bun_{G}^{e}, \Lambda)$ as in \cite[Definition VII.6.1]{Geometrization} in this context.

\begin{corol}\label{cor:Dlise}
\begin{enumerate}
\item{For $b \in B_{e}(G)$, we have a canonical morphism of v-stacks $\Bun_{G}^{e,b} \to [\pt/G_{b}(F)]$. Moreover, the pullback functors
\begin{equation*}
    \mc{D}_{\tn{lis}}(\Bun_{G}^{e,b}, \Lambda) \to \mc{D}_{\tn{lis}}([\pt/G_{b}(F)], \Lambda) \to \mc{D}_{\tn{lis}}([\Spa(C)/G_{b}(F)], \Lambda)
\end{equation*}
are equivalences, and all of these categories are naturally equivalent to the category of smooth representations of $G_{b}(F)$ on discrete $\Lambda$-modules.}
\item{Pullback induces an equivalence
\begin{equation*}
\mc{D}_{\tn{lis}}(\Bun_{G}^{e}, \Lambda) \to \mc{D}_{\tn{lis}}(\Bun_{G}^{e} \times \Spd(C), \Lambda),
\end{equation*}
similarly with $\Bun_{G}^{e}$ replaced by $\Bun_{G}^{e,\lambda}$ for $\lambda \in \Hom_{F}(u,Z)$ or $\Bun_{G}^{e,b}$ for $b \in B_{e}(G)_{\tn{basic}}$.
}
\end{enumerate}
\end{corol}

\begin{proof}
    This result comes from combining Proposition \ref{prop:BunGestructure} with \cite[Proposition VII.7.1]{Geometrization} and \cite[Proposition VIII.7.3]{Geometrization}.
\end{proof}

\subsection{Connected components}
Recall from \cite[D\'{e}finition 9.10]{Fargues22} that one has an ``extended'' version of the Galois coinvariants of the fundamental group, defined as follows. Recall that $\Phi$ denotes the (absolute) root system of $T$ (with coroots $\Phi^{\vee}$) and set $\langle \Phi \rangle^{*} := \{\nu \in X_{*}(T)_{\mathbb{Q}} | \langle \alpha, \nu \rangle \in \Z \hspace{1mm} \forall \alpha \in \Phi\}$. The group in question is given by
\begin{equation*}
\pi_{1}(G)^{e}_{\Gal_{F}} := \frac{\langle \Phi \rangle^{*}}{\langle \Phi^{\vee} \rangle + I_{\Gal_{F}} \cdot X_{*}(T)},
\end{equation*}
where $I_{\Gal_{F}}$ denotes the augmentation ideal for $\Gal_{F}$.

Recall from \cite[Section 9.5]{Fargues22} that we have an exact sequence
\begin{equation}\label{eq:pi1eSES}
    0 \to \pi_{1}(G)_{\Gal_{F}} \to \pi_{1}(G)^{e}_{\Gal_{F}}  \to \Hom_{F}(u,Z_{G}) \to 0,
\end{equation}
and also that there is an ``extended Kottwitz map'' (recalling Definition \ref{defi:BeGdef})
\begin{equation*}
    B_{e}(G) \xrightarrow{\kappa}  \pi_{1}(G)^{e}_{\Gal_{F}}
\end{equation*}
whose restriction to $B_{e}(G)_{\tn{basic}}$ is a bijection (thus giving the left-hand side the structure of an abelian group) and whose restriction to $B(G)$ recovers the usual Kottwitz map (see \cite[Section 9.5.3]{Fargues22} for these claims).

It follows from Proposition \ref{prop:BunGestructure} that there is a map
\begin{equation*}
    |\Bun_{G}^{e}| \xrightarrow{\kappa} \pi_{1}(G)^{e}_{\Gal_{F}}
\end{equation*}
which Fargues proves (\cite[Proposition 12.5]{Fargues22}) is locally constant. 

By combining Proposition \ref{prop:BunGestructure} with the proof of \cite[Proposition VII.7.3]{Geometrization} and grouping together the fibers of $\kappa$ by their image along the surjection $\pi_{1}(G)^{e}_{\Gal_{F}}  \to \Hom_{F}(u,Z_{G})$, we obtain:
\begin{prop} \begin{enumerate}
   \item{There is an infinite semi-orthogonal decomposition of  $\mc{D}_{\tn{lis}}(\Bun_{G}^{e},\Lambda)$ into $\mc{D}_{\tn{lis}}(\Bun_{G}^{e,b},\Lambda)$ as $b$ ranges over all $B_{e}(G)$. In particular, we have  
    \begin{equation*}
        \mc{D}_{\tn{lis}}(\Bun_{G}^{e},\Lambda) \xrightarrow{\sim} \prod_{x \in \pi_{1}(G)^{e}_{\Gal_{F}}}  \mc{D}_{\tn{lis}}(\Bun_{G}^{e,\kappa = x},\Lambda).
    \end{equation*}}
    \item{For a fixed $\lambda \in  \Hom_{F}(u,Z_{G})$, there is a semi-orthogonal decomposition \begin{equation*}
\mc{D}_{\tn{lis}}(\Bun_{G}^{e,\lambda},\Lambda) =  \prod_{x \in \pi_{1}(G)^{e}_{\Gal_{F}}, x|_{u} = \lambda}  \mc{D}_{\tn{lis}}(\Bun_{G}^{e,\kappa = x},\Lambda) \xrightarrow{\sim} 
 \prod_{x \in \pi_{1}(G_{b_{\lambda}})_{\Gal_{F}}}  \mc{D}_{\tn{lis}}(\Bun_{G_{b_{\lambda}}}^{\kappa = x},\Lambda),
\end{equation*}
   where $b_{\lambda} \in B_{e}(G)$ lifts $\lambda$ and the second isomorphism is obtained via twisting (cf. Proposition \ref{prop:BunGestructure}).}
   \end{enumerate}
\end{prop}

\subsection{An affine Grassmannian calculation}\label{sec:affGrasscalc}
The purpose of this subsection is to provide some calculations which will be used in \S \ref{sec:extendedheckedef}.

Consider the affine Grassmannian 
\begin{equation*}
\mathrm{Gr}_{t} := \varinjlim_{E/F} \tn{Gr}_{\tn{Res}_{E/F}(\mathbb{G}_{m})}
\end{equation*}
for $t$, as defined in \cite[Section 19]{SW20} (this direct limit is harmless because sheafification commutes with filtered colimits). A compatible system of untilts $\{\Spa(C^{\sharp},C^{\sharp,+}) \xrightarrow{{y^{(E)}}} X_{C}^{(E)}\}_{E/F}$, which we have fixed (cf. \S \ref{sec:choiceoflifts} for a discussion of this choice), provides an identification

\begin{equation*}
    \mathrm{Gr}_{t}(C) \xrightarrow{\sim} X_{*}(t),
\end{equation*}
and we consider the homomorphism of abelian groups $X_{*}(t) \ni \tn{Sh}\colon X^{*}(t) = \mathcal{C}(\Gal_{F},\Z) \to \Z$ sending a function $f$ to $f(e)$ (here $\tn{Sh}$ stands for ``Shapiro'', which we use because the isomorphism $\Hom_{F}(t,\mathbf{T}) \xrightarrow{\sim} X_{*}(\mathbf{T})$ sends $f$ to $f \circ \tn{Sh}$). 

Denote by $\bbD_{C}$ the formal disk $\Spec(B_{\tn{dR}}^{+}(C^{\sharp}))$ corresponding to $y$ and $\bbD_{C}^{\circ}$ the associated punctured disk. 

\begin{lem}\label{lem:Grt} There is a unique trivialization $\psi \colon \mc{T}_{C}^{\circ} \xrightarrow{\sim} \underline{t_{C}}$ whose image in $\tn{Gr}_{t}(C)$ (by pulling back to $\bbD_{C}$) is $\tn{Sh}$. 
\end{lem}

\begin{proof}
    For a fixed finite Galois $E/F$, the torsor $\mc{T}_{E,C}$ is the sheaf of trivializations of $\mathcal{L}_{E/F,C}$ (as in Definition \ref{def:LE/F}), viewed as a sheaf of $\mathcal{O}_{X_{C}}$-modules. The system $\{y^{(E)}\}_{E/F}$ of untilts determines an isomorphism of projective systems of line bundles 
    \begin{equation*}
    \{\mc{L}_{E/F,C}\}_{E/F} \xrightarrow{\sim} \{\mc{O}_{X^{(E)}_{C}}(y^{(E)})\}.
    \end{equation*}
    
    Therefore, to trivialize $\mc{T}_{E,C}$ it suffices to find a norm-compatible system of trivializations of each line bundle $\mc{O}_{X^{(E)}_{C}}(y^{(E)})$ and, in turn, a norm-compatible system of uniformizing parameters $\xi_{E}$ for each $y^{(E)}$. Note that each $\xi_{E}$ can be chosen to lie in $\mathcal{O}_{X_{C}^{(E)}}(y^{(E)})(X^{(E)}_{C} \setminus \Gamma_{y^{(E)}})$, and so the resulting trivialization $\psi$ of $\mc{T}_{C}^{\circ}$ is defined on $X_{C}^{\circ}$. 

    Now $\psi$ has the correct image in $\tn{Gr}_{t}(C)$, because, by definition of the isomorphism 
    \begin{equation*}
    \tn{Gr}_{\Res_{E/F}(\mathbb{G}_{m})}(C)\xrightarrow{\sim} \Hom_{\Z}(\Z[\Gal_{E/F}],\Z),
    \end{equation*}
    the element $\tn{Sh}$ corresponds to the automorphism of the trivial torsor on $\bbD_{C}^{\circ}$ sending $e$ to $\xi_{E}^{-1} \in \Res_{E/F}(\mathbb{G}_{m})(\bbD_{C}^{\circ})$, which is what our map $\psi$ yields (after picking an auxiliary trivialization of $\mc{T}_{C}$ on $\bbD_{C}$ by using some closed point of $X_{C}$ different from the image of $y$). Taking the projective limit over $E/F$ gives the result.

   For uniqueness, we observe that the collection of $X_{C}^{\circ}$-trivializations of $\mc{T}_{C}^{\circ}$ is a $t(X_{C}^{\circ})$-torsor, and our choice of untilts $\{y^{(E)}\}$ determines an isomorphism 
   \begin{equation*}
       t(X_{C}^{\circ}) \xrightarrow{\sim} \varprojlim_{E/F} \Z[\Gal_{E/F}]_{0},
   \end{equation*}
   where the subscript ``$0$'' denotes the kernel of the augmentation map. Note that $X_{*}(t) = \varprojlim \Z[\Gal_{E/F}]$ and, via the identification $\tn{Gr}_{t}(C) \xrightarrow{\sim} X_{*}(t)$ fixed above, replacing $\psi$ with an $x\in t(X_{C}^{\circ})$-translate sends it to $\tn{Sh} \cdot f_{x} \in X_{*}(t)$, where $f_{x} \in X_{*}(t)$ is the cocharacter corresponding to $x$.
\end{proof}

\begin{lem}\label{lem:cohomvanish}
    For any algebraically closed perfectoid field $C'$, the cohomology groups $H^{1}(\bbD_{C'}, t)$ and $H^{1}(\bbD_{C'}^{\circ}, t)$ are trivial. In particular, the torsor $\mc{T}_{C'}$ has trivializations on $\bbD_{C'}$.
\end{lem}

\begin{proof}
    We only give the proof for $H^{1}(\bbD_{C'}^{\circ}, t)$; the other case is trivial by Henselianity.
    
    Note that the map $\Res_{E'/F}(\mathbb{G}_{m})(\bbD_{C'}^{\circ}) \to \Res_{E/F}(\mathbb{G}_{m})(\bbD_{C'}^{\circ})$ is surjective, because the kernel $K$ of this map is a torus and $H^{1}(\bbD_{C'}^{\circ}, K) =0$, due to the fact that $\Bdr(C')$ is a complete discrete valuation field with algebraically closed residue field (by, e.g., \cite[Section III.2.3]{Serre94}). 

    In particular, we have $\varprojlim^{(1)} H^{0}(\bbD_{C'}^{\circ}, \Res_{E/F}(\mathbb{G}_{m})) =0$, and thus we have an isomorphism
    \begin{equation*}
        H^{1}(\bbD_{C'}^{\circ}, t) \xrightarrow{\sim} \varprojlim H^{1}(\bbD_{C'}^{\circ}, \Res_{E/F}(\mathbb{G}_{m})) = 0,
    \end{equation*}
    where the last equality is again from the fact that $\Bdr(C')$ is a complete discrete valuation field with algebraically closed residue field. 
\end{proof}

Fix also $\tilde{\psi}$ an auxiliary trivialization of $\mc{T}_{C}^{\loc}$ on $\bbD_{C}$, which exists by the above lemma.

\section{The Extended Hecke Action}\label{sec:extendedheckestacks}
Continue to fix $G$ a quasi-split connected reductive group over $F$, with a fixed finite central subgroup $Z$.
\subsection{Hecke stacks}\label{sec:hecke-stacks}
Let $S = \Spa(R,R^+) \in \Perfd$. Given a map $x : S \to \Div^1$ corresponding to an untilt $\Spa(R^{\#}, R^{\#,+})$, we denote by 
$$\D_x = \Spec(\Bdr^{+}(R^{\sharp}))$$
and by 
$$\D_x^{\circ} = \Spec(\Bdr(R^{\sharp})).$$
More generally for $\underline{x} : S \to (\Div^1)^d$, we denote by $D_{\underline{x}} \xrightarrow{\iota_{\underline{x}}} X_{S}$ the corresponding (degree $d$) closed Cartier divisor on $X_S$,  by $\mathbb{D}_{\underline{x}}$ the spectrum of the ring $B_{(\Div^{1})^{I}}^{+}(R^{\sharp})$ and by $\D_{\underline{x}}^{\circ} = \D_{\underline{x}} - \underline{x}$ and $D_{\underline{x}}^{\circ} \xrightarrow{\iota_{\underline{x}}^{\circ}} X_{S}$ the corresponding punctured objects, see \cite[Section VI.1]{Geometrization} for more details.

We first recall the local and global Hecke stacks used in \cite{Geometrization}.

\begin{defi}[Hecke stacks with $I$ legs]
Fix a finite set $I$ equipped with an ordered partition $I = \bigsqcup_{i=0}^{k-1} I_{i}$.
Define the functor $\tn{Hck}_{G}^{\tn{loc},(I_{0}, \dots, I_{k})}$, resp. $\tn{Hck}_{G}^{(I_{0}, \dots, I_{k})}$ on $\Perfd$ by sending $S$ to the following data:
\begin{enumerate}
\item{A map $S \xrightarrow{\underline{x} = (\underline{x}_{i})_{i=0}^{k-1}}(\Div^{1})^{I}$;}
\item{A $k+1$-tuple $(\cP_{i})_{i=0}^{k}$ of $G$-torsors on $\bbD_{\underline{x}}$, resp. on $X_{S}$;}
\item{A sequence $(f_{i})_{i=0}^{k-1}$ of isomorphisms of $G$-torsors
\begin{equation*}
    \cP_{i}|_{\bbD_{\underline{x}} \setminus \underline{x}_{i}} \xrightarrow{\sim} \cP_{i+1}|_{\bbD_{\underline{x}} \setminus \underline{x}_{i}}, 
\end{equation*}
resp. of isomorphisms of $G$-torsors 
    \begin{equation*}
    \cP_{i}|_{ X_{S} \setminus \Gamma_{\underline{x}_{i}}} \xrightarrow{\sim} \cP_{i+1}|_{X_{S} \setminus \Gamma_{\underline{x}_{i}}}.
\end{equation*}
}
\end{enumerate}
\end{defi}

 \begin{rque}
     We note that for any $S = \Spa(R,R^{+}) \in \Perfd$ and $\underline{x} \in (\tn{Div}^{1})^{I}(S)$ there is an \' {e}tale $t_{S}$-torsor over $\Spec(B^{+}_{(\Div^{1})^{I}}(R^{\sharp}))$, abusively denoted by $\iota_{\underline{x}}^{*}\mc{T}_{S}$, uniquely determined by the pullback of $\mc{T}_{S}$ to the completion of $X_{S}$ along $D_{\underline{x}}$, as in \cite[Section VI.1]{Geometrization}. 
 \end{rque}

The analogue of Lemma \ref{lem:inertialhom} does not hold for $\bbD_{\underline{x}}$, so we need to make the following definition:

\begin{defi}\label{defi:bandloc}
    For $S = \Spa(R,R^{+})$ in $\Perfd$ and $\underline{x} \in (\Div^{1})^{I}(S)$, we say that $\cP$ a $G_{S}$-torsor on $[\iota_{\underline{x}}^{*}\mc{T}_{S}^{\loc}/\tilde{t}_{S}]$ is \textbf{band-localized} if it is Kal-basic and $\lambda_{\cP} \in \Hom_{\bbD_{\underline{x}}}(u_{S}, Z_{G,S})$ extends (uniquely) to a morphism in $\Hom_{X_{S}}(u_{S}, Z_{G,S})$. Note that $\lambda_{\cP}$ factors through a fixed finite subgroup $Z \subseteq Z_{G}$ if and only if this global extension does.
\end{defi}

We now give the first ``extended'' variant of these objects. It will turn out that, although this is the most natural definition of an ``extended local Hecke stack,'' it is not quite enough for our full purposes.

\begin{defi}\label{defi:hcknaive}
    Fix a finite set $I$ equipped with an ordered partition $I = \bigsqcup_{i=0}^{k-1} I_{i}$.
Define the functor $\tn{Hck}_{G}^{e,(I_{i}),\tn{naive}}$, resp. $\tn{Hck}_{G}^{e,\tn{loc},(I_{i}),\tn{naive}}$ on $\Perfd$ by sending $S$ to the following data:
\begin{enumerate}
\item{A map $S \xrightarrow{\underline{x} = (\underline{x}_{i})_{i=0}^{k-1}}(\Div^{1})^{I}$;}
\item{A $k+1$-tuple $(\cP_{i})_{i=0}^{k}$ of Kal-basic $G$-torsors on $[\mc{T}_{S}/\tilde{t}_{S}]$, resp. band-localized $G$-torsors on $[\iota_{\underline{x}}^{*}\mc{T}_{S}/\tilde{t}_{S}]$;}
\item{A sequence $(f_{i})_{i=0}^{k-1}$ of isomorphisms of $G$-torsors
\begin{equation*}
    \cP_{i}|_{[\mc{T}_{S}/\tilde{t}_{S}] \times_{X_{S}} X_{S} \setminus \Gamma_{\underline{x}_{i}}} \xrightarrow{\sim} \cP_{i+1}|_{[\mc{T}_{S}/\tilde{t}_{S}] \times_{X_{S}} X_{S} \setminus \Gamma_{\underline{x}_{i}}},
\end{equation*}
resp. of isomorphisms of $G$-torsors
\begin{equation*}
    \cP_{i}|_{[\iota_{\underline{x}}^{*}\mc{T}_{S}/\tilde{t}_{S}] \times_{\bbD_{\underline{x}}} \bbD_{\underline{x}} \setminus \underline{x}_{i}} \xrightarrow{\sim} \cP_{i+1}|_{[\iota_{\underline{x}}^{*}\mc{T}_{S}/\tilde{t}_{S}] \times_{\bbD_{\underline{x}}} \bbD_{\underline{x}} \setminus \underline{x}_{i}}.
\end{equation*}
}
\end{enumerate}
One can also define $\tn{Hck}_{Z \subset G}^{e,(I_{i}),\tn{naive}}$ and $\tn{Hck}_{Z \subset G}^{e,\tn{loc},(I_{i}),\tn{naive}}$ by taking the above definitions but insisting that all $\lambda_{\cP}$ factor through the fixed subgroup $Z$.
\end{defi}
All of the above four stacks are equipped with natural convolution operations as defined in \cite[Section VI.8]{Geometrization}. 

\begin{defi}\label{defi:hom} Denote by $\underline{\Hom}_{X}(u,Z)$ the sheaf on $\Perfd$ sending $S$ to $\Hom_{X_{S}}(u_{S},Z_{S})$. 
\end{defi}
We deduce that there are locally constant (by the proof of Lemma \ref{lem:inertialhom}) morphisms
\begin{equation*}
    \Hck_{G}^{e,(I_{i}),\tn{naive}}, \Hck_{G}^{e,\loc,(I_{i}),\tn{naive}}\to \underline{\Hom}_{X}(u,Z),
\end{equation*}
and define, for a fixed $\lambda \in \Hom_{F}(u,Z_{G})$, the substacks $\Hck_{G}^{e,(I_{i}),\tn{naive},\lambda}, \Hck_{G}^{e,\loc,(I_{i}),\tn{naive},\lambda}$ to be the preimages of $\lambda$ under these morphisms (one could rephrase this condition to match the statement of Proposition \ref{prop:BunGestructure} by insisting that $\lambda_{\cP_{i,\bar{s}}} = \lambda$ for all geometric points $\bar{s}$ of $S$). 

\begin{lem}\label{lem:naiveHck}
    For a fixed $\lambda \in \Hom_{F}(u,Z_{G})$ and a fixed $b \in B_{e}(G)_{\tn{basic}}$ whose restriction to $u$ is $\lambda$, we can twist by $b$ to obtain isomorphisms
    \begin{equation*}
        \Hck_{G}^{e,(I_{i}),\tn{naive},\lambda} \xrightarrow{\sim} \Hck_{G_{b}}^{{(I_{i})}}
    \end{equation*}
    and 
    \begin{equation*}
        \Hck_{G}^{e,\loc,(I_{i}),\tn{naive},\lambda} \xrightarrow{\sim} \Hck_{G_{b}}^{\loc,{(I_{i})}}.
    \end{equation*}
\end{lem}

\begin{proof}
    We have already seen in Proposition \ref{prop:BunGestructure}.(3) that twisting gives an isomorphism $\Bun_{G_{b}} \to \Bun_{G}^{e,\lambda}$, and this twisting bijection holds on the local disks $\bbD_{\underline{x}}$ as well and is functorial on isomorphisms (it is given by applying the functor $- \times^{G_{b}} \cP_{b}$ for the fixed $G$-torsor $\cP_{b}$ over $[\mc{T}_{C}/\tilde{t}_{C}]$ corresponding to $b$, where $\cP_{b}$ is viewed as a $(G_{b},G)$-bitorsor on $[\mc{T}_{C}/\tilde{t}_{C}]$).
\end{proof}

We deduce from combining Lemma \ref{lem:naiveHck} with \cite[Section IX]{Geometrization} that the stacks $\Hck_{G}^{e,(I_{i}),\tn{naive}}$ and $\Hck_{G}^{e,\loc,(I_{i}),\tn{naive}}$ give, for any finite set $I$ and $V \in \tn{Rep}((\widehat{G} \rtimes Q)^{I})$ (where $Q$ is a finite quotient of $\Weil_{F}$ through which the action on $\widehat{G}$ factors), a Hecke operator
\begin{equation*}
    T^{e}_{V} \colon \mc{D}_{\tn{lis}}(\Bun_{G}^{e},\Lambda) \to \mc{D}_{\tn{lis}}(\Bun_{G}^{e},\Lambda)^{B\Weil_{F}^{I}},
\end{equation*}
as in \cite[Theorem IX.0.1]{Geometrization}, which satisfies all of the properties as in the operators $T_{V}$ \emph{loc. cit.} 

We can therefore deduce immediately from \cite[Corollary X.1.3]{Geometrization}:

\begin{prop}\label{prop:naiveSA}
    There is a compactly supported $\Lambda$-linear action of $\tn{Perf}(Z^{1}(\Weil_{F},\widehat{G})_{\Lambda}/\widehat{G})$ on $\mc{D}_{\tn{lis}}(\Bun_{G}^{e},\Lambda)^{\omega}$ which is compatible with the Hecke action and gives, functorially in finite sets $I$, exact $\tn{Rep}_{\Lambda}(Q^{I})$-linear functors
    \begin{equation*}
        \tn{Rep}_{\Lambda}((\widehat{G} \rtimes Q)^{I}) \to \End_{\Lambda}(\mc{D}_{\tn{lis}}(\Bun_{G}^{e},\Lambda)^{\omega})^{B\Weil_{F}^{I}},
    \end{equation*}
    where $Q$ denotes a finite quotient of $\Weil_{F}$ through which the action on $\widehat{G}$ factors and the superscript $\omega$ denotes $\mc{D}_{\tn{lis}}(\Bun_{G}^{e},\Lambda)^{\omega} \subseteq \mc{D}_{\tn{lis}}(\Bun_{G}^{e},\Lambda)$ the subcategory of compact objects.
\end{prop}

To summarize the situation, the above proposition and the resulting spectral action are obtained from the theory of \cite{Geometrization} by decomposing
\begin{equation}\label{eq:naivedecomp}
 \Hck_{G}^{e,(I_{i}),\tn{naive}} = \bigsqcup_{\lambda \in \Hom_{F}(u,Z)}  \Hck_{G}^{e,(I_{i}),\tn{naive},\lambda} \xrightarrow{\sim}  \bigsqcup_{\lambda \in \Hom_{F}(u,Z)}  \Hck_{G_{b_{\lambda}}}^{(I_{i})}
\end{equation}
(and its local analogue) using Lemma \ref{lem:naiveHck}.

In order to construct the full ``extended spectral action'' from \cite[Conjecture 12.8]{Fargues22}, one needs a version of the Hecke stacks $\Hck_{G}^{e,(I_{i}),\tn{naive}}$, $\Hck_{G}^{e,\loc,(I_{i}),\tn{naive}}$ which allows ``intertwining'' between different $\Hom_{F}(u,Z)$--components of $\Bun_{G}^{e}$.

We conclude with a technical lemma which will be needed later:

\begin{lem}\label{lem:Grgerbe}
    The action of $L^{+}G$ on $\tn{Gr}_{G/Z,\Spd(C)}$ has the same orbits as that of $L^{+}(G/Z)$ and connected stabilizers.
\end{lem}

\begin{proof}
    In mixed characteristic the orbit claim follows from the fact that $L^{+}G \to L^{+}(G/Z)$ is surjective on geometric points $\Spa(C')$, since if $x \colon \Spd(C') \to \Spd(C) \xrightarrow{y} \Div^{1}$ is the point induced by $y$ then $H^{1}_{\tn{\'{e}t}}(\bbD_{x}, Z) = 0$ because $\Bdr^{+}(C^{',\sharp})$ is a strictly Henselian local ring with algebraically closed residue field.

    In the general case, observe that for any geometric point $\Spa(C')$ we have the commutative diagram
    \[ 
\begin{tikzcd}
    L^{+}T(C') \arrow{r} \arrow{d} & L^{+}(T/Z)(C') \arrow{r} \arrow{d} & H^{1}_{\tn{fppf}}(\bbD_{x}, Z) \arrow["\tn{id}"]{d} \\
     L^{+}G(C') \arrow{r} & L^{+}(G/Z)(C') \arrow{r} & H^{1}_{\tn{fppf}}(\bbD_{x}, Z),
\end{tikzcd}
    \]
    which implies that the morphism $L^{+}T \backslash L^{+}(T/Z) \to L^{+}G \backslash L^{+}(G/Z)$ is surjective. 

   Recall that, by the Cartan decomposition (e.g. \cite[Proposition 19.2.1]{SW20}), any point $[g] \in \tn{Gr}_{G/Z}(C')$ lies in the same $L^{+}(G/Z)(C')$-orbit as some $t \in LT(C')$. We then obtain the result by writing any $g \in L^{+}(G/Z)$ as $g = g't'$ for $g' \in L^{+}G(C')$ and $t' \in L^{+}(T/Z)(C')$ and observing that
   \begin{equation*}
       [gt] = [g't't] = [g'tt'] = [g't],
   \end{equation*}
giving the equality of the orbits.

We move to the stabilizer claim. Note that for all points $X$ of $\tn{Gr}_{G/Z}(C')$ the subgroup $Z=L^{+}Z$ is contained in $(L^{+}G)_{X}^{\circ}$. Indeed, $(L^{+}G)_{X}$ is the preimage of $(L^{+}(G/Z))_{X}$ in $L^{+}G$; as above, we can reduce to the case where $X = [t^{\bar{\nu}}]$ with $t^{\bar{\nu}} \in L(T/Z)$ and $\bar{\nu} \in X_{*}(T/Z)$, so that $(L^{+}G)^{\circ}_{X}$ contains $L^{+}T$ and thus also $Z = L^{+}Z$. This reduces us to showing that the image $I$ of $(L^{+}G)_{X}$ in $L^{+}(G/Z)$, again assuming $X = [t^{\bar{\nu}}]$, is connected. In mixed characteristic we are now done by surjectivity, since $(L^{+}(G/Z))_{X}$ is connected.

This image $I$ fits into the following sequence of v-sheaves
\begin{equation*}
    L^{+}T/L^{+}Z   \hookrightarrow I \twoheadrightarrow I/(L^{+}T),
\end{equation*}
and, by the above argument, the natural map $I/(L^{+}T) \to (L^{+}(G/Z))_{X}/(L^{+}(T/Z))$ is an isomorphism. Since both $L^{+}T/L^{+}Z$ and $(L^{+}(G/Z))_{X}/(L^{+}(T/Z))$ are connected, so is $I$.
\end{proof}

\subsection{Extended Hecke stacks}\label{sec:extendedheckedef}

We now define local and global Hecke stacks adapted to the extended moduli of bundles $\Bun_{G}^{e}$. Recall that we denote by $\mathbf{T}$ the abstract Cartan of $G$. 

    Applying the functor $\Hom_{F}(-,\mathbf{T})$ to the sequence 
    \begin{equation*}
1 \to u \to \tilde{t} \to t \to 1
    \end{equation*}
    and taking a fibered product yields the short exact sequence
    \begin{equation}\label{eq:HomZSES}
        1 \to \Hom_{F}(t, \mathbf{T}) \to \Hom_{F}(\tilde{t}, \mathbf{T}) \times_{\Hom_{F}(u,\mathbf{T})} \Hom_{F}(u,Z) \to \Hom_{F}(u,Z) \to 1,
    \end{equation}
    where the surjectivity comes from applying Shapiro's Lemma and using that $\tn{Ext}^{1}_{\Z}(X^{*}(\mathbf{T}), X^{*}(t)) = 0$.
    We fix, for once and for all, a normalized (set-theoretic) section 
\begin{equation*}
    s \colon \Hom_{F}(u,Z_{G}) \to \Hom_{F}(\tilde{t}, \mathbf{T}) \times_{\Hom_{F}(u,\mathbf{T})} \Hom_{F}(u,Z_{G}),
\end{equation*}
determining a section 
\begin{equation*}
s \colon \Hom_{F}(u,Z) \to \Hom_{F}(\tilde{t}, \mathbf{T}) \times_{\Hom_{F}(u,\mathbf{T})} \Hom_{F}(u,Z)
\end{equation*}
of the right-hand map in \eqref{eq:HomZSES} with corresponding normalized $2$-cocycle 
\begin{equation*}
z \in Z^{2}(\Hom_{F}(u,Z),\Hom_{F}(t, \mathbf{T})).
\end{equation*}
By Lemma \ref{lem:inertialhom} we get a section $\Hom_{X_{S}}(u_{S},Z_{G,S}) \to \Hom_{X_{S}}(\tilde{t}_{S},\mathbf{T}_{S})$ for all $S \in \Perfd$ as well, which we also denote by $s$.

As a notational convention, we set $\bbD_{y} =: \bbD_{C}$, and for a fixed $S = \Spa(R,R^{+}) \to \Spa(C)$ in $\Perfd_{C}$ with corresponding morphism $\Spa(R^{\sharp},R^{\sharp,+}) \xrightarrow{x} \Spa(C^{\sharp},C^{\sharp,+})$, we set $\bbD_{S} = \bbD_{y \circ x}$ and $\bbD_{S}^{\circ} = \bbD_{y \circ x}^{\circ}$. We also set $X_{S}^{\circ}:= X_{S} \setminus \Gamma_{y \circ x}$ and (recall the $t_{S}$-torsor $\mc{T}_{S}$ over $X_{S}$ from Definition \ref{defi:Tdef}) $\mc{T}_{S}^{\tn{loc}}$, $\mc{T}_{S}^{\tn{loc},\circ}$, or $\mc{T}_{S}^{\circ}$  the pullback of $\mc{T}_{S}$ to $\bbD_{S}$, $\bbD_{S} ^{\circ}$, or $X_{S}^{\circ}$, respectively.

Note that, for any affinoid $\Spa(R,R^{+}) = S$ in $\Perfd_{C}$, there is a map 
\begin{equation}\label{eq:gerbesm}
[\mc{T}^{\tn{loc}}_{S}/\tilde{t}_{S}] \to [\bbD_{S}/\tilde{t}_{S}]
\end{equation}
induced by the structure map $\mc{T}_{S}^{\tn{loc}} \to \bbD_{S}$, similarly for $[\mc{T}_{S}/\tilde{t}_{S}] \to [X_{S}/\tilde{t}_{S}]$.

\begin{lem}\label{lem:equivtors} Let $X$ be a site, $A$ an abelian sheaf of groups over $X$, and $H$ a sheaf of groups over $X$. Defining an $H$-torsor on $[X/A]$ is equivalent to defining an $A$-equivariant $H$-torsor on $X$. 
\end{lem}

\begin{proof}
    See \cite[Remark 2.6]{Shin21}.
    \end{proof}

    For an $H$-torsor $\mc{U}$ on a site $X$ and $\Q$ a $Z \subset Z_{H}$-torsor on $X$ we denote by $\Q \cdot \mc{U}$ the $H$-torsor $\Q \times^{Z} \mc{U}$.

 \begin{lem}\label{lem:curlyT}
 Let $T$ be any $F$-rational torus. For any $S = \Spa(R,R^{+}) \in \Perfd_{C}$ and $f \in \Hom_{X_{S}}(t_{S}, T_{S})$, the pullback along \eqref{eq:gerbesm} of the $T_{S}$-torsor on $[\bbD_{S}/\tilde{t}_{S}]$ determined (in the sense of Lemma \ref{lem:equivtors}) by the trivial $T_{S}$-torsor with $\tilde{t}_{S}$-action defined by $f$ is canonically isomorphic to 
 \begin{equation*}
 f_{*}(\mc{T}_{S}^{\tn{loc}}) := \mc{T}_{S}^{\loc} \times^{t_{S},f} T_{S}.
 \end{equation*}
 The above object is viewed as a $T_{S}$-torsor on  $[\mc{T}_{S}^{\loc}/\tilde{t}_{S}]$ via the pullback along  $[\mc{T}_{S}^{\loc}/\tilde{t}_{S}] \to \bbD_{S}$. Moreover, the analogous statement holds with $\bbD_{S}$ replaced with $X_{S}$.
 \end{lem}

 \begin{proof}
     We only prove the $\bbD_{S}$-case (they are identical). By functoriality, it suffices to prove the result for $T = t$ (this is only a pro-torus, but for the purposes of this proof it does not matter) and $f = \tn{id}$. This case follows from the fact that one has a canonical identification $[\mc{T}_{S}^{\loc}/t_{S}] \xrightarrow{\sim} \bbD_{S}$ fitting into the commutative diagram
     \[
     \begin{tikzcd}
         \lb \mc{T}_{S}^{\loc}/\tilde{t}_{S} \rb \arrow{r} \arrow{d} & \lb \bbD_{S}/\tilde{t}_{S} \rb \arrow{d} \arrow{r} & \lb \bbD_{S}/t_{S} \rb \\
         \lb \mc{T}_{S}^{\loc}/t_{S} \rb \arrow{r} &\bbD_{S},\arrow{ru} & 
     \end{tikzcd}
     \]
     where the map $\bbD_{S} \to \lb \bbD_{S}/t_{S} \rb$ is the one canonically associated to $\mc{T}_{S}^{\loc}$, the map $\lb \bbD_{S}/\tilde{t}_{S} \rb \to \lb \bbD_{S}/t_{S} \rb$ is associated to the $t_{S}$-torsor on $\lb \bbD_{S}/\tilde{t}_{S} \rb$ from the statement of the lemma for $T = t$ and $f = \tn{id}$, and the map $[\mc{T}_{S}^{\loc}/\tilde{t}_{S}] \to \bbD_{S}$ induced by the diagram equals the structure map of the gerbe $[\mc{T}_{S}^{\loc}/\tilde{t}_{S}]$.
 \end{proof}

\begin{defi}\label{defi:Qdefs}
For a fixed $S \in \Perfd_{C}$ and $\lambda \in \Hom_{X_{S}}(u_{S}, Z_{G,S})$, we introduce the following objects:
\begin{enumerate}
    \item{Take $\mathfrak{Q}_{\lambda,S}$ to be the $\mathbf{T}_{S}$-torsor on $[\bbD_{S}/\tilde{t}_{S}]$ or $[X_{S}/\tilde{t}_{S}]$ given, via Lemma \ref{lem:equivtors}, by the trivial $\mathbf{T}_{S}$-torsor and the $\tilde{t}_{S}$-action determined by the homomorphism $s(\lambda) \in \Hom_{X_{S}}(\tilde{t}_{S},\mathbf{T}_{S})$. This is where we use the chosen section $s$.}
    \item{Take $\Q_{\lambda,S}'$ to be the pullback of $\mathfrak{Q}_{\lambda,S}$ to $[\mc{T}_{S}^{\loc}/\tilde{t}_{S}]$ or $[\mc{T}_{S}/\tilde{t}_{S}]$ via \eqref{eq:gerbesm} or its global analogue.}
    \item{Take $\widetilde{\Q}_{\lambda,S}$ to be the $G_{S}$-torsor $Q_{\lambda,S}' \times^{\mathbf{T}_{S}} G_{S}$ corresponding to the embedding $\mathbf{T} \to G$ coming from the choice of pinning.}
\end{enumerate}
\end{defi}

Set $\overline{\mathbf{T}}_{S}:= \mathbf{T}_{S}/Z_{S}$, and, for a $\mathbf{T}_{S}$-torsor $\cP$, denote by $\ov{\cP}$ the $\ov{\mathbf{T}}_{S}$-torsor $\mc{P} \times^{\mathbf{T}_{S}} \ov{\mathbf{T}}_{S}$, similarly with $G_{S}$ and $\ov{G}_{S}$. We alert the reader that we will still frequently use the notation $T/Z$, $G/Z$ despite introducing this notational shortcut---its primary purpose is to lighten the notational load of certain computations.

\begin{lem}\label{lem:Zred}
The trivialization $\psi$ of $\mc{T}_{C}^{\circ}$ from Lemma \ref{lem:Grt} gives a reduction of each $\mathbf{T}_{S}$-torsor $\Q_{\lambda,S}'|_{[\mc{T}_{S}^{\circ}/\tilde{t}_{S}]}$ and $\Q_{\lambda,S}'|_{[\mc{T}_{S}^{\loc,\circ}/\tilde{t}_{S}]}$ to a $Z_{S}$-torsor, which we denote by $\mc{Q}_{\lambda,S}$.
\end{lem}

\begin{proof}
We focus on the $\bbD_{S}$-case. Denoting by $\ov{s(\lambda)} \in \Hom_{X_{S}}(t_{S},\ov{\mathbf{T}}_{S})$ the composition $\tilde{t}_{S} \xrightarrow{s(\lambda)} \mathbf{T}_{S} \to \ov{\mathbf{T}}_{S}$, we observe by Lemma \ref{lem:curlyT} that there is a canonical identification
\begin{equation*}
    \ov{\Q_{\lambda,S}'}  \xrightarrow{\sim} \ov{s(\lambda)}_{*}(\mc{T}_{S}^{\loc}).
\end{equation*}
It follows that, via the pullback of our fixed $X_{C}^{\circ}$-trivialization $\psi$ to $\bbD_{S}^{\circ}$, which we abusively denote simply by $\psi$, we obtain an isomorphism of $\ov{\mathbf{T}}_{S}$-torsors over $\bbD_{S}^{\circ}$
\begin{equation*}
    \ov{\Q_{\lambda,S}'} \xrightarrow{\ov{s(\lambda)}_{*}(\psi)} \underline{\ov{\mathbf{T}}_{S}}.
\end{equation*}
We take our desired reduction $\Q_{\lambda,S}$ to be the fiber over $e \in \underline{\ov{\mathbf{T}}_{S}}(\bbD_{S}^{\circ})$ via the composition
\begin{equation*}
     \Q_{\lambda,S}' \to \ov{\Q_{\lambda,S}'} \xrightarrow{\ov{s(\lambda)}{*}(\psi)} \underline{\ov{\mathbf{T}}_{S}}.
\end{equation*}
The analogue of this construction with $\bbD_{S}$ replaced with $X_{S}$ goes through in an identical manner (since $\psi$ is defined on $X_{C}^{\circ}$).
\end{proof}

\begin{rque}
\begin{enumerate}
    \item{Our notation does not distinguish between the local and global $\Q_{\lambda,S}$-objects because the difference between them causes no risk of confusion---one checks easily that the localization of the global version is canonically isomorphic to the local version.}
    \item{Our above construction gives a map
    \begin{equation*}
        \Hom_{F}(u,Z) \to Z^{1}_{\tn{v}}([\mc{T}_{S}^{\loc,\circ}/\tilde{t}_{S}], Z_{S})
    \end{equation*}
    (one can take \'{e}tale $Z_{S}$-torsors when $F$ has mixed characteristic), which we warn the reader is not multiplicative. This is also true for $\bbD_{S}$ replaced with $X_{S}$.
    }
\end{enumerate}    
\end{rque}

We can finally define the extended Hecke stacks:
\begin{defi}[Extended local Hecke stack with $n$ legs]\label{defi:elochck}
Fix $n \in \mathbb{N}$. Define the functor $\tn{Hck}_{Z \subset G}^{e,\loc,n}$ from $\Perfd_{C}$ to groupoids by sending $S$ to the following data:
\begin{enumerate}
    
    \item{An $n+1$-tuple $(\cP_{0}, \cP_{1}, \dots, \cP_{n})$ of band-localized (as in Definition \ref{defi:bandloc}) $G_{S}$-torsors on the gerbe $[\mc{T}_{S}^{\loc}/\tilde{t}_{S}]$ such that each $\lambda_{\cP_{i}}$ factors through $Z_{S}$;}
    
    \item{A sequence $\{f_{i}\}_{i=0}^{n-1}$, where $f_{i}$ is an isomorphism of $G_{S}$-torsors
    \begin{equation*}
        \Q_{\lambda_{i+1}/\lambda_{i},S} \cdot \cP_{i}|_{[\mc{T}_{S}^{\loc,\circ}/\tilde{t}_{S}]} \xrightarrow{\sim} \cP_{i+1}|_{[\mc{T}_{S}^{\loc,\circ}/\tilde{t}_{S}]},
    \end{equation*}
    where $\lambda_{i}:= \lambda_{\cP_{i}}$.
    }

    \item{A morphism from $((\cP_{i}),(f_{i}))$ to $((\cP_{i}'),(f'_{i}))$ is a sequence $(h_{i})_{i=0}^{n}$ of isomorphisms $\cP_{i} \xrightarrow{h_{i}} \cP_{i}'$ of $G_{S}$-torsors over $[\mc{T}_{S}^{\loc}/\tilde{t}_{S}]$---note that this forces $\lambda_{i} = \lambda_{i}'$---such that $h_{i+1}^{-1} \circ f'_{i} \circ (\Q_{\lambda_{i+1}/\lambda_{i},S} \cdot h_{i}) = f_{i}$ for all $i$.}

\end{enumerate}
    
\end{defi}
The global version is completely analogous, but we write it out fully for completeness.

\begin{defi}[Extended global Hecke stack with $n$ legs]\label{defi:eglobhck}
Fix $n \in \mathbb{N}$. Define the functor $\tn{Hck}_{Z \subset G}^{e,n}$ by sending $S$ in $\Perfd_{C}$ to the following data:
\begin{enumerate}
    \item{An $n+1$-tuple $(\cP_{0}, \cP_{1}, \dots, \cP_{n})$ of $G_{S}$-torsors on the gerbe $[\mc{T}_{S}/\tilde{t}_{S}]$ such that each $\lambda_{\cP_{i}}$ factors through $Z_{S}$;}
    
    \item{A sequence $\{f_{i}\}_{i=0}^{n-1}$, where $f_{i}$ is an isomorphism of $G_{S}$-torsors
    \begin{equation*}
        \Q_{\lambda_{i+1}/\lambda_{i},S} \cdot \cP_{i}|_{[\mc{T}_{S}^{\circ}/\tilde{t}_{S}]} \xrightarrow{\sim} \cP_{i+1}|_{[\mc{T}_{S}^{\circ}/\tilde{t}_{S}]} .
    \end{equation*}
    }

 \item{Morphisms in $\tn{Hck}_{Z \subset G}^{e,n}(S)$ are defined in an identical way as for $\tn{Hck}_{Z \subset G}^{e,\loc,n}(S)$.}
\end{enumerate}
\end{defi}

The following is immediate:
\begin{lem}
For $Z \subseteq Z' \subseteq Z_{G}$ there are transition morphisms
\begin{equation}\label{eq:Ztrans}
\tn{Hck}_{Z \subset G}^{e,\loc,n} \to \tn{Hck}_{Z' \subset G}^{e,\loc,n}, \hspace{1mm} \tn{Hck}_{Z \subset G}^{e,n} \to \tn{Hck}_{Z' \subset G}^{e,n}.
\end{equation}
\end{lem}

See Remark \ref{rque:naivereform} for the relationship between the above definitions and Definition \ref{defi:hcknaive} (their ``naive'' analogues).

\begin{rque}\label{rque:non-gluing-to-div-1}
    A notable feature of these extended Hecke stacks is that, for $S = \Spa(R,R^{+}) \in \Perfd_{C}$, every coordinate in $(\Div^{1})^{n}(S)$ is a relative Cartier divisor for $S$ which factors through the map $\Spa(R^{\sharp},R^{\sharp,+}) \xrightarrow{x} \Spa(C^{\sharp},C^{\sharp,+}) \xrightarrow{y} X_{C}$. 
    In rough terms, our constructions of $\tn{Hck}_{Z \subset G}^{e,\loc,n}$ and $\tn{Hck}_{Z \subset G}^{e,n}$ do not ``globalize to all of $X_{C}$''. 
    
    This failure to globalize can be expected, since, as we will see in Section \ref{sec:extended-spectral}, if this construction did globalize (as in the case of the ``naive'' analogues in Definition \ref{defi:hcknaive}), we would obtain a $\tn{Rep}_{\Lambda}(Q^{I})$-linear functor
    \begin{equation*}
        \tn{Rep}_{\Lambda}((\widehat{G/Z} \rtimes Q)^{I}) \to \End_{\Lambda}(\mc{D}_{\tn{lis}}(\Bun_{G}^{e},\Lambda)^{\omega})^{B\Weil_{F}^{I}},
    \end{equation*}
which is unreasonable. In particular, it would imply that every (Fargues--Scholze) $L$-parameter for $\widehat{G}$ can be lifted to an $L$-parameter for $\widehat{G/Z}$.
\end{rque}

\subsection{Convolution on extended Hecke stacks}\label{sec:Hckconv}
We now define a convolution operation on these extended Hecke stacks. The local and global cases are identical, so we construct this in full detail only in the local case. This will be a morphism
\begin{equation*}
    \tn{Hck}_{Z \subset G}^{e,\loc,2} \xrightarrow{\mu} \tn{Hck}_{Z \subset G}^{e,\loc,1}.
\end{equation*}

Let $((\cP_{0},\cP_{1},\cP_{2}), (f_{0},f_{1})) \in \tn{Hck}_{Z \subset G}^{e,\loc,2}(S)$; the goal is to define an isomorphism of $G_{S}$-torsors on $[\mc{T}_{S}^{\loc,\circ}/\tilde{t}_{S}]$
\begin{equation*}
\Q_{\lambda_{2}/\lambda_{0},S} \cdot \cP_{0} \xrightarrow{g} \cP_{2}.
\end{equation*}

We first form the composition
 \begin{equation*}
 \Q_{\lambda_{2}/\lambda_{1},S}\cdot(\Q_{\lambda_{1}/\lambda_{0},S}  \cdot \cP_{0}) \xrightarrow{\Q_{\lambda_{2}/\lambda_{1},S} \cdot f_{0}} \Q_{\lambda_{2}/\lambda_{1},S} \cdot \cP_{1} \xrightarrow{f_{1}} \cP_{2},
 \end{equation*}
 and observe that 
 \begin{equation*}
 \Q_{\lambda_{2}/\lambda_{1},S}\cdot \Q_{\lambda_{1}/\lambda_{0},S} = (\Q_{\lambda_{2}/\lambda_{0},S} \cdot \Q_{\lambda_{2}/\lambda_{0},S}^{-1}) \cdot  \Q_{\lambda_{2}/\lambda_{1},S}\cdot \Q_{\lambda_{1}/\lambda_{0},S} = \Q_{\lambda_{2}/\lambda_{0},S} \cdot [\Q_{\lambda_{2}/\lambda_{0},S}^{-1} \cdot  \Q_{\lambda_{2}/\lambda_{1},S}\cdot \Q_{\lambda_{1}/\lambda_{0},S}]
 \end{equation*}

  To lighten the notational load, temporarily set $a:=z(\lambda_{2}/\lambda_{1},\lambda_{1}/\lambda_{0})$. We note that the right-hand bracketed factor is, by Lemma \ref{lem:curlyT}, identified with the $Z_{S}$-torsor $a_{*}(\mc{T}_{S}^{\loc,\circ})$ on $[\mc{T}_{S}^{\loc,\circ}/\tilde{t}_{S}]$ (viewed as the pullback of the corresponding torsor on $\bbD_{S}^{\circ}$).

    We thus have an isomorphism of $Z_{S}$-torsors
\begin{equation}\label{eq:Arnaudisom}
    \Q_{\lambda_{2}/\lambda_{1},S} \cdot \Q_{\lambda_{1}/\lambda_{0},S} = \Q_{a,S} \cdot \Q_{\lambda_{2}/\lambda_{0},S} = a_{*}(\mc{T}_{S}^{\loc,\circ}) \cdot \Q_{\lambda_{2}/\lambda_{0},S} \xrightarrow{a_{*}(\psi)} \Q_{\lambda_{2}/\lambda_{0},S},
\end{equation}
which we denote by $a_{*}(\psi)$.

We then define the desired isomorphism $g$ to be the composition (over $[\mc{T}_{S}^{\loc,\circ}/\tilde{t}_{S}]$)
 \begin{equation}\label{eq:gconv}
 \Q_{\lambda_{2}/\lambda_{0},S}  \cdot \cP_{0} \xrightarrow{a_{*}(\psi^{-1})}   \Q_{\lambda_{2}/\lambda_{1},S} \cdot (\Q_{\lambda_{1}/\lambda_{0},S} \cdot \cP_{0}) \xrightarrow{\Q_{\lambda_{2}/\lambda_{1},S} \cdot f_{0}} \Q_{\lambda_{2}/\lambda_{1},S} \cdot \cP_{1} \xrightarrow{f_{1}} \cP_{2},
\end{equation}
where we abusively use $a_{*}(\psi^{-1})$ above to denote the isomorphism obtained from the inverse of \eqref{eq:Arnaudisom} by acting on $\cP_{0}$.

 Everything done above works the same with $\bbD_{S}$ and $\bbD_{S}^{\circ}$ replaced with $X_{S}$ and $X_{S}^{\circ}$, defining a convolution morphism
 \begin{equation*}
    \tn{Hck}_{Z \subset G}^{e,2} \xrightarrow{\mu} \tn{Hck}_{Z \subset G}^{e,1}.
\end{equation*}

It follows that for any $1 \leq i \leq n-1$, there are well-defined maps
\begin{equation*}
\mu_i : \Hck_{Z \subset G}^{e,\loc, n} \to \Hck_{Z \subset G}^{e,\loc, n-1}, \hspace{1mm} \Hck_{Z \subset G}^{e, n} \to \Hck_{Z \subset G}^{e, n-1}
\end{equation*}
that compose $f_{i-1}$ and $f_{i}$ in the sense defined above. In the case $n=2$, we have $\mu_{1} = \mu$.

\begin{prop}
The above convolution map equips $\Hck_{Z \subset G}^{e,\loc,1}$ with the structure of an algebra in correspondences. More precisely, we have a well-defined functor $\Delta^{\tn{op}} \to \tn{Corr}$ that sends $[1]$ to $\Hck_{Z \subset G}^{e,\loc,1}$ and the map $[1] \to [2]$ to the above multiplication.
\end{prop}

\begin{proof}
This functor is defined by sending $[n]$ to $\Hck^{e,\loc,n}_{Z \subset G}$, and a morphism $[m] \to [n]$ to the corresponding map $\Hck_{Z \subset G}^{e,\loc,n} \to \Hck_{Z \subset G}^{e,\loc,m}$ determined by the convolution defined above. 

The fact that this assignment is unital is clear (by taking the trivial modification), and so it suffices to show that the following associativity diagram commutes
\[
\begin{tikzcd}
    \Hck^{e,\loc,3}_{Z \subset G} \arrow["\mu_{2}"]{d} \arrow["\mu_{1}"]{r} & \Hck^{e,\loc,2}_{Z \subset G} \arrow["\mu"]{d} \\
    \Hck^{e,\loc,2}_{Z \subset G} \arrow["\mu"]{r} & \Hck^{e,\loc,1}_{Z \subset G}.
\end{tikzcd}
\]

For notational brevity, we set $z(\lambda_{i_{1}}/\lambda_{i_{2}},\lambda_{i_{2}}/\lambda_{i_{3}})_{*}(\psi^{-1}) = \psi^{-1}_{i_{1},i_{2},i_{3}}$. First, we observe that the square
\[
\begin{tikzcd}
\Q_{\lambda_{3}/\lambda_{0},S} \cdot \cP_{0} \arrow["\psi^{-1}_{3,2,0}"]{r} \arrow["\psi^{-1}_{3,1,0}"]{d} & \Q_{\lambda_{3}/\lambda_{2},S} \cdot \Q_{\lambda_{2}/\lambda_{0},S} \cdot \cP_{0} \arrow["\Q_{\lambda_{3}/\lambda_{2},S} \cdot \psi^{-1}_{2,1,0}"]{d} \\
\Q_{\lambda_{3}/\lambda_{1},S} \cdot \Q_{\lambda_{1}/\lambda_{0},S} \cdot \cP_{0} \arrow["\psi^{-1}_{3,2,1}"]{r} & \Q_{\lambda_{3}/\lambda_{2},S} \cdot \Q_{\lambda_{2}/\lambda_{1},S} \cdot \Q_{\lambda_{1}/\lambda_{0},S} \cdot \cP_{0}
\end{tikzcd}
\]
commutes, by the two-cocycle condition.

The following square also commutes for elementary reasons:
\[
\begin{tikzcd}[column sep=7em]
    \Q_{\lambda_{3}/\lambda_{1},S} \cdot \Q_{\lambda_{1}/\lambda_{0},S} \cdot \cP_{0} \arrow["\Q_{\lambda_{3}/\lambda_{1},S} \cdot f_{0}"]{r} \arrow["\psi^{-1}_{3,2,1}"]{d} & \Q_{\lambda_{3}/\lambda_{1},S} \cdot \cP_{1} \arrow["\psi^{-1}_{3,2,1}"]{d}\\
    \Q_{\lambda_{3}/\lambda_{2},S} \cdot \Q_{\lambda_{2}/\lambda_{1},S} \cdot \Q_{\lambda_{1}/\lambda_{0},S} \cdot \cP_{0} \arrow["\Q_{\lambda_{3}/\lambda_{2},S} \cdot \Q_{\lambda_{2}/\lambda_{1},S} \cdot f_{0}"]{r} & \Q_{\lambda_{3}/\lambda_{2},S} \cdot \Q_{\lambda_{2}/\lambda_{1},S} \cdot \cP_{1}.
\end{tikzcd}
\]

Starting with $(\cP_{0}, \cP_{1}, \cP_{2}, \cP_{3}, (f_{0},f_{1},f_{2}))$, going right and then going down yields the modification isomorphism
\begin{equation}\label{eq:assoc1}
     f_{2} \circ \Q_{\lambda_{3}/\lambda_{2},S} \cdot [f_{1} \circ (\Q_{\lambda_{2}/\lambda_{1},S} \cdot f_{0}) \circ z(\lambda_{2}/\lambda_{1}, \lambda_{1}/\lambda_{0})_{*}(\psi^{-1})] \circ z(\lambda_{3}/\lambda_{2},\lambda_{2}/\lambda_{0})_{*}(\psi^{-1})
\end{equation}
from $\cP_{0}$ to $\cP_{3}$ (on $[\mathcal{T}_{S}^{\loc,\circ}/\tilde{t}_{S}]$).

Similarly, going the other direction in the diagram yields the modification
\begin{equation}\label{eq:assoc2}
[f_{2} \circ (\Q_{\lambda_{3}/\lambda_{2},S} \cdot f_{1}) \circ z(\lambda_{3}/\lambda_{2},\lambda_{2}/\lambda_{1})_{*}(\psi^{-1})] \circ (\Q_{\lambda_{3}/\lambda_{1},S} \cdot f_{0}) \circ z(\lambda_{3}/\lambda_{1},\lambda_{1}/\lambda_{0})_{*}(\psi^{-1}).
\end{equation}

Noting that the last two maps in both of the above modifications are equal, we are reduced to proving that the corresponding first parts of the compositions are equal (remove the ``$f_{2} \circ \Q_{\lambda_{3}/\lambda_{2},S} \cdot f_{1}$'' terms from the end of both compositions). 

We then deduce the result from the fact that the first part of the modification \eqref{eq:assoc1} is obtained from going right and then down in the first diagram, followed by taking the bottom map in the second diagram, and the first part of the modification \eqref{eq:assoc2} is obtained from going down and then right in the first diagram, followed by taking the long path (up and around) between the bottom two terms in the second diagram. 
\end{proof}

We conclude this subsection by recording some standard compatibilities. 

\begin{prop}\label{prop:hckloc}
    There is a ``global-to-local'' map
    \begin{equation*}
        \tn{Hck}_{Z \subset G}^{e,n} \to \tn{Hck}_{Z \subset G}^{e,\loc,n}
    \end{equation*}
    which for any $i$ makes the following diagram strictly commutative
    \begin{equation}\label{eq:CartConv}
    \begin{tikzcd}
        \tn{Hck}_{Z \subset G}^{e,n} \arrow["\mu_{i}"]{r} \arrow{d} & \tn{Hck}_{Z \subset G}^{e,n-1} \arrow{d} \\
        \tn{Hck}_{Z \subset G}^{e,\loc,n}  \arrow["\mu_{i}"]{r}  & \tn{Hck}_{Z \subset G}^{
    e,\loc,n-1}.
    \end{tikzcd}
\end{equation}
\end{prop}

\begin{proof}
   This is a straightforward consequence of the fact that our local trivializations $\iota_{\underline{x}}^{*}\psi$ are pulled back from the fixed global trivialization $\psi$.
\end{proof}

\begin{rque}
    We will show in Corollary \ref{corol:CartConv} that the diagram \eqref{eq:CartConv} is Cartesian.
\end{rque}

\begin{lem}\label{lem:changeinZconv}
   For $Z \subseteq Z' \subseteq Z_{G}$, convolution maps on the local and global extended Hecke stacks are compatible with the transition morphisms \eqref{eq:Ztrans} in the obvious sense.
\end{lem}

\begin{proof}
This is clear.
\end{proof}

\begin{defi}
    We set $\tn{Hck}_{G}^{e,\loc,n} = \varinjlim_{Z} \tn{Hck}_{Z \subset G}^{e,\loc,n}$ and $\tn{Hck}_{G}^{e,n} = \varinjlim_{Z} \tn{Hck}_{Z \subset G}^{e,n}$. 
\end{defi}

By the above two results, we have a localization map $\tn{Hck}_{G}^{e,n} \to \tn{Hck}_{G}^{e,\loc,n} $ and convolution maps
    \begin{equation*}
 \tn{Hck}_{G}^{e,\loc,n} \to  \tn{Hck}_{G}^{e,\loc,n-1},  \hspace{1mm} \tn{Hck}_{G}^{e,n} \to  \tn{Hck}_{G}^{e,n-1}
    \end{equation*}
    which are compatible with localization.

\subsection{Quotient maps}\label{sec:quotmap}
As we will see shortly, a crucial property of the extended Hecke stacks defined above is that they admit maps 
\begin{equation*}
\tn{Hck}^{e,\loc,n}_{Z \subset G} \xrightarrow{\pi} \tn{Hck}^{\loc,n}_{G/Z}, \hspace{1mm} \tn{Hck}^{e,n}_{Z \subset G} \xrightarrow{\pi} \tn{Hck}^{n}_{G/Z}.
\end{equation*}

We explain the construction of the local version of $\pi$; the construction of its global analogue is essentially the same. Recall that for $S \in \Perfd$ we use $\ov{G}_{S}$ to denote $G_{S}/Z_{S}$.

The following general result about descending torsors on gerbes will be used repeatedly:

\begin{lem}\label{lem:descendtors}
    Let $\mathfrak{X} \xrightarrow{f} X$ be a gerbe banded by an abelian group $A$. Let $H$ be a sheaf of groups on $X$, and suppose that $\mc{U}$ is an $f^{*}H$-torsor on $\mathfrak{X}$ such that the inertial action \eqref{eq:bandmap} is trivial. Then $f_{*}\mc{U}$ is an $H$-torsor on $X$ and the adjunction map $f^{*}f_{*}\mc{U} \to \mc{U}$ is an isomorphism of $f^{*}H$-torsors on $\mathfrak{X}$.
\end{lem}

\begin{proof}
    First, observe that we have an identification $f_{*}f^{*}H \xrightarrow{\sim} H$, and so there is an action morphism $f_{*}\mc{U} \times_{X} H \to f_{*}\mc{U}$. It suffices to show that the map $f_{*}\mc{U} \times_{X} H \to f_{*}\mc{U} \times_{X} f_{*}\mc{U}$ is an isomorphism. This may be checked after passing to a cover which trivializes $\mathfrak{X}$, so we can assume that $\mathfrak{X} = [X/A]$.

    For the map $[X/A] \to X$, if we use Lemma \ref{lem:equivtors} to describe $H$-torsors on $[X/A]$, then pushforward sends an $A$-equivariant $H$-torsor $V$ on $[X/A]$ to $V^{A}$. When the $A$-action is trivial, we get the claimed result.
\end{proof}

We will now explain how to map a point $(\cP_{i},(f_{i})) \in \tn{Hck}^{e,\loc,n}_{Z \subset G}(S)$ for $S \in \Perfd_{C}$ to a point of $\Hck_{G/Z}^{\loc,n}(S)$. We define $\pi$ on tuples of torsors by sending $(\cP_{0}, \dots, \cP_{n})$ to $(\ov{\cP_{0}}, \dots, \ov{\cP_{n}})$ (recall also that $\ov{\cP_{i}} := \cP_{i} \times^{G_{S}} \ov{G}_{S}$), which descends canonically (by pushing forward along $[\mc{T}_{S}^{\loc}/\tilde{t}_{S}] \to \bbD_{S}$, using Lemma \ref{lem:descendtors}) to a $\ov{G}_{S}$-torsor on $\bbD_{S}$, also denoted by $\ov{\cP_{i}}$.

Observe that there is a canonical isomorphism 
\begin{equation}\label{eq:varphi}
   \varphi_{\lambda_{i+1}/\lambda_{i},\cP_{i}} \colon \ov{\Q_{\lambda_{i+1}/\lambda_{i},S} \cdot \cP_{i}} \xrightarrow{\sim} \ov{\cP_{i}}
\end{equation}
because, by construction, $\Q_{\lambda_{i+1}/\lambda_{i},S}$ is a $Z_{S}$-torsor.

To define the modifications $(g_{i})$ using the modifications $(f_{i})$ from $\tn{Hck}^{e,\loc,n}_{Z \subset G}(S)$, we take the composition 
\begin{equation*}
g_{i}:= \ov{\cP_{i}} \xrightarrow{\varphi_{\lambda_{i+1}/\lambda_{i},\cP_{i}}^{-1}} \ov{\Q_{\lambda_{i+1}/\lambda_{i},S} \cdot \cP_{i}}  \xrightarrow{\bar{f}_{i}} \ov{\cP_{i+1}},
\end{equation*}
where $\bar{f}_{i} := (-\times^{G_{S}} \ov{G}_{S})(f_{i})$.

Since all of the above constructions are evidently functorial in $S \in \Perfd_{C}$ and globalize, we obtain:
\begin{prop}\label{prop:Hckquot}
    The assignment
    \begin{equation*}
        (\cP_{i},(f_{i})) \mapsto (\ov{\cP_{i}}, (g_{i}))
    \end{equation*}
    defined above gives morphisms of functors over $\Perfd_{C}$
    \begin{equation*}
        \Hck_{Z \subset G}^{e,\loc,n} \xrightarrow{\pi} \Hck_{G/Z}^{\loc,n}, \hspace{1mm} \Hck_{Z \subset G}^{e,n} \xrightarrow{\pi} \Hck_{G/Z}^{n}
    \end{equation*}
    which are compatible with the transition maps \eqref{eq:Ztrans} and the global-to-local maps from Proposition \ref{prop:hckloc}.
\end{prop}

\begin{rque}
    \begin{enumerate}
        \item We will not use the global map $\pi$, but we include it above for completeness.
        \item It is clear that the above map factors canonically as a composition
        \begin{equation*}
             \Hck_{Z \subset G}^{e,\loc,n} \to \Hck_{G/Z,\Spd(C)}^{\loc,n} := \Hck_{G/Z}^{\loc,n} \times_{\Div^{1}} \Spd(C) \to \Hck_{G/Z}^{\loc,n}, 
        \end{equation*}
        and by abuse of notation we will sometimes also use $\pi$ to denote the first map in the composition (this is harmless).
    \end{enumerate}
\end{rque}

The next result is essentially a formal consequence of our construction, but we write out the proof for completeness (the reader is advised to skip it on the first reading).

\begin{prop}\label{prop:Hckconv}
    The following diagram is strictly commutative for any $n$ and $1 \leq i \leq n-1$,
    \[
    \begin{tikzcd}
    \Hck^{e,\loc,n}_{Z \subset G} \arrow["\mu_{i}"]{r} \arrow["\pi"]{d} & \Hck^{e,\loc,n-1}_{Z \subset G}  \arrow["\pi"]{d} \\
    \Hck^{\loc,n}_{G/Z} \arrow["\mu_{i}"]{r} & \Hck^{\loc,n-1}_{G/Z}.
    \end{tikzcd}
    \]
     The same statement holds for the global analogues.
\end{prop}

\begin{proof}
    We prove the local case when $n=2$---the general case follows immediately from this one. 

   For $S \in \Perfd_{C}$, applying the convolution first on a pair  $(\cP_{0}, \cP_{1}, f_{0})$ and $(\cP_{1}, \cP_{2}, f_{1})$ and then applying $\pi$ gives $(\ov{\cP_{0}}, \ov{\cP_{2}}, g)$, where $g$ is the composition 
\begin{equation*}
\ov{\cP_{0}} \xrightarrow{\varphi_{\lambda_{2}/\lambda_{0},\cP_{0}}^{-1}} \ov{\Q_{\lambda_{2}/\lambda_{0},S} \cdot \cP_{0}} \xrightarrow{\ov{z(\lambda_{2}/\lambda_{1},\lambda_{1}/\lambda_{0})_{*}(\psi^{-1})}}  \ov{\Q_{\lambda_{2}/\lambda_{1},S}\cdot(\Q_{\lambda_{1}/\lambda_{0},S}  \cdot \cP_{0})} 
\end{equation*}
followed by 
\begin{equation*}
\ov{\Q_{\lambda_{2}/\lambda_{1},S}\cdot(\Q_{\lambda_{1}/\lambda_{0},S}  \cdot \cP_{0})}  \xrightarrow{\ov{\Q_{\lambda_{2}/\lambda_{1},S} \cdot f_{0}}} \ov{\Q_{\lambda_{2}/\lambda_{1},S} \cdot \cP_{1}} \xrightarrow{\bar{f_{1}}} \ov{\cP_{2}}.
\end{equation*}

On the other hand, applying $\pi$ and then convolving gives the composition
\begin{equation*}
\ov{\cP_{0}} \xrightarrow{\varphi_{\lambda_{1}/\lambda_{0},\cP_{0}}^{-1}} \ov{\Q_{\lambda_{1}/\lambda_{0},S}\cdot \cP_{0}}  \xrightarrow{\bar{f}_{0}} \ov{\cP_{1}} \xrightarrow{\varphi_{\lambda_{2}/\lambda_{1},\cP_{1}}^{-1}} 
\ov{\Q_{\lambda_{2}/\lambda_{1},S}\cdot \cP_{1}}  \xrightarrow{\bar{f}_{1}} \ov{\cP_{2}}.
\end{equation*}
Equality of the two compositions then reduces to proving the equality
\begin{equation*}
\ov{\Q_{\lambda_{2}/\lambda_{1},S} \cdot f_{0}} \circ  \ov{z(\lambda_{2}/\lambda_{1},\lambda_{1}/\lambda_{0})_{*}(\psi^{-1})} \circ \varphi_{\lambda_{2}/\lambda_{0},\cP_{0}}^{-1} = \varphi_{\lambda_{2}/\lambda_{1},\cP_{1}}^{-1} \circ \bar{f}_{0} \circ \varphi_{\lambda_{1}/\lambda_{0},\cP_{0}}^{-1},
\end{equation*}
and therefore further, using $\varphi_{\lambda_{2}/\lambda_{1},\cP_{1}}^{-1} \circ \bar{f}_{0} = \ov{\Q_{\lambda_{2}/\lambda_{1},S} \cdot f_{0}} \circ \varphi_{\lambda_{2}/\lambda_{1},\cP_{1}}^{-1}$, to the equality 
\begin{equation}\label{eq:varphieq1}
    \ov{z(\lambda_{2}/\lambda_{1},\lambda_{1}/\lambda_{0})_{*}(\psi^{-1})} \circ \varphi_{\lambda_{2}/\lambda_{0},\cP_{0}}^{-1} =  \varphi_{\lambda_{2}/\lambda_{1},\cP_{1}}^{-1} \circ \varphi_{\lambda_{1}/\lambda_{0},\cP_{0}}^{-1}.
\end{equation}
We note that, by definition,
\begin{equation*}
\ov{z(\lambda_{2}/\lambda_{1},\lambda_{1}/\lambda_{0})_{*}(\psi^{-1})} = \varphi^{-1}_{z(\lambda_{2}/\lambda_{1},\lambda_{1}/\lambda_{0}),\cP_{0}}
\end{equation*}
and that, as $Z_{S}$-torsors,
\begin{equation*}
    \Q_{z(\lambda_{2}/\lambda_{1},\lambda_{1}/\lambda_{0}),S} \cdot \Q_{\lambda_{2}/\lambda_{0},S}= \Q_{\lambda_{2}/\lambda_{1},S} \cdot \Q_{\lambda_{1}/\lambda_{0},S} ,
\end{equation*}
and so we deduce the equality \eqref{eq:varphieq1}.
\end{proof}

Recall that we have a locally constant morphism
\begin{equation*}
    \Hck^{e,\loc,1}_{Z \subset G} \xrightarrow{\nu^{e}} \underline{\Hom}_{X}(u,Z)^{2}
\end{equation*}
(where $\underline{\Hom}_{X}(u,Z)$ is from Definition \ref{defi:hom}) given by sending $(\cP_{0}, \cP_{1}, f)$ to $(\lambda_{0},\lambda_{1})$, which induces a decomposition (as v-stacks)
\begin{equation}\label{eq:HckeDecomp}
\Hck^{e,\loc,1}_{Z \subset G} = \bigsqcup_{(\lambda_{0},\lambda_{1}) \in \Hom_{F}(u,Z)^{2}} \Hck^{e,\loc,1,(\lambda_{0},\lambda_{1})}_{Z \subset G}.
\end{equation}

\begin{rque}\label{rque:naivereform}
It is easy to see that 
\begin{equation*}
\Hck^{e,\loc,\{0,1\},\tn{naive}}_{Z \subset G} \times_{\Div^{1}} \Spd(C) = \bigsqcup_{\lambda \in \Hom_{F}(u,Z)} \Hck_{Z \subset G}^{e,\loc,1,(\lambda,\lambda)},
\end{equation*}
similarly with the global versions.
\end{rque}

There is an analogous decomposition for $\Hck_{G/Z}^{\loc,1}$. Recall that, by \cite[Proposition 21.1.4]{SW20}, $\pi_{0}(\Hck^{\loc,1}_{G/Z})$ is canonically identified with $\pi_{1}(G/Z)$, which fits into an exact sequence
\begin{equation*}
    0 \to \pi_{1}(G) \to \pi_{1}(G/Z) \to \Hom_{\ov{F}}(\mu,Z) \to 0.
\end{equation*}
We thus obtain a locally constant map
\begin{equation*}
\Hck^{\loc,1}_{G/Z} \xrightarrow{\nu} \Hom_{\ov{F}}(\mu, Z)
\end{equation*}
by taking the composition
\begin{equation*}
\Hck^{\loc,1}_{G/Z} \to \pi_{0}(\Hck^{\loc,1}_{G/Z}) \to \Hom_{\ov{F}}(\mu,Z),
\end{equation*}
as well as a decomposition (as v-stacks)
\begin{equation*}
\Hck^{\loc,1}_{G/Z} = \bigsqcup_{\lambda \in \Hom_{\ov{F}}(\mu, Z)} \Hck^{\loc,1,\lambda}_{G/Z}.
\end{equation*}
One defines $\Hck^{\loc,2,(\nu_{0},\nu_{1})}_{G/Z}$ for $\nu_{i} \in \Hom_{\ov{F}}(\mu,Z)$ by mapping $\Hck^{\loc,2}_{G/Z}$ to $\Hck^{\loc,1}_{G/Z} \times \Hck^{\loc,1}_{G/Z}$ via $(\mc{P}_{0},\mc{P}_{1},\mc{P}_{2},(f_{i})) \mapsto (\mc{P}_{0},\mc{P}_{1},f_{0}) \times (\mc{P}_{1},\mc{P}_{2},f_{1})$ and then applying $\nu \times \nu$. 

We make the elementary but important observation that we have an isomorphism
\begin{equation}\label{eq:utomu}
    \Hom_{F}(u,Z) \xrightarrow{\sim} \Hom_{\ov{F}}(\mu,Z)
\end{equation}
given by $\lambda \mapsto \lambda \circ \widetilde{\tn{Sh}}$, where $\widetilde{\tn{Sh}}$ is the map $\mu_{\ov{F}} \to u_{\ov{F}}$ given by the morphism on character modules 
\begin{equation*}
    \mathcal{C}(\Gal_{F}, \mathbb{Q}/\Z) \to \mathbb{Q}/\Z
\end{equation*}
sending $f$ to $f(e)$.

\begin{prop}\label{prop:Quotslope}
  For $\lambda_{0}, \lambda_{1} \in \Hom_{F}(u,Z)$, the substack $\Hck_{Z \subset G}^{e,\loc,1,(\lambda_{0},\lambda_{1})}$ maps via $\pi$ into $\Hck_{G/Z}^{\loc,1,(\lambda_{0}/\lambda_{1}) \circ \widetilde{\tn{Sh}}}$.  
\end{prop}

\begin{proof}
The fact that $\nu^{e}$ is locally constant implies that it suffices to show that $\Hck_{Z \subset G}^{e,\loc,1,(\lambda_{0},\lambda_{1})}(S)$ maps into $\Hck_{G/Z}^{\loc,1,(\lambda_{0}/\lambda_{1}) \circ \widetilde{\tn{Sh}}}(S)$ for $S = \Spa(R,R^{+}) \in \Perfd_{C}$ with $R$ an algebraically closed field. By the proof of Lemma \ref{lem:cohomvanish} (with $t$ replaced by $G$), we may assume that the fixed $X \in \Hck_{Z \subset G}^{e,\loc,1,(\lambda_{0},\lambda_{1})}(S)$ is of the form $(\widetilde{\Q}_{\lambda_{0},S}, \widetilde{\Q}_{\lambda_{1},S},f)$ (see Definition \ref{defi:Qdefs} to recall the definition of $\widetilde{\Q}_{\lambda_{i},S}$). 

Moreover, for a $Z_{S}$-torsor $\mc{U}$ on  $[\mc{T}_{S}^{\loc}/\tilde{t}_{S}]$, the tuple  $(\mc{U} \cdot \widetilde{\Q}_{\lambda_{0},S}, \mc{U} \cdot \widetilde{\Q}_{\lambda_{1},S},\mc{U} \cdot f)$ also lies in $\Hck_{Z \subset G}^{e,\loc,1}(S)$ and has the same image in $\Hck_{G/Z}^{\loc,1}(S)$ as $(\widetilde{\Q}_{\lambda_{0},S}, \widetilde{\Q}_{\lambda_{1},S},f)$. This lets us replace $X$ by a $\mc{U}$-translate to further assume that $X = (\underline{G_{S}},\widetilde{\Q}_{\lambda,S},f)$ for some $\lambda \in \Hom_{F}(u,Z)$. Finally, it is straightforward to check that replacing $f$ in $X$ with $f \circ (\Q_{\lambda,S} \cdot a)$ for any automorphism of torsors $\underline{G_{S}}|_{\bbD_{S}} \xrightarrow{a} \underline{G_{S}}|_{\bbD_{S}}$ does not affect the image of $\pi(X)$ in
$\Hom_{F}(u,Z) = \pi_{1}(G/Z)/\pi_{1}(G)$. 

Recall the isomorphism 
\begin{equation*}
    \Q_{\lambda,S} \cdot \underline{G_{S}} \xrightarrow{h} \widetilde{\Q}_{\lambda,S},
\end{equation*}
where $h$ is the canonical map 
\begin{equation*}
     \Q_{\lambda,S} \times^{Z_{S}} \underline{G_{S}} \xrightarrow{h'} \mathcal{Q}'_{\lambda,S} \times^{\mathbf{T}_{S}} \underline{G_{S}} = \widetilde{\Q}_{\lambda,S}
\end{equation*}
(again, see Definition \ref{defi:Qdefs} to recall $\mathcal{Q}'_{\lambda,S}$---here $h'$ is obtained via the isomorphism from $\Q_{\lambda,S} \times^{Z_{S}} \mathbf{T}_{S}$ to $\mathcal{Q}'_{\lambda,S}$ implicit in the definition of $\Q_{\lambda,S}$). Since any given isomorphism
\begin{equation*}
    \Q_{\lambda,S} \cdot \underline{G_{S}} \xrightarrow{f} \widetilde{\Q}_{\lambda,S}
\end{equation*}
differs from the map $h$ by such an automorphism $a$, it is enough to prove the result for the object $X =(\underline{G_{S}},\widetilde{\Q}_{\lambda,S},h) \in \Hck_{Z \subset G}^{e,\loc,1,(1,\lambda)}(S)$. By construction, $X$ is the image of the element $(\underline{\mathbf{T}_{S}}, \mathcal{Q}'_{\lambda,S},h') \in \Hck_{Z \subset \mathbf{T}}^{e,\loc,1,(1,\lambda)}(S)$. Since we have a commutative diagram
\[
\begin{tikzcd}
\Hck^{e,\loc,1}_{Z \subset \mathbf{T}}(S) \arrow{r} \arrow{d} & \Hck_{\mathbf{T}/Z}^{\loc,1}(S) \arrow{r} \arrow{d} & X_{*}(\mathbf{T}/Z) \arrow{r} \arrow{d} & \Hom_{F}(u,Z) \arrow["\tn{id}"]{d} \\
\Hck_{Z \subset G}^{e,\loc,1}(S) \arrow{r} & \Hck_{G/Z}^{\loc,1}(S) \arrow{r}  &  \pi_{1}(G/Z)\arrow{r} & \Hom_{F}(u,Z),
\end{tikzcd}
\]
we can reduce further to the case where $G=\mathbf{T}$ is a torus.

By definition, the image $X'$ of $X$ in $\Hck_{\mathbf{T}/Z}^{\loc,1}(S)$ is $(\underline{\ov{\mathbf{T}}_{S}},\ov{\mathcal{Q}'_{\lambda,S}},\bar{h} \circ \varphi_{\lambda,\underline{\ov{\mathbf{T}}_{S}}}^{-1})$---after switching the two torsors, which negates the corresponding point of $\pi_{0}(\Hck_{\mathbf{T}/Z}^{\loc,1})$ by \cite[Section I.6]{Geometrization}, we can and do view this as a point in the affine Grassmannian $\tn{Gr}_{\mathbf{T}/Z}(S)$. By the way we have set things up, this resulting point in $\tn{Gr}_{\mathbf{T}/Z}(S)$ is the image of the point $(\underline{\ov{\mathbf{T}}_{S}},\mc{T}_{S}^{\loc},\psi^{-1}) \in \tn{Gr}_{t}(S)$ via the map $t \xrightarrow{\ov{s(\lambda)}} \mathbf{T}/Z$. We deduce from the commutativity of  
\[
\begin{tikzcd}
\tn{Gr}_{t}(S) \arrow{d} \arrow["\ov{s(\lambda)}"]{r} & \tn{Gr}_{\mathbf{T}/Z}(S) \arrow{d} \\
X_{*}(t) \arrow["\ov{s(\lambda)} \circ -"]{r} & X_{*}(\mathbf{T}/Z)
\end{tikzcd}
\]
(and the definition of $\psi$ from Lemma \ref{lem:Grt}) that the image of $X'$ in $X_{*}(\mathbf{T}/Z)$ is $\ov{s(\lambda)} \circ \tn{Sh}$, whose image in $\Hom_{\ov{F}}(\mu,Z)$ is exactly $\lambda \circ \widetilde{\tn{Sh}}$. Applying the inversion from switching the torsors yields $(1/\lambda) \circ \widetilde{\tn{Sh}}$, as claimed.
\end{proof}

Henceforth, we will implicitly use the isomorphism \eqref{eq:utomu} to identify $\Hom_{F}(u,Z)$ with $\Hom_{\ov{F}}(\mu,Z)$ rather than writing the map explicitly. For example, we write $\Hck_{G/Z}^{\loc,1,\lambda_{0}/\lambda_{1}}$ to denote the stack $\Hck_{G/Z}^{\loc,1,(\lambda_{0}/\lambda_{1}) \circ \widetilde{\tn{Sh}}}$.

\begin{lem}\label{lem:outer-faces}
Fix $\lambda_0,\lambda_1,\lambda_2 \in \Hom_F(u,Z)$. The diagram
\[
\begin{tikzcd}[column sep=huge, row sep=large]
\Hck^{e,\loc,1,(\lambda_0,\lambda_1)}_{Z \subset G}
  \arrow[d,"\pi"']
& \Hck^{e,\loc,2,(\lambda_0,\lambda_1,\lambda_2)}_{Z \subset G}
  \arrow[l]
  \arrow[r]
  \arrow[d,"\pi"]
& \Hck^{e,\loc,1,(\lambda_1,\lambda_2)}_{Z\subset G}
  \arrow[d,"\pi"]
\\
\Hck^{\loc,1,\lambda_0/\lambda_1}_{G/Z}
& \Hck^{\loc,2,(\lambda_0/\lambda_1,\,\lambda_1/\lambda_2)}_{G/Z}
  \arrow[l]
  \arrow[r]
& \Hck^{\loc,1,\lambda_1/\lambda_2}_{G/Z}
\end{tikzcd}
\]
is strictly commutative, where the vertical maps are those of Proposition \ref{prop:Hckquot}. In particular, $\pi$ carries $\Hck^{e,\mathrm{loc},2,(\lambda_0,\lambda_1,\lambda_2)}_{Z\subset G}$ into $\Hck^{\mathrm{loc},2,(\lambda_0/\lambda_1,\,\lambda_1/\lambda_2)}_{G/Z}$.
\end{lem}

\begin{proof}
    This is immediate.
\end{proof}

\subsection{Extended affine Grassmannians}
To understand perverse sheaves on $\Hck^{e,\loc,1}_{Z \subset G}$ we need to introduce and study a version of the affine Grassmannian which is adapted to the extended Hecke stack $\Hck^{e,\loc,1}_{Z \subset G}$.

First, denote the local Hecke correspondence by
\[
\begin{tikzcd}
    & \Hck_{Z \subset G}^{e,\loc,1} \arrow["\overleftarrow{h}"]{ld} \arrow["\overrightarrow{h}"]{rd}& \\
    B^{e}L^{+}G & & B^{e}L^{+}G,
\end{tikzcd}
\]
where $B^{e}L^{+}G$ denotes the functor on $\Perfd_{C}$ sending $S$ to the groupoid of $G_{S}$-torsors on $[\mc{T}_{S}^{\loc}/\tilde{t}_{S}]$.

\begin{defi}\label{defi:Gre} \begin{enumerate}
   \item{ Define the functor $\tn{Gr}_{Z \subset G}^{e,(\lambda_{0},\lambda_{1})}$ from $\Perfd_{C}$ to groupoids as the fiber of $\widetilde{\Q}_{\lambda_{1},-}$ under $\overrightarrow{h}|_{\Hck_{Z \subset G}^{e,\loc,1,(\lambda_{0},\lambda_{1})}}$. 
   
   Explicitly, this is the functor sending $S$ to all pairs $(\cP, f)$, where $\cP$ is a band-localized $G_{S}$-torsor on $[\mc{T}_{S}^{\loc}/\tilde{t}_{S}]$ whose inertial morphism is $\lambda_{0}$ and $f$ is an isomorphism of  $G_{S}$-torsors over $[\mc{T}_{S}^{\loc,\circ}/\tilde{t}_{S}]$
     \begin{equation*}
        \Q_{\lambda_{1}/\lambda_{0},S} \cdot \cP \xrightarrow{f} \widetilde{\Q}_{\lambda_{1},S}.
     \end{equation*}
     Morphisms between $(\cP,f)$ and $(\cP',f')$ are isomorphisms of $G_{S}$-torsors over $[\mc{T}_{S}^{\loc}/\tilde{t}_{S}]$ which are compatible with $f$ and $f'$. }

    \item{More generally, we can define $\tn{Gr}_{Z \subset G}^{e,\lambda}$ as the fiber of $\widetilde{\Q}_{\lambda,-}$ under $\overrightarrow{h}$.
         It is easy to see that we have a decomposition
     \begin{equation*}
         \tn{Gr}_{Z \subset G}^{e,\lambda} = \bigsqcup_{\lambda' \in \Hom_{F}(u,Z)} \tn{Gr}_{Z \subset G}^{e,(\lambda',\lambda)}.
     \end{equation*}}

\item{We set $\tn{Gr}_{Z \subset G}^{e} = \bigsqcup_{\lambda \in \Hom_{F}(u,Z)} \tn{Gr}_{Z \subset G}^{e,\lambda}$.}
     \end{enumerate}
     
\end{defi}

By analogy to the usual affine Grassmannian, we have:
\begin{lem}\label{lem:Grrig}
    The automorphism group of a fixed point $X \in \tn{Gr}_{Z \subset G}^{e,\lambda}(S)$ for $S \in \Perfd_{C}$ is trivial.
\end{lem}

\begin{proof}
    Let $a$ be an automorphism of the $G_{S}$-torsor $\cP$ over $[\mc{T}_{S}^{\loc}/\tilde{t}_{S}]$ such that $f \circ (\Q_{\lambda,S} \cdot a)  = f$. It follows that  $\Q_{\lambda,S} \cdot a$ is trivial as an automorphism of $\Q_{\lambda,S} \cdot \cP$ over $[\mc{T}_{S}^{\loc,\circ}/\tilde{t}_{S}]$, and is therefore trivial on $[\mc{T}_{S}^{\loc}/\tilde{t}_{S}]$.
\end{proof}

Since $\pi_{0}(\tn{Gr}_{G}) = \pi_{0}(\Hck_{G}^{\loc,1}) = \pi_{1}(G)$ for any connected reductive $G$, we obtain, as in the previous subsection, a locally constant map
\begin{equation*}
\nu \colon \tn{Gr}_{G/Z,\Spd(C)} \to \pi_{1}(G/Z) \to \Hom_{\ov{F}}(\mu,Z) = \Hom_{F}(u,Z),
\end{equation*}
and a decomposition 
\begin{equation*}
  \tn{Gr}_{G/Z,\Spd(C)}   = \bigsqcup_{\lambda \in \Hom_{F}(u,Z)} \tn{Gr}_{G/Z,\Spd(C)}^{\lambda}.
\end{equation*}

\begin{prop}\label{prop:Grquot}
    There is a functor between groupoids over $\Perfd_{C}$
    \begin{equation*}
\tn{Gr}_{Z \subset G}^{e,(\lambda_{0},\lambda_{1})} \xrightarrow{\pi^{(\lambda_{0},\lambda_{1})}} \tn{Gr}_{G/Z,\Spd(C)}^{\lambda_{0}/\lambda_{1}}
    \end{equation*}
    making the following diagram commute up to $2$-isomorphism:
\[
\begin{tikzcd}
    \Hck_{Z \subset G}^{e,\loc,1,(\lambda_{0},\lambda_{1})} \arrow["\pi"]{r} & \Hck_{G/Z}^{\loc,1,\lambda_{0}/\lambda_{1}} \\
    \tn{Gr}_{Z \subset G}^{e,(\lambda_{0},\lambda_{1})} \arrow{u} \arrow["\pi^{(\lambda_{0},\lambda_{1})}"]{r} & \tn{Gr}_{G/Z,\Spd(C)}^{\lambda_{0}/\lambda_{1}} \arrow{u},
\end{tikzcd}
    \]
    where the top map is from Proposition \ref{prop:Hckquot} (and we are also implicitly using Proposition \ref{prop:Quotslope}).
\end{prop}

\begin{proof}
Recall the fixed auxiliary $\bbD_{C}$-trivialization $\tilde{\psi}$ of $\mc{T}_{C}^{\loc}$ from Section \ref{sec:affGrasscalc}. Given an $S$-point $(\cP, f)$ on the left-hand side, we post-compose the map 
   \begin{equation*}
   \ov{\cP} \xrightarrow{\varphi_{\lambda_{1}/\lambda_{0},\cP}} \ov{\Q_{\lambda_{1}/\lambda_{0},S} \cdot \cP} \xrightarrow{\bar{f}} \ov{\widetilde{\Q}_{\lambda_{1},S}}
   \end{equation*}
   with the trivialization
   \begin{equation*}
  \ov{\widetilde{\Q}_{\lambda_{1},S}} = \ov{s(\lambda_{1})}_{*}(\mc{T}^{\loc}_{S}) \xrightarrow{\ov{s(\lambda_{1})}_{*}(\tilde{\psi})} \underline{\ov{G}_{S}}
   \end{equation*}
   to obtain the point $(\ov{\cP}, \ov{s(\lambda_{1})}_{*}(\tilde{\psi}) \circ \bar{f} \circ \varphi_{\lambda_{1}/\lambda_{0},\cP}) \in \tn{Gr}_{G/Z,\Spd(C)}(S)$, which we take to be the image of our object under $\pi^{(\lambda_{0},\lambda_{1})}$. It is straightforward to check that a different choice of $\tilde{\psi}$ as above gives an isomorphic object in $\tn{Gr}_{G/Z,\Spd(C)}(S)$ (in a way which is functorial in $S$). An isomorphism $h$ simply goes to $\bar{h}$, finishing the definition of $\pi^{(\lambda_{0},\lambda_{1})}$.
   
     The fact that the above map lands in $\tn{Gr}_{G/Z,\Spd(C)}^{\lambda_{0}/\lambda_{1}}$ and commutes up to $2$-isomorphism is immediate from the definition of the (extended and classical) affine Grassmannians and the compatibility of the two local Hecke correspondences.
\end{proof}

\begin{lem}\label{lem:Grtwist} There is an isomorphism 
\begin{equation*}
    \tn{Gr}^{e,(\lambda_{0},\lambda_{1})}_{Z \subset G} \xrightarrow{\sim} \tn{Gr}^{e,(1,\lambda_{1}/\lambda_{0})}_{Z \subset G}
\end{equation*}
over $\tn{Gr}_{G/Z,\Spd(C)}$, defined by sending $(\cP, f)$ to $(\Q_{\lambda_{0}^{-1},S} \cdot \cP, z(\lambda_{1},\lambda_{0}^{-1})_{*}(\psi^{-1}) \circ (\Q_{\lambda_{0}^{-1},S} \cdot f))$. 
\end{lem}
\begin{proof}
    This is a straightforward computation.
\end{proof}

\begin{prop}\label{prop:mainGrisom}
    The map $\pi^{(\lambda_{0},\lambda_{1})}$ is an equivalence of $\tn{v}$-stacks. 
\end{prop}

\begin{proof}
Lemma \ref{lem:Grtwist} reduces us to the case where $\lambda_{0}=1$ and we set $\lambda_{1}/\lambda_{0} = \lambda$, which simplifies computations. Next, observe that there is an isomorphism of functors (which is not $L^{+}G$-equivariant) over $\Perfd_{C}$
    \begin{equation}\label{eq:Grpfeq1}
        \tn{Gr}_{G,\Spd(C)} \xrightarrow{\sim} \tn{Gr}_{Z \subset G}^{e,(1,\lambda)}
    \end{equation}
    defined by sending $(\cP,f)$ to the torsor $\cP$ along with the isomorphism
\begin{equation*}
    \Q_{\lambda,S} \cdot \cP \xrightarrow{\Q_{\lambda,S}\cdot f} \Q_{\lambda,S} \cdot \underline{G_{S}} = \widetilde{\Q}_{\lambda,S}.
\end{equation*}

On the other hand, one has a (non-$L^{+}G$-equivariant) ``translation'' isomorphism 
 \begin{equation}\label{eq:Grpfeq2}
     \tn{Gr}_{G/Z,\Spd(C)}^{1} \xrightarrow{\sim} \tn{Gr}_{G/Z,\Spd(C)}^{1/\lambda}
 \end{equation}
 given by sending $(\mc{M},f)$ to $(\mc{M}, \ov{s(\lambda)}_{*}(\tilde{\psi}) \circ \ov{s(\lambda)}_{*}(\psi^{-1}) \circ f)$, cf. the proof of Proposition \ref{prop:Quotslope} (and $\tilde{\psi}$ is the auxiliary $\bbD_{C}$-trivialization from the end of Section \ref{sec:affGrasscalc}). The map $\tn{Gr}_{G,\Spd(C)} \to \tn{Gr}_{G/Z,\Spd(C)}$ factors as the composition
 \begin{equation*}
     \tn{Gr}_{G,\Spd(C)} \to \tn{Gr}_{G/Z,\Spd(C)}^{1} \hookrightarrow \tn{Gr}_{G/Z,\Spd(C)}.
 \end{equation*}

 Now, observing that the following diagram commutes
 \[
\begin{tikzcd}
\tn{Gr}_{G,\Spd(C)} \arrow["\eqref{eq:Grpfeq1}"]{r} \arrow["\tn{id}"]{d} & \tn{Gr}_{Z \subset G}^{e,(1,\lambda)} \arrow["\pi^{(1,\lambda)}"]{r} & \tn{Gr}_{G/Z,\Spd(C)}^{1/\lambda} \arrow["\tn{id}"]{d} \\
\tn{Gr}_{G,\Spd(C)} \arrow{r} & \tn{Gr}_{G/Z,\Spd(C)}^{1} \arrow["\eqref{eq:Grpfeq2}"]{r} & \tn{Gr}_{G/Z,\Spd(C)}^{1/\lambda},
\end{tikzcd}
 \]
 we deduce the result from the fact that the bottom line is an equivalence.
\end{proof}

\begin{corol}\label{rque:sheafequiv}
   Pullback along $\pi^{(\lambda_{0},\lambda_{1})}$ induces an equivalence between the categories of sheaves of abelian groups on $\tn{Gr}_{Z \subset G}^{e,(\lambda_{0},\lambda_{1})}$ and $\tn{Gr}_{G/Z,\Spd(C)}^{\lambda_{0}/\lambda_{1}}$.  
\end{corol}

\begin{proof}
     This comes from combining Proposition \ref{prop:mainGrisom} with Lemma \ref{lem:Grrig}.
\end{proof}

\begin{defi} We denote by $\mc{A}_{\lambda_{1}}$ the functor on $\Perfd_{C}$ sending $S$ to $\underline{\tn{Aut}}_{[\mc{T}_{S}^{\loc}/\tilde{t}_{S}]}(\widetilde{\Q}_{\lambda_{1},S})$. 
\end{defi}

\begin{rque}\label{rque:Alambda} \begin{enumerate}
    \item Using the homomorphism $\mc{A}_{\lambda_{1}} \to L^{+}(G/Z)$ induced by $\tilde{\psi}$, the map $\pi^{(\lambda_{0},\lambda_{1})}$ is $\mc{A}_{\lambda_{1}}$-equivariant. 

    \item Note that $\tilde{\psi}$ induces an isomorphism of gerbes $[\mc{T}_{C}^{\loc}/\tilde{t}_{C}] \xrightarrow{\sim} [\bbD_{C}/u_{C}]$ over $\bbD_{C}$ under which $\widetilde{\Q}_{\lambda_{1},C}$ corresponds to the trivial $G_{C}$-torsor with $u_{C}$-action given by $\lambda_{1}$ (using Lemma \ref{lem:equivtors}). Since $\lambda_{1}$ is valued in $Z_{G}$, every $\bbD_{C}$-automorphism of the trivial $G_{C}$-torsor is $u_{C}$-equivariant, giving an isomorphism $\mc{A}_{\lambda_{1}} \xrightarrow{\sim} L^{+}G$ which is compatible with the maps from both groups to $L^{+}(G/Z)$.

\end{enumerate}
\end{rque}

\begin{rque}\label{rque:strat}
    Proposition \ref{prop:mainGrisom} shows that, via the isomorphism $\pi^{(\lambda_{0},\lambda_{1})}$, the locally spatial diamond $\tn{Gr}_{Z \subset G}^{e,(\lambda_{0},\lambda_{1})}$ inherits a natural stratification from the Schubert stratification of $\tn{Gr}_{G/Z,\Spd(C)}^{\lambda_{0}/\lambda_{1}}$ which is indexed by all $\nu \in X_{*}(T/Z)^{+,\lambda_{0}/\lambda_{1}}$, where the superscript ``$\lambda_{0}/\lambda_{1}$'' means that the image of $\nu$ in $\Hom_{\ov{F}}(\mu, Z)$ is $\lambda_{0}/\lambda_{1}$. The same then holds for $\Hck_{Z \subset G}^{e,\loc,1}$. In particular, we can talk about the Schubert cells $\tn{Gr}_{Z \subset G,\nu}^{e,(\lambda_{0},\lambda_{1})}$ and $\Hck_{Z \subset G,\nu}^{e,\loc,1}$ for $\nu \in X_{*}(T/Z)^{+,\lambda_{0}/\lambda_{1}}$.
\end{rque}

By construction, there is an identification
\begin{equation*}
 \Hck_{Z \subset G}^{e,\loc,1,(\lambda_{0},\lambda_{1})} =   \mc{A}_{\lambda_{1}} \backslash \tn{Gr}_{Z \subset G}^{e,(\lambda_{0},\lambda_{1})}.
\end{equation*}

Finally, all of the constructions in this section are functorial in $Z \subseteq Z' \subseteq Z_{G}$, in the sense that there are canonical ``transition morphisms''
\begin{equation*}
\tn{Gr}_{Z \subset G}^{e} \to \tn{Gr}_{Z' \subset G}^{e}
\end{equation*}
which are compatible with the map from Proposition \ref{prop:Grquot} and the above identifications with $\Hck_{Z \subset G}^{e,\loc,1}$ and  $\Hck_{Z' \subset G}^{e,\loc,1}$.

\subsection{Geometric Satake}
 We assume the coefficient ring $\Lambda$  is killed by $n$ with $(n,p)=1$.
 
 \begin{defi}\label{defi:convGre} Given a triple $\lambda_{0},\lambda_{1}, \lambda_{2} \in \Hom_{F}(u,Z)$, define the \textbf{extended convolution affine Grassmannian} $\tn{Gr}_{Z \subset G}^{e,(\lambda_{0},\lambda_{1},\lambda_{2})}$ to be the fiber over $\widetilde{\Q}_{\lambda_{2},-}$ under $\overrightarrow{h}|_{\Hck_{Z \subset G}^{e,\loc,2,(\lambda_{0},\lambda_{1},\lambda_{2})}}$.
 \end{defi}

The convolution morphism from Section \ref{sec:Hckconv} induces a convolution morphism
\begin{equation*}
    \tn{Gr}_{Z \subset G}^{e,(\lambda_{0},\lambda_{1},\lambda_{2})} \xrightarrow{\mu} \tn{Gr}_{Z \subset G}^{e,(\lambda_{0},\lambda_{2})}.
\end{equation*}

Recall the classical convolution affine Grassmannian $\tn{Gr}_{G/Z,\Spd(C)} \tilde{\times} \tn{Gr}_{G/Z,\Spd(C)}$ for $G/Z$, which is equipped with a canonical isomorphism
\begin{equation}\label{eq:convGriso}
    \tn{Gr}_{G/Z,\Spd(C)} \tilde{\times} \tn{Gr}_{G/Z,\Spd(C)} \xrightarrow{\sim}  \tn{Gr}_{G/Z,\Spd(C)} \times \tn{Gr}_{G/Z,\Spd(C)}
\end{equation}
sending $(\cP_{0}, \cP_{1}, (f_{0},f_{1}))$ to $((\cP_{0}, f_{1} \circ f_{0}),(\cP_{1},f_{1})).$

There is a locally constant map
\begin{equation*}
    \tn{Gr}_{G/Z,\Spd(C)} \tilde{\times} \tn{Gr}_{G/Z,\Spd(C)} \to \Hck_{G/Z,\Spd(C)}^{\loc,2} \to \underline{\Hom}_{X}(u,Z)^{2},
 \end{equation*}
and for $\lambda, \lambda' \in \Hom_{F}(u,Z)$ we define $\tn{Gr}_{G/Z,\Spd(C)}^{(\lambda,\lambda')}$ as the preimage of $(\lambda,\lambda')$ under this map. Using this convention, we observe that convolution gives a map
\begin{equation*}
\tn{Gr}_{G/Z,\Spd(C)}^{(\lambda,\lambda')} \to \tn{Gr}_{G/Z,\Spd(C)}^{\lambda \lambda'}.
\end{equation*}

\begin{rque} The above displayed map is the same as applying the identification 
\begin{equation*}\tn{Gr}_{G/Z,\Spd(C)} \tilde{\times} \tn{Gr}_{G/Z,\Spd(C)} \xrightarrow{\eqref{eq:convGriso}} \tn{Gr}_{G/Z,\Spd(C)} \times \tn{Gr}_{G/Z,\Spd(C)}
\end{equation*}
and then, denoting the map $\tn{Gr}_{G/Z,\Spd(C)} \to \underline{\Hom}_{X}(u,Z)$ by $q$, taking the map $(q_{1}/q_{2}, q_{2})$ (where $q_{i}$ is $q$ composed with the $i$-projection).
\end{rque}

\begin{lem}\label{lem:Grcompconv} For fixed $\lambda_{0}, \lambda_{1},\lambda_{2} \in \Hom_{F}(u,Z)$, there is an $\mc{A}_{\lambda_{2}}$-equivariant equivalence 
\begin{equation*}
 \pi^{(\lambda_{0},\lambda_{1},\lambda_{2})}\colon \tn{Gr}_{Z \subset G}^{e,(\lambda_{0},\lambda_{1},\lambda_{2})} \to \tn{Gr}_{G/Z,\Spd(C)}^{(\lambda_{0}/\lambda_{1},\lambda_{1}/\lambda_{2})}
\end{equation*}
making the following diagram commute
    \begin{equation}\label{eq:Grconvdiag}
    \begin{tikzcd}
        \tn{Gr}_{Z \subset G}^{e,(\lambda_{0},\lambda_{1},\lambda_{2})} \arrow{r} \arrow["\pi^{(\lambda_{0},\lambda_{1},\lambda_{2})}"]{d} &   \tn{Gr}_{Z \subset G}^{e,(\lambda_{0},\lambda_{2})} \arrow{d} \\
       \tn{Gr}_{G/Z,\Spd(C)}^{(\lambda_{0}/\lambda_{1},\lambda_{1}/\lambda_{2})} \arrow{r} & \tn{Gr}_{G/Z,\Spd(C)}^{\lambda_{0}/
\lambda_{2}},
    \end{tikzcd}
    \end{equation}
    where the horizontal maps are convolution and the right-hand vertical map is from Proposition \ref{prop:mainGrisom}.
    \end{lem}

    \begin{proof}
We first define $\pi^{(\lambda_{0},\lambda_{1},\lambda_{2})}$ and prove it is an equivalence. It is given by sending the object $(\cP_{0},\cP_{1}, (f_{0}, f_{1})) \in \tn{Gr}_{Z \subset G}^{e,(\lambda_{0},\lambda_{1},\lambda_{2})} (S)$ to 
\begin{equation*}
(\ov{\cP_{0}},\ov{\cP_{1}},\bar{f}_{0} \circ \varphi_{\lambda_{1}/\lambda_{0},\cP_{0}},\ov{s(\lambda_{2})}_{*}(\tilde{\psi})\circ (\bar{f}_{1} \circ \varphi_{\lambda_{2}/\lambda_{1},\cP_{1}}))
\end{equation*}
where again we are using a fixed auxiliary $\bbD_{C}$-trivialization $\tilde{\psi}$ of $\mc{T}_{C}^{\loc}$. As is the case (the isomorphism \eqref{eq:convGriso}) for the classical convolution affine Grassmannian, combining convolution with the projection gives a canonical isomorphism
   \begin{equation}\label{eq:econvGriso}
        \tn{Gr}_{Z \subset G}^{e,(\lambda_{0},\lambda_{1},\lambda_{2})} \xrightarrow{\sim}  \tn{Gr}_{Z \subset G}^{e,(\lambda_{0},\lambda_{2})} \times \tn{Gr}_{Z \subset G}^{e,(\lambda_{1},\lambda_{2})};
   \end{equation}
   the proof of this isomorphism claim is straightforward and left to the reader.

   The fact that $\pi^{(\lambda_{0},\lambda_{1},\lambda_{2})}$ is an equivalence follows from the commutativity of the following diagram:
   \[
\begin{tikzcd}
     \tn{Gr}_{Z \subset G}^{e,(\lambda_{0},\lambda_{1},\lambda_{2})} \arrow["\eqref{eq:econvGriso}"]{r} \arrow["\pi^{(\lambda_{0},\lambda_{1},\lambda_{2})}"]{d} & \tn{Gr}_{Z \subset G}^{e,(\lambda_{0},\lambda_{2})} \times \tn{Gr}_{Z \subset G}^{e,(\lambda_{1},\lambda_{2})} \arrow{d} \\     \tn{Gr}_{G/Z,\Spd(C)}^{(\lambda_{0}/\lambda_{1},\lambda_{1}/\lambda_{2})} \arrow["\eqref{eq:convGriso}"]{r} & \tn{Gr}_{G/Z,\Spd(C)}^{\lambda_{0}/\lambda_{2}} \times \tn{Gr}_{G/Z,\Spd(C)}^{\lambda_{1}/\lambda_{2}},
\end{tikzcd}
   \]
   where all other maps are isomorphisms by the above arguments and Proposition \ref{prop:mainGrisom}.
   
   It remains to show that the diagram \eqref{eq:Grconvdiag} commutes. Going right and then down gives
   \begin{equation*}
       (\ov{\cP_{0}}, \ov{s(\lambda_{2})}_{*}(\tilde{\psi}) \circ \ov{z(\lambda_{1}/\lambda_{0},\lambda_{2}/\lambda_{1})_{*}(\psi^{-1}) \circ f_{1} \circ (\Q_{\lambda_{2}/\lambda_{1},S} \cdot f_{0})} \circ \varphi_{\lambda_{2}/\lambda_{0},\cP_{0}}).
   \end{equation*}
   Going the other direction yields
   \begin{equation*}
        (\ov{\cP_{0}},  \ov{s(\lambda_{2})}_{*}(\tilde{\psi}) \circ \bar{f}_{1} \circ \varphi_{\lambda_{2}/\lambda_{1},\cP_{1}} \circ \bar{f}_{0} \circ \varphi_{\lambda_{1}/\lambda_{0},\cP_{0}}),
   \end{equation*}
    and so again we are done by the definition of the cocycle $z$ in the same way as in the proof of Proposition \ref{prop:Hckconv} (the Hecke stack analogue of this statement).        
    \end{proof}

    \begin{rque}
        One can define $\tn{Gr}_{Z \subset G}^{e,(\lambda_{0}, \dots, \lambda_{n})}$ in a manner completely analogous to Definition \ref{defi:convGre}. The same proof as that of Lemma \ref{lem:Grcompconv} shows that there is an equivalence 
        \begin{equation*}
        \tn{Gr}_{Z \subset G}^{e,(\lambda_{0}, \dots, \lambda_{n})} \to \tn{Gr}_{G/Z,\Spd(C)}^{(\lambda_{0}/\lambda_{1}, \dots, \lambda_{n-1}/\lambda_{n})}
       \end{equation*}
        for any $n \geq 1$. Moreover, we have an identification 
        \begin{equation*}
      \mc{A}_{\lambda_{n}} \backslash \tn{Gr}_{Z \subset G}^{e,(\lambda_{0}, \dots, \lambda_{n})}  = \Hck_{Z \subset G}^{e,\loc,n,(\lambda_{0}, \dots, \lambda_{n})} 
        \end{equation*}
        which is compatible with any of the convolution maps for $1 \leq i \leq n-1$.
    \end{rque}

 For any small v-stack $S \to \Spd(C)$ we set $\Hck_{Z \subset G,S}^{e,\loc,1} := \Hck_{Z \subset G}^{e,\loc,1} \times_{\Spd(C)} S$, similarly with $\Hck_{Z \subset G,S}^{e,\loc,1,(\lambda_{0},\lambda_{1})}$ and $\tn{Gr}_{Z \subset G,S}^{e}$. One can define perverse sheaves on $\Hck_{Z \subset G,S}^{e,\loc,1}$ in a manner which is a verbatim analogue of \cite[Definition/Proposition VI.7.1]{Geometrization}, whose proof is identical because of the isomorphism of affine Grassmannians provided by Proposition \ref{prop:mainGrisom}:

 \begin{prop}
     Let $S \to \Spd(C)$ be a small v-stack. For any $\lambda_{0},\lambda_{1} \in \Hom_{F}(u,Z)$, there is a unique $t$-structure $({^{p}}D^{\leq0},{^{p}D}^{\geq 0})$ on $D_{\tn{\'{e}t}}(\Hck_{Z \subset G,S}^{e,\loc,1,(\lambda_{0},\lambda_{1})},\Lambda)$ such that 
     \begin{equation*}
     A \in {^{p}D}^{\leq 0}_{\tn{\'{e}t}}(\Hck_{Z \subset G,S}^{e,\loc,1,(\lambda_{0},\lambda_{1})},\Lambda)
     \end{equation*}
     if and only if for all geometric points $\bar{s}$ of $S$ and all $\nu \in X_{*}(T/Z)^{+,\lambda_{0}/\lambda_{1}}$ the pullback of $A$ to each open Schubert cell $\Hck_{Z \subset G,\bar{s},\nu}^{e,\loc,1,(\lambda_{0},\lambda_{1})}$ of $\Hck_{Z \subset G,\bar{s}}^{e,\loc,1,(\lambda_{0},\lambda_{1})}$ (as in Remark \ref{rque:strat}) sits in cohomological degrees $\leq - \langle 2\rho, \nu \rangle$, where $\rho$ is the half-sum of the positive roots of $G$.
 \end{prop}

 The above proposition, along with the decomposition 
 \begin{equation*}
 \Hck_{Z \subset G,S}^{e,\loc,1} = \bigsqcup_{(\lambda_{0},\lambda_{1}) \in \Hom_{F}(u,Z)^{2}}\Hck_{Z \subset G,S}^{e,\loc,1,(\lambda_{0},\lambda_{1})}
 \end{equation*}
 lets us define (relative) perverse sheaves on  $\Hck_{Z \subset G,S}^{e,\loc,1}$ for $S$ as above.

     \begin{corol}\label{prop:perveequiv1}
For a fixed $\lambda_{0}, \lambda_{1} \in \Hom_{F}(u,Z)$, the pullback along
\begin{equation*}
\Hck^{e,\loc,1,(\lambda_{0},\lambda_{1})}_{Z \subset G} \xrightarrow{\pi} \Hck^{\loc,1,\lambda_{0}/
\lambda_{1}}_{G/Z,\Spd(C)}
\end{equation*}
induces an equivalence of categories
\begin{equation}\label{eq:pervequivbis}
    \tn{Perv}(\Hck^{\loc,1,\lambda_{0}/
\lambda_{1}}_{G/Z,\Spd(C)},\Lambda) \to \tn{Perv}(\Hck^{e,\loc,1,(\lambda_{0},\lambda_{1})}_{Z \subset G},\Lambda).
\end{equation}
\end{corol}

\begin{proof}
   Proposition \ref{prop:mainGrisom} implies (cf. Corollary \ref{rque:sheafequiv}) that pullback along $\pi$ gives an equivalence 
    \begin{equation*}
    \tn{Perv}(\Hck^{\loc,1,\lambda_{0}/\lambda_{1}}_{G/Z,\Spd(C)},\Lambda) \xrightarrow{\sim} \tn{Perv}(\mc{A}_{\lambda_{1}} \backslash \tn{Gr}_{G/Z,\Spd(C)}^{\lambda_{0}/\lambda_{1}},\Lambda).
\end{equation*}

We claim that the pullback map
\begin{equation*}
  \tn{Perv}(L^{+}(G/Z) \backslash \tn{Gr}_{G/Z,\Spd(C)}^{\lambda_{0}/\lambda_{1}},\Lambda)  \to \tn{Perv}(\mc{A}_{\lambda_{1}} \backslash \tn{Gr}_{G/Z,\Spd(C)}^{\lambda_{0}/\lambda_{1}},\Lambda)
\end{equation*}
is also an equivalence; it suffices by Remark \ref{rque:Alambda} to prove this after identifying $\mc{A}_{\lambda_{1}}$ with $L^{+}G$. From here, we deduce the result from Lemma \ref{lem:Grgerbe}.
\end{proof}

\begin{defi}\label{defi:satake-category}
   We define the Satake category
    \begin{equation*}
    \tn{Sat}(\Hck_{Z \subset G}^{e,\loc,1,(\lambda_{0},\lambda_{1})},\Lambda) \subset D_{\tn{\'{e}t}}(\tn{Gr}_{Z \subset G}^{e,(\lambda_{0},\lambda_{1})},\Lambda)
    \end{equation*}
    to be the preimage of $\tn{Sat}(\Hck_{G/Z,\Spd(C)}^{\loc,1,\lambda_{0}/\lambda_{1}},\Lambda)$ (in the sense of \cite[Definition VI.7.8]{Geometrization})
    under pullback by the isomorphism $\pi^{(\lambda_{0},\lambda_{1})}$. 
    
    \begin{rque} It is straightforward to check (using Proposition \ref{prop:mainGrisom}) that this is the same as the full subcategory of all objects in $\mc{D}_{\tn{\'{e}t}}(\Hck_{Z \subset G}^{e,\loc,1,(\lambda_{0},\lambda_{1})},\Lambda)^{\tn{bd}}$ that are universally locally acyclic and flat perverse. 
    \end{rque}
    
    Similarly, we define
    \begin{equation*}
\tn{Sat}(\Hck_{Z \subset G}^{e,\loc,1},\Lambda) = \prod_{(\lambda_{0},\lambda_{1}) \in \Hom_{F}(u,Z)^{2}} \tn{Sat}(\Hck_{Z \subset G}^{e,\loc,1,(\lambda_{0},\lambda_{1})},\Lambda)  \subset \mc{D}_{\tn{\'{e}t}}(\Hck_{Z \subset G}^{e,\loc,1},\Lambda)^{\tn{bd}},
    \end{equation*}
    which again we can characterize as all bounded objects that are universally locally acyclic and flat perverse. 
\end{defi}

 \begin{prop}\label{prop:pervconv} Fix $\lambda_{0}, \lambda_{1},\lambda_{2} \in \Hom_{F}(u,Z)$.
     \begin{enumerate}
         \item{If $A \in {^{p}D}^{\leq 0}_{\tn{\'{e}t}}(\Hck_{Z \subset G}^{e,\loc,1,(\lambda_{0},\lambda_{1})},\Lambda)^{\tn{bd}}$ and $A' \in {^{p}D}^{\leq 0}_{\tn{\'{e}t}}(\Hck_{Z \subset G}^{e,\loc,1,(\lambda_{1},\lambda_{2})},\Lambda)^{\tn{bd}}$ then $A \star A' \in {^{p}D}^{\leq 0}_{\tn{\'{e}t}}(\Hck_{Z \subset G}^{e,\loc,1,(\lambda_{0},\lambda_{2})},\Lambda)^{\tn{bd}}$.}
\item{If $A,A' \in \tn{Sat}(\Hck_{Z \subset G}^{e,\loc,1},\Lambda)$ then $A \star A' \in \tn{Sat}(\Hck_{Z \subset G}^{e,\loc,1},\Lambda)$.}
         \item{There is a commutative convolution diagram
    \[
    \begin{tikzcd}
       \tn{Sat}(\Hck^{e,\loc,1,(\lambda_{0},\lambda_{1})}_{Z \subset G},\Lambda)\times \tn{Sat}( \Hck^{e,\loc,1,(\lambda_{1},\lambda_{2})}_{Z \subset G},\Lambda) \arrow{r} \arrow["\sim"]{d} &  \tn{Sat}(\tn{Hck}_{Z \subset G}^{e,\loc,1,(\lambda_{0},\lambda_{2})},\Lambda) \arrow["\sim"]{d} \\
        \tn{Sat}(\Hck^{\loc,1,\lambda_{0}/
\lambda_{1}}_{G/Z,\Spd(C)},\Lambda) \times \tn{Sat}(\Hck^{\loc,1,\lambda_{1}/
\lambda_{2}}_{G/Z,\Spd(C)},\Lambda) \arrow{r}  & \tn{Sat}(\Hck^{\loc,1,\lambda_{0}/\lambda_{2}
}_{G/Z,\Spd(C)},\Lambda).
\end{tikzcd}
\]}

     \end{enumerate}
 \end{prop}

 \begin{proof}
     Combining Lemma \ref{lem:outer-faces} with Corollary \ref{prop:perveequiv1} and then using the commutative diagram
     \[
     \begin{tikzcd}
     \Hck_{Z \subset G}^{e,\loc,2,(\lambda_{0},\lambda_{1},\lambda_{2})} \arrow{d} \arrow{r} & \Hck_{Z \subset G}^{e,\loc,1,(\lambda_{0},\lambda_{2})} \arrow{d} \\
\Hck_{G/Z,\Spd(C)}^{\loc,2,(\lambda_{0}/\lambda_{1},\lambda_{1}/\lambda_{2})} \arrow{r} & \Hck_{G/Z,\Spd(C)}^{\loc,1,\lambda_{0}/\lambda_{2}},
     \end{tikzcd}
     \]
     which is Cartesian by Proposition \ref{prop:mainGrisom} and Lemma \ref{lem:Grcompconv}, implies that convolution of perverse sheaves on $\Hck_{Z \subset G}^{e,1,\loc}$ can be computed using the analogous convolution on $\Hck_{G/Z,\Spd(C)}^{\loc,1}$, giving statement (3).  All of the remaining claims therefore follow from \cite[Proposition VI.8.1]{Geometrization}.
 \end{proof}

Passing to the dual side, denote by $\widehat{Z}$ the $\Lambda$-group scheme Cartier dual to $\Hom_{\Z}(X^{*}(Z), \mathbb{Q}/\Z)$, which fits into the central extension
\begin{equation*}
    1 \to \widehat{Z} \to \widehat{G/Z} \to \widehat{G} \to 1.
\end{equation*}
 Note that there is an identification
\begin{equation*}
    X^{*}(\widehat{Z}) \xrightarrow{\sim} \Hom_{\ov{F}}(\mu,Z) \xrightarrow{\sim} \Hom_{F}(u,Z),
\end{equation*}
which we will use without comment. 

By taking the weight decomposition of $V \in \Rep_{\Lambda}(\widehat{G/Z})$ and restricting to $\widehat{Z}$ (viewed as a $\Lambda$-group scheme), this category has a natural $X^{*}(\widehat{Z})$-grading.

Consider the category 
\begin{equation*}
\bigoplus_{(\lambda_{0}, \lambda_{1}) \in \Hom_{F}(u, Z)^{2}} \tn{Rep}_{\Lambda}(\widehat{G/Z})^{\lambda_{0}/\lambda_{1}},
\end{equation*}
which carries an additive, associative, non-symmetric ``convolution'' monoidal structure determined by the formula (defined on individual summands)
\begin{equation*}
v_{(\lambda_{i},\lambda_{j})} \star w_{(\lambda_{k}, \lambda_{l})} = (v \otimes w)_{(\lambda_{i},\lambda_{l})}
\end{equation*}
if $\lambda_{j} =\lambda_{k}$ and zero otherwise.

\begin{rque}\label{rque:altmonoidal}
One can identify the monoidal category $\bigoplus_{(\lambda_{0}, \lambda_{1}) \in \Hom_{F}(u, Z)^{2}} \tn{Rep}_{\Lambda}(\widehat{G/Z})^{\lambda_{0}/\lambda_{1}}$ (with $\star$) with the category 
\begin{equation*}
\Coh^{\heartsuit, \Lambda-\tn{flat}}(\pt/\widehat{G/Z} \times_{\pt/\widehat{G}} \pt/\widehat{G/Z})
\end{equation*}
equipped with the convolution monoidal structure. This gives an alternative interpretation of the monoidal category appearing above.
\end{rque}

We obtain an analogue of the geometric Satake isomorphism:

\begin{thm}\label{thm:satake}
    We have an equivalence of monoidal categories
    \begin{equation}\label{eq:pervequiv}
        \tn{Sat}(\Hck^{e,\loc,1}_{Z \subset G},\Lambda) \xrightarrow{\sim} (\bigoplus_{(\lambda_{0}, \lambda_{1}) \in \Hom_{F}(u, Z)^{2}} \tn{Rep}_{\Lambda}(\widehat{G/Z})^{\lambda_{0}/\lambda_{1}},\star).
    \end{equation}
\end{thm}

\begin{proof}
    Using the identification of monoidal categories
    \begin{equation*}
\tn{Sat}(\Hck^{\loc,1}_{G/Z,\Spd(C)},\Lambda) \xrightarrow{\sim} \tn{Rep}_{\Lambda}(\widehat{G/Z}),
    \end{equation*}
    from \cite[Theorem I.6.3]{Geometrization} we observe that the equivalence \eqref{eq:pervequivbis} from Corollary \ref{prop:perveequiv1} is identified with an equivalence of categories
     \begin{equation*}
\tn{Sat}(\Hck^{e,\loc,1,(\lambda_{0},
\lambda_{1})}_{Z \subset G},\Lambda) \xrightarrow{\sim} \tn{Rep}_{\Lambda}(\widehat{G/Z})^{\lambda_{0}/\lambda_{1}}.
    \end{equation*}
 
    This map is given explicitly by sending a representation $V \in \tn{Rep}_{\Lambda}(\widehat{G/Z})^{\lambda_{0}/\lambda_{1}}$ to the pullback of the Satake sheaf $\mathcal{S}_{V} \in \tn{Sat}(\Hck_{G/Z}^{\loc,1,\lambda_{0}/\lambda_{1}},\Lambda)$ to $\tn{Sat}(\Hck_{Z \subset G}^{e,\loc,1,(\lambda_{0},\lambda_{1})},\Lambda)$.

    Moreover, Proposition \ref{prop:pervconv} implies that, via the above identification, the convolution map
    \begin{equation*}
\tn{Sat}(\Hck^{e,\loc,1,(\lambda_{0},
\lambda_{1})}_{Z \subset G},\Lambda)  \times \tn{Sat}(\Hck^{e,\loc,1,(\lambda_{1},
\lambda_{2})}_{Z \subset G},\Lambda) \to \tn{Sat}(\Hck^{e,\loc,1,(\lambda_{0},
\lambda_{2})}_{Z \subset G},\Lambda) 
    \end{equation*}
    is identified with the tensor product map
    \begin{equation*}
        \tn{Rep}_{\Lambda}(\widehat{G/Z})^{\lambda_{0}/\lambda_{1}} \times \tn{Rep}_{\Lambda}(\widehat{G/Z})^{\lambda_{1}/\lambda_{2}} \to \tn{Rep}_{\Lambda}(\widehat{G/Z})^{\lambda_{0}/\lambda_{2}}.
    \end{equation*}
    Now applying the disjoint union decomposition
    \begin{equation*}
\Hck^{e,\loc,1}_{Z \subset G} = \bigsqcup_{(\lambda_{0}, \lambda_{1}) \in \Hom_{F}(u, Z)^{2}} \Hck^{e,\loc,1,(\lambda_{0},\lambda_{1})}_{Z \subset G}
\end{equation*}
from above, we see that convolution between perverse sheaves supported on $(\lambda_{i},\lambda_{j})$ and $(\lambda_{k}, \lambda_{l})$-components with $\lambda_{j} \neq \lambda_{k}$ yields zero. Together with our previous observations from this proof, we deduce an identification
\begin{equation*}
     \tn{Sat}(\Hck^{e,\loc,1}_{Z \subset G},\Lambda) \xrightarrow{\sim} \bigoplus_{(\lambda_{0}, \lambda_{1}) \in \Hom_{F}(u, Z)^{2}} \tn{Rep}_{\Lambda}(\widehat{G/Z})^{\lambda_{0}/\lambda_{1}},
\end{equation*}
with the monoidal structure $\star$ on the right-hand side. 
\end{proof}

It is also clear that, via the above identification, the pullback map 
\begin{equation*}
    \tn{Sat}(\Hck_{G/Z,\Spd(C)}^{\loc,1},\Lambda) \to \tn{Sat}(\Hck^{e,\loc,1}_{Z \subset G},\Lambda)
\end{equation*}
corresponds to the embedding
\begin{equation}\label{eq:G/Zpullback}
\tn{Rep}_{\Lambda}(\widehat{G/Z}) \hookrightarrow \bigoplus_{(\lambda_{0}, \lambda_{1}) \in \Hom_{F}(u, Z)^{2}} \tn{Rep}_{\Lambda}(\widehat{G/Z})^{\lambda_{0}/\lambda_{1}}
\end{equation}
determined by sending $\pi \in \tn{Rep}_{\Lambda}(\widehat{G/Z})^{\lambda}$ to $\bigoplus_{(\lambda' \lambda,\lambda') \in \Hom_{F}(u,Z)^{2}} \pi$. 

\begin{rque}
Using the alternative description from Remark \ref{rque:altmonoidal}, the embedding \eqref{eq:G/Zpullback} corresponds to pushforward of coherent sheaves along the diagonal
\begin{equation*}
\pt/(\widehat{G/Z}) \xrightarrow{\Delta} \pt/(\widehat{G/Z}) \times_{\pt/\widehat{G}} \pt/(\widehat{G/Z}).
\end{equation*}
\end{rque}

\begin{lem}
    The embedding \eqref{eq:G/Zpullback} is a monoidal functor.
\end{lem}

\begin{proof}
    This is an elementary calculation.
\end{proof}

The following result comes from the proof of Theorem \ref{thm:satake} and the functoriality of the geometric Satake isomorphism from \cite{Geometrization}:

\begin{corol}\label{corol:satZfunc} The isomorphism from Theorem \ref{thm:satake} is functorial in $Z \subset Z' \subset Z_{G}$ via pullback along the transition maps \eqref{eq:Ztrans} and the functor
\begin{equation*}
    \bigoplus_{(\lambda_{0}, \lambda_{1}) \in \Hom_{F}(u, Z)^{2}} \tn{Rep}_{\Lambda}(\widehat{G/Z})^{\lambda_{0}/\lambda_{1}}\to \bigoplus_{(\lambda_{0}, \lambda_{1}) \in \Hom_{F}(u, Z')^{2}} \tn{Rep}_{\Lambda}(\widehat{G/Z'})^{\lambda_{0}/\lambda_{1}}
\end{equation*}
induced by the maps $\widehat{G/Z'} \to \widehat{G/Z}$ and $Z \hookrightarrow Z'$.
\end{corol}

\subsection{Global Hecke stacks}

Recall from \cite[Proposition 19.1.2]{SW20} that, for any $S \in \Perfd_{C}$ and effective Cartier divisor $y' \in \Div^{d}(S)$, there is a faithful functor from the category of triples $(U,V,f)$ consisting of a $G_{S}$-torsor $U$ on $X_{S} \setminus \Gamma_{y'}$, a $G_{S}$-torsor $V$ on $\bbD_{y'}$, and an isomorphism 
\begin{equation*}
f \colon U|_{\bbD_{y'}^{\circ}} \xrightarrow{\sim} V|_{\bbD_{y'}^{\circ}}
\end{equation*}
to the category of $G_{S}$-torsors on $X_{S}$ (see also \cite[Section III.3]{Geometrization}). We call such a triple $(U,V,f)$ a \textbf{Beauville--Laszlo triple} for $X_{S}$. One can define the category of Beauville-Laszlo triples in a completely analogous manner for $[\mc{T}_{S}/\tilde{t}_{S}]$ as well.

\begin{lem}\label{lem:GerbeBL}
There is a functor from the category of Beauville--Laszlo triples for $[\mc{T}_{S}/\tilde{t}_{S}]$ to the category of $G_{S}$-torsors on $[\mc{T}_{S}/\tilde{t}_{S}]$.
\end{lem}

\begin{proof}
By Remark \ref{rque:fingerbe}.(1) we can replace $[\mc{T}_{S}/\tilde{t}_{S}]$ and $[\mc{T}_{S}^{\loc}/\tilde{t}_{S}]$ with $\mc{E}:=[\mc{T}_{E/F,S}/\tilde{t}_{E/F,n,S}]$ and $\mc{E}^{\loc} := [\mc{T}_{E/F,S}^{\loc}/\tilde{t}_{E/F,n,S}]$ for some finite Galois $E/F$ and $n \in \mathbb{N}$. By Remark \ref{rque:fingerbe}.(2) the gerbe $\mc{E}$ is trivialized by the finite \'{e}tale cover $X_{S}^{(E')}$ for some finite Galois extension $E'/F$ and is represented by the \v{C}ech $2$-cocycle $\xi_{\mc{E}} := \xi_{E/F,n}$, which are compatible for varying $E/F$ and $n$.

According to \cite[Proposition 2.50]{Dillery23}, the category of $G_{S}$-torsors on $\mc{E}$ is equivalent to the category of \textbf{$\xi_{\mc{E}}$-twisted $G_{S}$-torsors} whose objects are pairs $(T,\phi)$, where $T$ is a $u_{E/F,n}$-equivariant $G_{S}$-torsor on $X_{S}^{(E')}$ and $\pi_{1}^{*}T \xrightarrow{\phi} \pi_{2}^{*}T$ is an equivariant isomorphism of $G_{S}$-torsors on $X_{S}^{(E')} \times_{X_{S}} X_{S}^{(E')}$ such that $d\phi$ (as defined in \cite[Lemma 2.37]{Dillery23}) is translation by $\xi_{\mc{E}}$. Morphisms are isomorphisms of torsors that are compatible with the respective maps $\phi$.

Given a Beauville-Laszlo triple $(\mc{U},\mc{V},f)$ for $\mc{E}$ (descended from a given triple for $[\mc{T}_{S}/\tilde{t}_{S}]$), we denote by $((U,\phi_{U}),(V,\phi_{V}),g)$ the corresponding triple where $(U,\phi_{U})$ and $(V,\phi_{V})$ are twisted torsors.

The triple $(U,V,g)$ is a Beauville-Laszlo triple for $X_{S}^{(E')}$, so by Beauville-Laszlo for $X_{S}^{(E')}$ we can glue this to a $G_{S}$-torsor $T$ over $X_{S}^{(E')}$. Moreover, the functoriality of Beauville-Laszlo for $X_{S}^{(E')}$, after applying the identification 
\begin{equation*}
X_{S}^{(E')} \times_{X_{S}} X_{S}^{(E')} = \bigsqcup_{\Gal_{E'/F}} X_{S}^{(E')},
\end{equation*}
implies that, since $\phi_{V} = \phi_{U} \circ \pi_{1}^{*}g$ on $\bigsqcup\bbD_{S}^{(E'),\circ}$, we can glue these two maps to obtain an isomorphism of torsors 
\begin{equation*}
    \pi_{1}^{*}T \xrightarrow{\phi} \pi_{2}^{*}T
\end{equation*}
over $X_{S}^{(E')} \times_{X_{S}} X_{S}^{(E')}$. Finally, faithfulness of Beauville-Laszlo for $X_{S}^{(E')}$ implies that since $d\phi_{V}$ and $d\phi_{U}$ are translation by $\xi_{\mc{E}}$, so is $d\phi$. The pair $(T,\phi)$ gives the desired $G_{S}$-torsor on $\mc{E}$. All of these constructions are functorial in the triple $(\mc{U},\mc{V},f)$ because of the functoriality of Beauville-Laszlo for $X_{S}^{(E')}$. 
\end{proof}

Consider the Hecke correspondence
\[
\begin{tikzcd}
  &  \Hck_{Z \subset G}^{e,1} \arrow["\overleftarrow{h}"]{ld} \arrow["\overrightarrow{h}"]{rd} & \\
  \Bun_{G}^{e} & & \Bun_{G}^{e} \times \Spd(C).
\end{tikzcd}
\]

\begin{prop}\label{prop:Heckeproper} The maps $\overleftarrow{h}$ and $\overrightarrow{h}$ for $\Hck^{e,1}_{Z \subset G}$ are ind-proper.
\end{prop}
\begin{proof} It suffices to prove this for each $\Hck^{e,1,(\lambda_{0},\lambda_{1})}_{Z \subset G}$. In this setting, pulling back the closed Schubert cells of $\Hck_{Z \subset G}^{e,\loc,1}$ (cf. Remark \ref{rque:strat}) gives (as a v-stack)
\begin{equation*}
\Hck^{e,1,(\lambda_{0},\lambda_{1})}_{Z \subset G} = \varinjlim_{\nu} \Hck^{e,1,(\lambda_{0},\lambda_{1})}_{Z \subset G,\leq \nu} 
\end{equation*}
for $\nu \in X_{*}(T/Z)^{+,\lambda_{0}/\lambda_{1}}$. Moreover, it suffices to show that $\overrightarrow{h}$ is proper for each $\Hck^{e,1,(\lambda_{0},\lambda_{1})}_{Z \subset G,\leq \nu}$, since $\overleftarrow{h}$ can be identified with $\overrightarrow{h}$ (for $\Hck_{Z \subset G}^{e,1,(\lambda_{1},\lambda_{0})}$) after applying the isomorphism given by switching the torsors.

Working v-locally on the target means that it suffices to show that for $S \to \Bun_{G}^{e,\lambda_{1}}$ strictly totally disconnected the morphism $\Hck_{Z \subset G}^{e,1,(\lambda_{0},\lambda_{1})} \times_{\overrightarrow{h},\Bun_{G}^{e,\lambda_{1}}} S \to S$ is ind-proper. Denote by $\mc{X} \in \Bun_{G}^{e,\lambda_{1}}(S)$ for such $S \in \Perfd_{C}$ the corresponding endpoint; we claim that $\Hck_{Z \subset G}^{e,1,(\lambda_{1},\lambda_{0})} \times_{\overrightarrow{h},\Bun_{G}^{e,\lambda_{1}}} S$ is isomorphic to the extended twisted affine Grassmannian
\begin{equation*}
\tn{Gr}_{Z \subset G,\leq \nu}^{e,(\lambda_{0},\lambda_{1})} \tilde{\times} \mc{X}
\end{equation*}
parametrizing, for $S' \to S$, pairs $(\cP,f)$ of a $G_{S'}$-torsor $\cP$ on $[\mc{T}_{S'}^{\loc}/\tilde{t}_{S'}]$ with $\lambda_{\cP} = \lambda_{0}$ and $f$ an isomorphism from $\Q_{\lambda_{1}/\lambda_{0},S'} \cdot \cP$ to $\mc{X}$ on $[\mc{T}_{S'}^{\loc,\circ}/\tilde{t}_{S'}]$ bounded by $\nu$.
We could then deduce the result by Proposition \ref{prop:mainGrisom}, since identifying $\mc{X}$ with $\widetilde{\Q}_{\lambda_{1},S}$ over $[\mc{T}_{S}^{\loc}/\tilde{t}_{S}]$ (possible by \cite[Lemma III.2.6]{Geometrization} because $S$ is strictly totally disconnected) identifies $\tn{Gr}_{Z \subset G,\leq \nu}^{e,(\lambda_{0},\lambda_{1})} \tilde{\times} \mc{X}$ with $\tn{Gr}_{Z \subset G,\leq \nu}^{e,(\lambda_{0},\lambda_{1})}\times S$. 

To prove this description of the fiber, the same argument as for $\Hck_{G}$ holds provided that one has a surjective ``Beauville-Laszlo'' morphism
\begin{equation*}
    \tn{Gr}_{Z \subset G}^{e,(\lambda_{0},\lambda_{1})} \tilde{\times} \mc{X} \to \Bun_{G}^{e,\lambda_{0}} \times S,
\end{equation*}
of v-stacks. 

The existence of the morphism comes from Lemma \ref{lem:GerbeBL} because a given object $(\cP,f) \in [\tn{Gr}_{Z \subset G}^{e,(\lambda_{0},\lambda_{1})} \tilde{\times} \mc{X}](S)$ defines a Beauville-Laszlo triple
\begin{equation*}
(\cP, \Q_{\lambda_{1}/\lambda_{0},S}^{-1} \cdot\mc{X},\Q_{\lambda_{1}/\lambda_{0},S}^{-1}\cdot f)
\end{equation*}
for $[\mc{T}_{S}/\tilde{t}_{S}]$.
This morphism is surjective by the argument in \cite[Proposition III.3.1]{Geometrization}.
\end{proof}

\begin{corol}\label{corol:CartConv}
    The global-to-local convolution diagram \eqref{eq:CartConv} is Cartesian.
\end{corol}

\begin{proof}
    We prove the case $n=2$; the general case is identical. Suppose we are given $(\cP_{0},\cP_{2},g) \in \Hck_{Z \subset G}^{e,1,(\lambda_{0},\lambda_{2})}(S)$ and $(\cP_{0}',\cP_{1}',\cP_{2}',(f_{0},f_{1})) \in \Hck_{Z \subset G}^{e,\loc,2,(\lambda_{0},\lambda_{1},\lambda_{2})}(S)$ along with an isomorphism $(\cP_{0}^{\loc},\cP_{2}^{\loc},g) \xrightarrow{(h_{0},h_{2})} (\cP_{0}',\cP_{2}', f)$, where $f$ is obtained via convolution (as in \eqref{eq:gconv}) and the superscript ``loc'' denotes the corresponding local torsor.
    
    One fills in the ``missing'' global torsor $\mc{P}_{1}$ in the claimed Cartesian diagram by applying Lemma \ref{lem:GerbeBL} to the Beauville--Laszlo triple $(\Q_{\lambda_{1}/\lambda_{0},S} \cdot \cP_{0}, \cP_{1}',f_{0} \circ (\Q_{\lambda_{1}/\lambda_{0},S } \cdot h_{0}))$, which by construction also defines a modification $\Q_{\lambda_{1}/\lambda_{0},S} \cdot (\cP_{0}|_{[\mc{T}_{S}^{\circ}/\tilde{t}_{S}]}) \xrightarrow{i} \cP_{1}|_{{[\mc{T}_{S}^{\circ}/\tilde{t}_{S}]}}$. 
    
    The required modification $\Q_{\lambda_{2}/\lambda_{1},S} \cdot (\cP_{1}|_{[\mc{T}_{S}^{\circ}/\tilde{t}_{S}]}) \to \cP_{2}|_{[\mc{T}_{S}^{\circ}/\tilde{t}_{S}]}$ is then given by 
    \begin{equation*}
    g \circ z(\lambda_{2}/\lambda_{1},\lambda_{1}/\lambda_{0})_{*}(\psi)\circ (\Q_{\lambda_{2}/\lambda_{1},S} \cdot i^{-1}),
    \end{equation*}
    proving the result.
\end{proof}

\subsection{Extended Hecke action}\label{sec:Extended-Hecke-Action}
This subsection is mostly (until the main theorems at the end) a formal summary of some of the ideas in \cite[Section IX]{Geometrization} and why they still work in this extended setting. In this section, $\Lambda$ is a $\ov{\Z_{\ell}}$-algebra.

First, taking the inverse limit of the isomorphism from Theorem \ref{thm:satake} over all $\Z/\ell^{r}\Z[\sqrt{q}]$ gives an exact $\Z/\ell\Z[\sqrt{q}]$-linear monoidal functor
\begin{equation*}
    \bigoplus_{(\lambda_{0}, \lambda_{1}) \in \Hom_{F}(u, Z)^{2}} \tn{Rep}_{\Z_{\ell}[\sqrt{q}]}(\widehat{G/Z})^{\lambda_{0}/\lambda_{1}} \to \tn{Sat}(\Hck_{Z \subset G}^{e,\loc,1}, \Z_{\ell}[\sqrt{q}]),
\end{equation*}
where the right-hand side is the inverse limit of $\tn{Sat}(\Hck_{Z \subset G}^{e,\loc,1}, \Z/\ell^{r}\Z[\sqrt{q}])$ for all $r$. We pass everything further to $\ov{\Z_{\ell}}$ for simplicity.

Denote by $D_{\blacksquare}(\Bun_{G}^{e},\Lambda)$ the category of solid sheaves of $\Lambda$-modules on the v-site of $\Bun_{G}^{e}$. Recall from \cite[Section IX.2]{Geometrization} that there is a $\tn{Rep}_{\Lambda}(Q)$-linear monoidal functor
\begin{equation*}
\tn{Rep}_{\ov{\Z_{\ell}}}(\widehat{G} \rtimes Q) \to D_{\blacksquare}(\Hck_{G}^{\loc,1}, \ov{\Z_{\ell}})
\end{equation*}
given by composing the functor $V \mapsto \mathcal{S}_{V}$ with $A \mapsto \mathbb{D}(A)^{\vee}$, where $\bbD(A)$ denotes the (relative) Verdier dual in the sense of \cite[Section V.6]{Geometrization}.

In our setting, we obtain an analogous monoidal functor
\begin{equation*}
\bigoplus_{(\lambda_{0}, \lambda_{1}) \in \Hom_{F}(u, Z)^{2}} \tn{Rep}_{\ov{\Z_{\ell}}}(\widehat{G/Z})^{\lambda_{0}/\lambda_{1}} \to \tn{Sat}(\Hck_{Z \subset G}^{e,\loc,1},\ov{\Z_{\ell}})
\end{equation*}
and then composing with $A \mapsto \mathbb{D}(A)^{\vee}$, where $\mathbb{D}$ is the Verdier dual relative to the projection $\Hck_{Z \subset G}^{e,\loc,1} \to \Spd(C)/L^{+}G$
gives an exact monoidal functor
\begin{equation}\label{eq:solidmon}
\bigoplus_{(\lambda_{0}, \lambda_{1}) \in \Hom_{F}(u, Z)^{2}} \tn{Rep}_{\ov{\Z_{\ell}}}(\widehat{G/Z})^{\lambda_{0}/\lambda_{1}}  \to D_{\blacksquare}(\Hck_{Z \subset G}^{e,\loc,1},\ov{\Z_{\ell}}),
\end{equation}
where convolution in $D_{\blacksquare}(\Hck_{Z \subset G}^{e,\loc,1},\ov{\Z_{\ell}})$ is as defined in \cite[Section VII.5]{Geometrization} (all of the machinery developed in this aforementioned section of \cite{Geometrization} works for $\Sat(\Hck_{G}^{\loc,1},\ov{\Z_{\ell}})$ replaced with $\Sat(\Hck_{Z \subset G}^{e,\loc,1,(\lambda_{0},\lambda_{1})},\ov{\Z_{\ell}})$ or $\Sat(\Hck_{Z \subset G}^{e,\loc,1},\ov{\Z_{\ell}})$). 

We can then extend linearly to obtain an exact $\Lambda$-linear monoidal functor
\begin{equation*}
    \bigoplus_{(\lambda_{0}, \lambda_{1}) \in \Hom_{F}(u, Z)^{2}} \tn{Rep}_{\Lambda}(\widehat{G/Z})^{\lambda_{0}/\lambda_{1}}  \to D_{\blacksquare}(\Hck_{Z \subset G}^{e,\loc,1},\Lambda);
\end{equation*}
we denote the image of $V^{e}$ under this functor by $\mc{S}_{V^{e}}'$.

\begin{rque}
    Since each $\Hck_{Z \subset G}^{e,\loc,1,(\lambda_{0},\lambda_{1})} \to \Hck_{G/Z,\Spd(C)}^{\loc,1,\lambda_{0}/\lambda_{1}}$ is a $Z$-gerbe in mixed characteristic and we do not exclude the case where $\ell \mid |Z|$, it is not automatically true that the above functor \eqref{eq:solidmon} (and therefore its extension) is the same as pulling back the corresponding functor for $\Hck_{G/Z}^{\loc,1}$ along this gerbe. This potential discrepancy is due to the failure of taking Verdier dual to commute with pullback. 
\end{rque}

We then obtain a monoidal functor
\begin{equation}\label{eq:monoidal1}
    \bigoplus_{(\lambda_{0}, \lambda_{1}) \in \Hom_{F}(u, Z)^{2}} \tn{Rep}_{\Lambda}(\widehat{G/Z})^{\lambda_{0}/\lambda_{1}}  \to D_{\blacksquare}(\Hck_{Z \subset G}^{e,1},\Lambda)
\end{equation}
by pullback along $\Hck_{Z \subset G}^{e,1} \xrightarrow{\varepsilon} \Hck_{Z \subset G}^{e,\loc,1}$ (using Proposition \ref{prop:hckloc} for the monoidality).

The Hecke correspondence for $ \Hck_{Z \subset G}^{e,1,(\lambda_{0},\lambda_{1})}$ gives rise to a $\Lambda$-linear functor
\begin{equation*}
    D_{\blacksquare}(\Hck_{Z \subset G}^{e,1,(\lambda_{0},\lambda_{1})},\Lambda) \to \tn{Fun}_{\Lambda}(D_{\blacksquare}(\Bun_{G}^{e,\lambda_{0}} \times \Spd(C),\Lambda), D_{\blacksquare}(\Bun_{G}^{e,\lambda_{1}} \times \Spd(C),\Lambda)).
\end{equation*}
Combining this across all $\lambda_{0},\lambda_{1}$ gives an analogous monoidal functor 
\begin{equation}\label{eq:monoidal2}
D_{\blacksquare}(\Hck_{Z \subset G}^{e,1},\Lambda) \to \tn{End}_{\Lambda}(D_{\blacksquare}(\Bun_{G}^{e} \times \Spd(C),\Lambda)).
\end{equation}

We thus obtain the composite functor, for any $V^{e} \in   \bigoplus_{(\lambda_{0}, \lambda_{1}) \in \Hom_{F}(u, Z)^{2}} \tn{Rep}_{\Lambda}(\widehat{G/Z})^{\lambda_{0}/\lambda_{1}}$:
\begin{equation*}
T^{e}_{V^{e}} \colon \mc{D}_{\tn{lis}}(\Bun_{G}^{e},\Lambda) \to D_{\blacksquare}(\Bun_{G}^{e} \times \Spd(C), \Lambda)
\end{equation*}
given by the formula
\begin{equation*}
    A \mapsto \overrightarrow{h}_{\natural}(\overleftarrow{h}^{*}A \prescript{\mathsmaller{\blacksquare}}{}\otimes_{\Lambda}^{L} (\varepsilon^{*}\mathcal{S}_{V^{e}}')),
\end{equation*}
where $\overrightarrow{h}_{\natural}$ is the left adjoint to $\overrightarrow{h}^{*}$ in the category of solid complexes of $\Lambda$-modules.

The following result is an analogue of \cite[Proposition IX.2.1]{Geometrization}:
\begin{prop}
    For any $V^{e} \in  \bigoplus_{(\lambda_{0}, \lambda_{1}) \in \Hom_{F}(u, Z)^{2}} \tn{Rep}_{\Lambda}(\widehat{G/Z})^{\lambda_{0}/\lambda_{1}}$, the functor $T^{e}_{V^{e}}$ defined above restricts to a functor 
    \begin{equation*}
T^{e}_{V^{e}} \colon \mc{D}_{\tn{lis}}(\Bun_{G}^{e}, \Lambda) \to \mc{D}_{\tn{lis}}(\Bun_{G}^{e}, \Lambda).
    \end{equation*}
    The resulting functor
    \begin{equation*}
         \bigoplus_{(\lambda_{0}, \lambda_{1}) \in \Hom_{F}(u, Z)^{2}} \tn{Rep}_{\Lambda}(\widehat{G/Z})^{\lambda_{0}/\lambda_{1}}  \to \tn{End}_{\Lambda}(\mc{D}_{\tn{lis}}(\Bun_{G}^{e},\Lambda)), \hspace{1mm} V^{e} \mapsto T^{e}_{V^{e}},
    \end{equation*}
    is monoidal.
\end{prop}

\begin{proof}
    For the first part, the proof of \cite[Proposition IX.2.1]{Geometrization} holds verbatim if one replaces $\Hck_{G}^{\loc,1}$ (in our notation---the notation \emph{loc. cit.} is different) with $\Hck_{Z \subset G}^{e,\loc,1,(\lambda_{0},\lambda_{1})}$ and uses Proposition \ref{prop:Heckeproper} and the description of the fibers of $\overrightarrow{h}$ from its proof along with the ``extended'' analogue of Demazure resolutions (as in \cite[Section VI.5]{Geometrization}) for the closure of Schubert cells of $\tn{Gr}^{e,(\lambda_{0},\lambda_{1})}_{Z \subset G}$ obtained by pulling back the usual Demazure resolutions for cells (cf. Remark \ref{rque:strat}) of $\tn{Gr}^{\lambda_{0}/\lambda_{1}}_{G/Z,\Spd(C)}$ along the isomorphism $\pi^{(\lambda_{0},\lambda_{1})}$ from Proposition \ref{prop:mainGrisom}. We are also using Corollary \ref{cor:Dlise} for the equivalence between $\mc{D}_{\tn{lis}}(\Bun_{G}^{e},\Lambda)$ and $\mc{D}_{\tn{lis}}(\Bun_{G}^{e} \times \Spd(C),\Lambda)$.

    The second part follows from the fact that $T^{e}_{V^{e}}$ is  obtained by restricting the endomorphism of $D_{\blacksquare}(\Bun_{G}^{e} \times \Spd(C),\Lambda)$ obtained from mapping $V^{e}$ along the composition \eqref{eq:monoidal1} with \eqref{eq:monoidal2}, which is monoidal, to an endomorphism of $\mc{D}_{\tn{lis}}(\Bun_{G}^{e},\Lambda)$.
\end{proof}

There is a second Hecke correspondence from the ``naive'' extended Hecke stacks (Definition \ref{defi:hcknaive}) $\Hck_{Z \subset G}^{e,(I_{i}),\tn{naive}}$ and $\Hck_{Z \subset G}^{e,\loc,(I_{i}),\tn{naive}}$ for a finite set $I$. Explicitly, it is given by 
\[
\begin{tikzcd}
  &  \Hck_{Z \subset G}^{e,(I_{i}),\tn{naive}} \arrow{ld} \arrow{rd} & \\
  \Bun_{G}^{e} & & \Bun_{G}^{e} \times (\Div^{1})^{I}.
\end{tikzcd}
\]

Recall from Proposition \ref{prop:naiveSA} that there is, for each $V \in \tn{Rep}_{\Lambda}((\widehat{G} \rtimes Q)^{I})$, a functor
\begin{equation*}
T^{e}_{V} \colon  \mc{D}_{\tn{lis}}(\Bun_{G}^{e},\Lambda) \to \mc{D}_{\tn{lis}}(\Bun_{G}^{e},\Lambda)^{B\Weil_{F}^{I}}.
\end{equation*}
These functors are defined by running the \cite{Geometrization} machinery for  a completely identical recipe as in \cite{Geometrization} but for the naive extended Hecke stacks, using the decomposition \eqref{eq:naivedecomp} to rigorously justify that all of the proofs go through verbatim.

In particular, the assignment $V \mapsto T^{e}_{V}$ gives a monoidal functor 
\begin{equation*}
    \tn{Rep}_{\Lambda}(\widehat{G} \rtimes Q) \to \tn{End}_{\tn{Rep}_{\Lambda}(Q)}(\mc{D}_{\tn{lis}}(\Bun_{G}^{e},\Lambda))^{B\Weil_{F}}.
\end{equation*}

Observe that base-changing $\Hck_{Z \subset G}^{e,\{0,1\},\tn{naive}}$ along our fixed divisor $\Spd(C) \xrightarrow{y} \Div^{1}$ yields a Cartesian diagram
\[
\begin{tikzcd}
\Hck_{Z \subset G}^{e,\{0,1\},\tn{naive}} \times_{\Div^{1}}  \Spd(C) \arrow{r} \arrow{d} & \Hck_{Z \subset G}^{e,1} \arrow{d} \\
\Hck_{Z \subset G}^{e,\loc,\{0,1\},\tn{naive}} \times_{\Div^{1}} \Spd(C) \arrow{r}  & \Hck_{Z \subset G}^{e,\loc,1}
\end{tikzcd}
\]
and that we have an embedding
\begin{equation*}
    \Hck_{Z \subset G}^{e,\{0,1\},\tn{naive}} \times_{\Div^{1}}  \Spd(C) \to \Hck_{Z \subset G}^{e,1}
\end{equation*}
which realizes $\Hck_{Z \subset G}^{e,\{0,1\},\tn{naive}} \times_{\Div^{1}} \Spd(C)$ as the substack
\begin{equation}\label{eq:naiveincl}
     \bigsqcup_{\lambda \in \Hom_{F}(u,Z)} \Hck_{Z \subset G}^{e,1,(\lambda,\lambda)} \hookrightarrow \bigsqcup_{(\lambda,\lambda') \in \Hom_{F}(u,Z)^{2}} \Hck_{Z \subset G}^{e,1,(\lambda,\lambda')} = \Hck_{Z \subset G}^{e,1},
\end{equation}
and that this also holds for the local analogues of both stacks.

Using these maps, we obtain:

\begin{thm}[Extended Hecke action]\label{thm:extended-hecke-action}
\begin{enumerate}
\item{There is a $\Weil_F^{\bullet}$-linear monoidal natural transformation 
\begin{equation*}
\bigoplus_{\Hom_{F}(u,Z)}\Rep_{\Lambda}(({^L}G)^{\bullet}) \xrightarrow{T} \End_{\Lambda}(\mc{D}_{\tn{lis}}(\Bun_{G}^{e},\Lambda))^{B\Weil_F^{\bullet}},
\end{equation*}
where the monoidal structure on the left-hand side is given by the usual tensor product on summands in the same graded component and is zero otherwise.
}
\item{There is a commutative diagram in $\Cat^{\otimes,\ex, \mathrm{st}}_{\Lambda}$  
\[\begin{tikzcd}
	{ \bigoplus_{(\lambda_{0}, \lambda_{1}) \in \Hom_{F}(u, Z)^{2}} \tn{Rep}_{\Lambda}(\widehat{G/Z})^{\lambda_{0}/\lambda_{1}} } & {\End_{\Lambda}(\mc{D}_{\tn{lis}}(\Bun_{G}^{e},\Lambda)) } \\
	{\bigoplus_{\Hom_{F}(u,Z)}\Rep_{\Lambda}({^L}G)} & {\End_{\Lambda}(\mc{D}_{\tn{lis}}(\Bun_{G}^{e},\Lambda))^{B\Weil_{F}},}
	\arrow["{T^e}", from=1-1, to=1-2]
	\arrow["\oblv", from=2-1, to=1-1]
	\arrow["T"', from=2-1, to=2-2]
	\arrow["\oblv"', from=2-2, to=1-2]
\end{tikzcd}\]}
\end{enumerate}
where the ``$\tn{oblv}$'' map first restricts to $\bigoplus_{\Hom_{F}(u,Z)}\Rep_{\Lambda}(\widehat{G})$ and then embeds via the map $\lambda \mapsto (\lambda,\lambda)$ on the gradings.
\end{thm}

\begin{rque}
    Part (1) of the above theorem implies that there is a $W_{F}^{\bullet}$-linear action of the monoidal category $\Rep_{\Lambda}(({^L}G)^{\bullet})$ on $\mc{D}_{\tn{lis}}(\Bun_{G}^{e},\Lambda) = \prod_{\lambda \in \Hom_{F}(u,Z)} \mc{D}_{\tn{lis}}(\Bun_{G}^{e,\lambda},\Lambda)$ which preserves each factor in the product.
\end{rque}

\begin{proof}
    The first part is Proposition \ref{prop:naiveSA}, which, to recall, just breaks up the naive Hecke stack and applies the theory of \cite{Geometrization} to each component. 
    
    For the second part, note that one has separate Hecke operators associated to $\Hck_{Z \subset G}^{e,\loc,\{0,1\},\tn{naive}} \times_{\Div^{1}} \Spd(C)$ and $\Hck_{Z \subset G}^{e,\loc,\{0,1\},\tn{naive}}$ (and their global analogues), which fit into a commutative square
    \[\begin{tikzcd}
	{\bigoplus_{\Hom_{F}(u,Z)} \Rep_{\Lambda}\widehat{G}} & {\End_{\Lambda}(\mc{D}_{\tn{lis}}(\Bun_{G}^{e},\Lambda)) } \\
	{\bigoplus_{\Hom_{F}(u,Z)} \Rep_{\Lambda}{^L}G} & {\End_{\Lambda}(\mc{D}_{\tn{lis}}(\Bun_{G}^{e},\Lambda))^{B\Weil_{F}}}.
	\arrow[from=1-1, to=1-2]
	\arrow["\oblv", from=2-1, to=1-1]
	\arrow["T"', from=2-1, to=2-2]
	\arrow["\oblv"', from=2-2, to=1-2]
\end{tikzcd}\]
   The claimed commutativity holds by \cite[Theorem 1.4.4]{Hansen24} (which itself summarizes \cite{Geometrization}), after applying the decomposition \eqref{eq:naivedecomp} to reduce to the non-extended case.

   We deduce the result from the commutativity of
    \[\begin{tikzcd}
	{\bigoplus_{(\lambda_{0},\lambda_{1}) \in \Hom_{F}(u,Z)^{2}}\Rep_{\Lambda}(\widehat{G/Z})^{\lambda_{0}/\lambda_{1}}} & {\End_{\Lambda}(\mc{D}_{\tn{lis}}(\Bun_{G}^{e},\Lambda)) } \\
	{\bigoplus_{\Hom_{F}(u,Z)}\Rep_{\Lambda}({\widehat{G})}} & {\End_{\Lambda}(\mc{D}_{\tn{lis}}(\Bun_{G}^{e},\Lambda))}
	\arrow["T^{e}", from=1-1, to=1-2]
	\arrow[from=2-1, to=1-1]
	\arrow[from=2-1, to=2-2]
	\arrow["\tn{id}"', from=2-2, to=1-2],
\end{tikzcd}\]
which is immediate from the inclusion \eqref{eq:naiveincl} and the above construction of the operators $T^{e}_{V^{e}}$.
\end{proof}

 By considering the monoidal embeddings
 \begin{equation*}
   \Rep_{\Lambda}({^L}G) \hookrightarrow \bigoplus_{\Hom_{F}(u,Z)}\Rep_{\Lambda}{^L}G, \hspace{1mm} \Rep_{\Lambda}(\widehat{G/Z) } \xrightarrow{\eqref{eq:G/Zpullback}} \bigoplus_{(\lambda_{0},\lambda_{1}) \in \Hom_{F}(u,Z)^{2}}\Rep_{\Lambda}(\widehat{G/Z})^{\lambda_{0}/\lambda_{1}},
 \end{equation*}
 one obtains:

 \begin{corol}\label{cor:maindiag}
 \begin{enumerate}
 \item{There is a $\Weil_F^{\bullet}$-linear monoidal natural transformation 
\begin{equation*}
\Rep_{\Lambda}(({^L}G)^{\bullet}) \to \End_{\Lambda}(\mc{D}_{\tn{lis}}(\Bun_{G}^{e},\Lambda))^{B\Weil_F^{\bullet}};
\end{equation*}
}
    \item{There is a commutative diagram in $\Cat^{\otimes,\ex, \mathrm{st}}_{\Lambda}$  
\[\begin{tikzcd}
	{\Rep_{\Lambda}(\widehat{G/Z})} & {\End_{\Lambda}(\mc{D}_{\tn{lis}}(\Bun_{G}^{e},\Lambda)) } \\
	{\Rep_{\Lambda}({^L}G}) & {\End_{\Lambda}(\mc{D}_{\tn{lis}}(\Bun_{G}^{e},\Lambda))^{B\Weil_{F}}.}
	\arrow["{T^e}", from=1-1, to=1-2]
	\arrow["\oblv", from=2-1, to=1-1]
	\arrow["T"', from=2-1, to=2-2]
	\arrow["\oblv"', from=2-2, to=1-2]
\end{tikzcd}\]}

\end{enumerate}
 \end{corol}

\begin{thm}\label{thm:diagHckact}
     There is a commutative diagram in $\Cat^{\otimes,\ex, \mathrm{st}}_{\Lambda}$  
\[\begin{tikzcd}
	{\Rep_{\Lambda}(\widehat{G}^{e})} & {\End_{\Lambda}(\mc{D}_{\tn{lis}}(\Bun_{G}^{e},\Lambda)) } \\
	{\Rep_{\Lambda}({^L}G}) & {\End_{\Lambda}(\mc{D}_{\tn{lis}}(\Bun_{G}^{e},\Lambda))^{B\Weil_{F}}.}
	\arrow["{T^e}", from=1-1, to=1-2]
	\arrow["\oblv", from=2-1, to=1-1]
	\arrow["T"', from=2-1, to=2-2]
	\arrow["\oblv"', from=2-2, to=1-2]
\end{tikzcd}\]
Moreover, for $V^{e} \in \Rep_{\Lambda}(\widehat{G}^{e})$, the action of $T^{e}_{V^{e}}$ on $\mc{D}_{\tn{lis}}(\Bun_{G}^{e},\Lambda)$ preserves all
limits and colimits, and the full subcategories of compact objects.
 \end{thm}

 \begin{proof}
     Fix $Z \subset Z' \subset Z_{G}$ both finite. The construction of the Hecke action for $Z$ takes place on the closed and open substack
     \begin{equation*}
\Hck_{Z \subset G}^{e,1} = \bigsqcup_{(\lambda_{0},\lambda_{1}) \in \Hom_{F}(u,Z)^{2}} \Hck_{Z' \subset G}^{e,1,(\lambda_{0},\lambda_{1})} \hookrightarrow  \bigsqcup_{(\lambda_{0},\lambda_{1}) \in \Hom_{F}(u,Z')} \Hck_{Z' \subset G}^{e,1,(\lambda_{0},\lambda_{1})} = \Hck_{Z' \subset G}^{e,1},
     \end{equation*}
     in a way which is compatible with the decomposition 
     \begin{equation*}
\Bun_{Z \subset G}^{e} = \bigsqcup_{\lambda \in \Hom_{F}(u,Z)} \Bun_{G}^{e,\lambda} \hookrightarrow \bigsqcup_{\lambda \in \Hom_{F}(u,Z')} \Bun_{G}^{e,\lambda} = \Bun_{Z' \subset G}^{e}.
     \end{equation*}
     The first part of the theorem then follows by Corollary \ref{cor:maindiag}, Lemma \ref{lem:changeinZconv}, and taking the direct limit over all $Z$.

     The proof of the second part comes from the identical argument as in the proof of \cite[Theorem IX.2.2]{Geometrization} (for each finite $Z$).
 \end{proof}

 \subsection{Dependence on choices}
Continue the notation of the previous subsections. There were two choices made above: 
\begin{enumerate}
\item{The choice of isomorphism $\hat{\ov{F}} \xrightarrow{\sim} C^{\sharp}$, used to define the untilts $\{y^{(E)}\}_{E/F}$ lifting a given $y$ required for the identification $\tn{Gr}_{t}(C) \xrightarrow{\sim} X_{*}(t)$;}
\item{The choice of section $s$ of the surjection
\begin{equation}\label{eq:sectionsurj}
     \Hom_{F}(\tilde{t}, \mathbf{T}) \times_{\Hom_{F}(u,\mathbf{T})} \Hom_{F}(u,Z) \to \Hom_{F}(u,Z).
 \end{equation}
}
\end{enumerate}

\begin{rque}
    Although the pinning of $G$ was required to define each $\tn{Gr}_{Z \subset G}^{e,(\lambda_{0},\lambda_{1})}$, the Hecke action constructed above only depends on $\Hck_{Z \subset G}^{e,1,(\lambda_{0},\lambda_{1})}$, whose definition is independent of this choice.
\end{rque}

\subsubsection{The divisors}\label{sec:choiceoflifts}
The choice of a fixed $y \in \Div^{1}(C)$ already implicitly appears in \cite{Geometrization} by using the \'{e}tale fundamental group of $\Div^{1}$. For fixed $y$, any two choices of systems of untilts $\{y_{1}^{(E)}\}_{E/F}$ and $\{y_{2}^{(E)}\}_{E/F}$ are conjugate under some unique $\sigma \in \Gal_{F}$. The difference between the two resulting trivializations $\psi_{2} \circ \psi_{1}^{-1}$ is translation by a unique element $t_{\sigma} \in t(X_{C}^{\circ})$ from which one obtains $dt_{\sigma} \in Z^{1}(X_{C}^{\circ},u)$.

\begin{lem}\label{lem:sigmaaut}
    The class $[dt_{\sigma}] \in H_{\tn{v}}^{1}(X_{C}^{\circ},u)$ is nontrivial. In fact, it is even nontrivial in $\tn{Coker}[H_{\tn{fppf}}^{1}(\D_{C},u) \to H_{\tn{fppf}}^{1}(\D_{C}^{\circ},u)]$.
\end{lem}

\begin{proof}
Combining the Kummer sequence with Shapiro's Lemma gives an identification
\begin{equation*}
H_{\tn{v}}^{1}(X_{C}^{\circ},u) \xrightarrow{\sim} \varprojlim \Z[\Gal_{E/F}]_{0},
\end{equation*}
and under this map $[dt_{\sigma}]$ maps to $[e] -[\sigma]$.

For the second claim, we use purity for closed immersions (\cite[Theorem 7.1.2]{CS24}) to identify $\tn{Coker}[H_{\tn{fppf}}^{1}(\D_{C},u) \to H_{\tn{fppf}}^{1}(\D_{C}^{\circ},u)]$ with $H^{0}(C, u(-1)) = \varprojlim \widehat{\Z}[\Gal_{E/F}]$ and observe that $[dt_{\sigma}]$ again maps to $[e] - [\sigma]$. 
\end{proof}

In particular, since $H_{\tn{v}}^{1}(X_{C},u) = H_{\tn{fppf}}^{1}(F,u) = 0$ by Proposition \ref{prop:ucohom} (using \cite[Th\'{e}or\`{e}me 11.4]{Fargues22} for the first equality), the class $[dt_{\sigma}]$ does not lift to a class defined over $X_{C}$ and, locally, does not lift to a class defined over $\D_{C}$.

The cocycle $dt_{\sigma}$ defines an automorphism of the gerbe $[\mc{T}_{C}^{\circ}/\tilde{t}_{C}]$, denoted by $z_{\sigma}$. 

\begin{lem}
    For $\lambda \in \Hom_{F}(u,Z)$ and $S \in \Perfd_{C}$, the $Z_{S}$-torsor obtained using $\psi_{2}$, denoted by $\Q^{(2)}_{\lambda,S}$, is related to $\Q^{(1)}_{\lambda,S}$ (defined using $\psi_{1}$) by the formula 
    \begin{equation*}
\Q^{(2)}_{\lambda,S} = z_{\sigma,*}(\Q^{(1)}_{\lambda,S}).
    \end{equation*}
\end{lem}

Lemma \ref{lem:sigmaaut} (and the sentence following it) implies there is no automorphism of $[\mc{T}_{C}/\tilde{t}_{C}]$, or even of $[\mc{T}_{C}^{\loc}/\tilde{t}_{C}]$, lifting $z_{\sigma}$, and therefore no obvious way to use $z_{\sigma}$ to define an isomorphism of the (local or global) Hecke stacks associated to the two distinct $\{y_{i}^{(E)}\}$. It is therefore not clear to us what effect the choice of $\{y^{(E)}\}_{E/F}$ (coming from the choice of $\hat{\ov{F}} \xrightarrow{\sim} C^{\sharp}$) has on the resulting Hecke action (Theorem \ref{thm:extended-hecke-action}).

\begin{rque}
It is interesting to note the parallel between the necessity of choosing the lifts $\{y^{(E)}\}_{E/F}$ for $y \in \Div^{1}(C)$ in this paper and the necessity of choosing a section of the map $V_{\ov{F}} \to V_{F}$ (where $V$ denotes the set of places of a global field $F$) in order to define the global Kaletha gerbe in \cite{Kaletha18b}.
\end{rque}

\subsubsection{The section $s$}
The construction of the extended Hecke stacks (Definitions \ref{defi:elochck} and \ref{defi:eglobhck}) relied on a choice of section $s$ for the surjection \eqref{eq:sectionsurj}.
 
 Using Shapiro's lemma, we may identify $\Hom_{F}(\tilde{t},\mathbf{T})$ with $X_{*}(\mathbf{T})_{\mathbb{Q}}$ and $\Hom_{F}(u,\mathbf{T})$ with the group $X_{*}(\mathbf{T})_{\mathbb{Q}/\Z}$. It follows that picking a section of $\mathbb{Q} \to \mathbb{Q}/\mathbb{Z}$ determines a section of 
 \begin{equation*}
\Hom_{F}(\tilde{t},\mathbf{T}) \to \Hom_{F}(u,\mathbf{T}),
 \end{equation*}
 and therefore, by restricting to $\Hom_{F}(u,Z)$, the desired section. We can thus make a ``distinguished'' choice of section by taking the one corresponding to the section $\mathbb{Q}/\Z \to \mathbb{Q}$ valued in $[0,1)$. However, we will show shortly that the resulting Hecke action (Theorem \ref{thm:extended-hecke-action}) does not depend on the section $s$ at all. 
 
Any two sections $s_{1}$ and $s_{2}$ are related by the formula $s_{2} = s_{1}c$, where $\Hom_{F}(u,Z_{G}) \xrightarrow{c} \Hom_{F}(t,\mathbf{T})$ is a function (a ``$1$-cochain''). Denote by $\prescript{i}{}\Hck^{e,-}_{Z \subset G}$ the corresponding extended (local or global) Hecke stacks for $i=1,2$. There is an isomorphism
\begin{equation*}
    \prescript{1}{}\Hck^{e,-,n}_{Z \subset G} \xrightarrow{\eta_{c}} \prescript{2}{}\Hck^{e,-,n}_{Z \subset G}
\end{equation*}
sending (for $n=1$) $(\cP_{0}, \cP_{1}, f)$ to $(\cP_{0}, \cP_{1}, f \circ c(\lambda_{1}/\lambda_{0})_{*}(\psi)),$ where by $f \circ c(\lambda_{1}/\lambda_{0})_{*}(\psi)$ we mean the composition (using Lemma \ref{lem:curlyT})
\begin{equation*}
    \Q^{(2)}_{\lambda_{1}/\lambda_{0},S} \cdot \cP_{0} = (\Q^{(1)}_{\lambda_{1}/\lambda_{0},S} \cdot c(\lambda_{1}/\lambda_{0})_{*}(\mc{T}_{S})) \cdot \cP_{0} \xrightarrow{c(\lambda_{1}/\lambda_{0})_{*}(\psi)}\Q^{(1)}_{\lambda_{1}/\lambda_{0},S}  \cdot \cP_{0} \xrightarrow{f} \cP_{1}.
\end{equation*}

\begin{lem}
\begin{enumerate}
    \item{The map $\eta_{c}$ is compatible with convolution.}
    \item{The map $\eta_{c}$ commutes with the projection maps $\prescript{i}{}\Hck^{e,\loc,1}_{Z \subset G} \to \Hck_{G/Z}^{\loc,1}$ from Proposition \ref{prop:Hckquot}.}
    \end{enumerate}
\end{lem}

\begin{proof}
    The first claim is a straightforward calculation. Unpacking the constructions, the second claim reduces to proving the identity
    \begin{equation*}
        \varphi^{(1)}_{\lambda_{1}/\lambda_{0},\cP_{0}} = \ov{c(\lambda_{1}/\lambda_{0})_{*}(\psi)} \circ \varphi^{(2)}_{\lambda_{1}/\lambda_{0},\cP_{0}},
    \end{equation*}
    where $\varphi^{(i)}_{\lambda_{1}/\lambda_{0},\cP_{0}}$ denotes the canonical isomorphism $\ov{\cP_{0}} \xrightarrow{\sim} \ov{\Q_{\lambda_{1}/\lambda_{0},S}^{(i)} \cdot \cP_{0}}$ from \eqref{eq:varphi}. The desired identity then follows from the equality
    \begin{equation*}
        s_{1}(\lambda_{1}/\lambda_{0})_{*}(\psi^{-1}) = c(\lambda_{1}/\lambda_{0})_{*}(\psi) \circ s_{2}(\lambda_{1}/\lambda_{0})_{*}(\psi^{-1}) = c(\lambda_{1}/\lambda_{0})_{*}(\psi) \circ [s_{1}(\lambda_{1}/\lambda_{0})c(\lambda_{1}/\lambda_{0})]_{*}(\psi^{-1} )
    \end{equation*}
    of isomorphisms of torsors from $\underline{\mathbf{T}_{S}}$ to $s_{1}(\lambda_{1}/\lambda_{0})_{*}(\mc{T}_{S}^{\loc,\circ})$. 
\end{proof}

Putting the above lemma together with basic functoriality properties gives:

\begin{prop}
  All of the maps in the commutative diagram in part (2) of Theorem \ref{thm:extended-hecke-action} remain the same for any choice of section $s$ of \eqref{eq:sectionsurj}.
\end{prop}

\section{Extended Spectral action}\label{sec:extended-spectral}
We continue to assume that $G$ is a quasi-split connected reductive group over $F$.
\subsection{Extended stack of parameters}

Recall that we denote by $\Par_G := \Par_G^{\square}/\widehat{G}$ the stack of parameters of $G$ as introduced by \cite{Geometrization}, \cite{DHKM} and \cite{ZhuCoherentSheaves}. Recall as well that we have introduced the extended dual group
$$\widehat{G}^e = \varprojlim_{Z} \widehat{G/Z}$$
where the limit ranges through finite central subgroups $Z \subset Z_G$. As in \cite{Fargues22}, we introduce the extended stack of parameters to be 
$$\Par_G^e := \Par_G^{\square}/\widehat{G}^e.$$
We note that it naturally fits into the following pullback square of stacks over $\Spec(\Lambda)$ 
\begin{equation}\label{eq:diag-defi-par-g-e}
    \begin{tikzcd}
	{\Par_G^e} & {\pt/\widehat{G}^e} \\
	{\Par_G} & {\pt/\widehat{G}.}
	\arrow[from=1-1, to=1-2]
	\arrow[from=1-1, to=2-1]
	\arrow[from=1-2, to=2-2]
	\arrow["\alpha", from=2-1, to=2-2]
\end{tikzcd}
\end{equation}

\subsection{Spectral action patterns}

In this subsection, we fix $\Ccal \in \Pr_{\Lambda}$ a  compactly generated $\Lambda$-linear presentable category. The goal of this subsection is to explain what data determines an action of $\Perf(\Par_G^e)$ on $\Ccal$.

Let us define functors $\FinSet \to \Cat^{\otimes, \ex, \mathrm{st}}_{\Lambda}$

\begin{equation*}
\Rep_{\Lambda}(({^L}G)^{\bullet}) : I \mapsto \Rep_{\Lambda}(({^L}G)^I);
\end{equation*}
here, the category $\Rep(\Gamma)$ has to be understood as $\Perf(B\Gamma)$.

We denote by
\begin{equation*}\label{eq:natTransB}
\End(\Ccal)^{B\Weil_F^{\bullet}} : I \mapsto \End(\Ccal)^{B\Weil_F^I}.
\end{equation*}

\begin{thm}[\cite{AGKRRV}, \protect{\cite[Theorem X.0.1]{Geometrization}}]\label{thm:spectralDecompositionTheorem}
Assume that $\ell$ does not divide $|\pi_1(\widehat{G})_{\tn{tor}}|$ if $\Lambda = \Zlb, \Flb$, then there is an equivalence of categories between $\Weil_F^{\bullet}$-linear natural transformations of functors
$$\Rep_{\Lambda}(({^L}G)^{\bullet}) \to \End(\Ccal)^{B\Weil_F^{\bullet}}$$
and compactly supported actions of $\Perf(\Par_G)$ on $\Ccal$. 
\end{thm}

\begin{rque}
    Recall from \cite{Geometrization} that compactly supported means that for all compact objects $X \in \Ccal$, the functor 
    \begin{align*}
        \Perf(\Par_G) &\to \Ccal \\
        \mathcal{F} &\mapsto \mathcal{F} * X,
    \end{align*}
    factors through a quasi-compact substack $\Par_G' \subset \Par_G$ of parameters which are trivial on some small enough compact open subgroup $P' \subset P_F$ of the wild inertia. 
\end{rque}

We now want to state and prove the extension to get actions of $\Perf(\Par_G^e)$. 

\begin{thm}\label{thm:extendedSpectralDecompositionTheorem}
Assume $\ell$ is as in Theorem \ref{thm:spectralDecompositionTheorem}, then there is a canonical equivalence between quasi-compact actions of $\Perf(\Par_G^e)$ on $\mc{C}$ and the following data :  
\begin{enumerate}
\item a $\Weil_F^{\bullet}$-linear natural transformation $\Rep_{\Lambda}(({^L}G)^{\bullet}) \to \End(\Ccal)^{B\Weil_F^{\bullet}}$, 
\item a commutative diagram in $\Cat^{\otimes,\ex, \mathrm{st}}_{\Lambda}$  
\begin{equation}\label{diag:required-commutative-diagram}
    \begin{tikzcd}
	{\Rep_{\Lambda}(\widehat{G}^e}) & {\End(\Ccal) } \\
	{\Rep_{\Lambda}({^L}G}) & {\End(\Ccal)^{B\Weil_F}}
	\arrow["{T^e}", from=1-1, to=1-2]
	\arrow["\oblv", from=2-1, to=1-1]
	\arrow["T"', from=2-1, to=2-2]
	\arrow["\oblv"', from=2-2, to=1-2]
\end{tikzcd}
\end{equation}
\end{enumerate}
\end{thm}

\begin{proof}
    It follows from \cite[Theorem X.0.1]{Geometrization} that there is a well-defined quasi-compact action of $\Perf(\Par_G)$ on $\Ccal$. 

    Recall also from \emph{loc. cit.} that for each $V \in \Rep_{\Lambda}{^L}G$, there is a canonical vector bundle $\Ecal_V \in \Perf(\Par_G)^{B\Weil_F}$, such that the functor 
    $$\Ccal \xrightarrow{\Ecal_V * - } \Ccal^{B\Weil_F}$$
    is canonically isomorphic to the functor $T_V$. Let us denote by $\mathrm{univ} : \Rep_{\Lambda}({^L}G) \to \Perf(\Par_G)^{B\Weil_F}$ the functor $V \mapsto \Ecal_V$. There is a commutative diagram in $\Cat_{\Lambda}^{\otimes, \ex, \mathrm{st}}$
    \[\begin{tikzcd}
	{\Rep_{\Lambda}({^L}G}) & {\Perf(\Par_G)^{B\Weil_F}} & {\End(\Ccal)^{B\Weil_F}} \\
	{\Rep_{\Lambda}(\widehat{G}}) & {\Perf(\Par_G)} & {\End(\Ccal)}
	\arrow["{\mathrm{univ}}", from=1-1, to=1-2]
	\arrow["\oblv"', from=1-1, to=2-1]
	\arrow[from=1-2, to=1-3]
	\arrow["\oblv", from=1-2, to=2-2]
	\arrow["\oblv", from=1-3, to=2-3]
	\arrow["{\alpha^*}"', from=2-1, to=2-2]
	\arrow[from=2-2, to=2-3]
\end{tikzcd}\]

    The commutative diagram \eqref{diag:required-commutative-diagram} yields a commutative diagram in $\Cat_{\Lambda}^{\otimes, \ex, \mathrm{st}}$
    \[\begin{tikzcd}
	{\Rep_{\Lambda}(\widehat{G}}) & {\Perf(\Par_G)} \\
	{\Rep_{\Lambda}(\widehat{G}^e}) & {\End(\Ccal).}
	\arrow["{\alpha^*}", from=1-1, to=1-2]
	\arrow["\oblv"', from=1-1, to=2-1]
	\arrow[from=1-2, to=2-2]
	\arrow["{T^e}"', from=2-1, to=2-2]
\end{tikzcd}\]
    It follows that there is a monoidal functor 
    $$\Perf(\Par_G) \otimes_{\Perf(\pt/\widehat{G})} \Perf(\pt/\widehat{G}^e) \to \End(\Ccal).$$
    By \Cref{thm:relative-tensor-product}, we have $\Perf(\Par_G) \otimes_{\Perf(\pt/\widehat{G})} \Perf(\pt/\widehat{G}^e) \cong \Perf(\Par_G^e)$ and the theorem is proven.
\end{proof}

The pullback square \eqref{eq:diag-defi-par-g-e} yields a map 
\begin{equation}\label{eq:relative-tensor-product}
   \Ind\Perf(\Par_G) \otimes_{\Ind\Perf(\pt/\widehat{G})} \Ind\Perf(\pt/\widehat{G}^e) \to \Ind\Perf(\Par_G^e).
\end{equation}

\begin{thm}\label{thm:relative-tensor-product}
    Assume that $\ell$ does not divide the order of $\pi_1(\widehat{G})_{\tn{tor}}$ if $\Lambda = \Flb, \Zlb$, the functor \eqref{eq:relative-tensor-product} is an equivalence.
\end{thm}

\begin{proof}
    If $\Lambda = \Qlb$, this follows from \cite[Theorem 4.7]{BenZviFrancisNadler}. We now assume that $\Lambda = \Flb, \Zlb$; by \cite[Theorem VIII.5.1]{Geometrization}, the map $\Par_G \to \pt/\widehat{G}$ yields an equivalence
    $$\Ind\Perf(\Par_G) \cong \Ocal(\Par_G^{\square})-\Mod(\Ind\Perf(\pt/\widehat{G})).$$
    The statement now follows formally from \cite[Chapter 1, Corollary 8.5.7]{StudyInDerivedAlgGeomI}.
\end{proof}

\subsection{Spectral action}
We now specialize \Cref{thm:extendedSpectralDecompositionTheorem} to the case of $\mc{D}_{\tn{lis}}(\Bun_{G}^{e},\Lambda)$. 

\begin{thm}\label{thm:extended-spectral-action-on-bun_G}
    Assume that $\ell$ does not divide the order of $\pi_1(\widehat{G})_{\tn{tor}}$ if $\Lambda = \Zlb, \Flb$, there is an action of $\Perf(\Par_G^e)$ on $\mc{D}_{\tn{lis}}(\Bun_{G}^{e},\Lambda)$ satisfying the following property: 
    
    Let $\lambda \in \Hom_F(u, Z_{G})$, upon identifying $\Bun_G^{e,\lambda}$ with $\Bun_{G_b}$ as in Proposition \ref{prop:BunGestructure} $(3)$, the restriction of the action of $\Perf(\Par_G^e)$ to $\Perf(\Par_G)$ is the Fargues--Scholze spectral action. 
\end{thm}

\begin{proof}
This is a direct combination of Theorems \ref{thm:diagHckact} and \ref{thm:extendedSpectralDecompositionTheorem}.
\end{proof}

\subsection{Gradings}
Denote by $Z_{\widehat{G}}^{\Gal_{F},+}$ the preimage of $Z_{\widehat{G}}^{\Gal_{F}}$ in $\widehat{G}^{e}$. Observe that $\Rep_{\Lambda}(\widehat{G}^{e})$ carries a natural $X^{*}(Z_{\widehat{G}}^{\Gal_{F},+})$-grading (and the same is true for $Z_{\widehat{G}}$ replaced by $\widehat{Z}$ for any finite $Z \subseteq Z_{G}$), and, as Fargues proves in \cite[Section 12]{Fargues22}, we have a canonical identification
\begin{equation*}
X^{*}(Z_{\widehat{G}}^{\Gal_{F},+}) \xrightarrow{\sim} \pi_{1}(G)^{e}_{\Gal_{F}}.
\end{equation*}
We will use this identification without comment in the following.

\begin{lem}[Extended analogue of \protect{\cite[Lemma 5.3.2]{Zou24}}]\label{lem:Konradlem1}
For $x \in \pi_{1}(G)^{e}_{\Gal_{F}}$, the Hecke operator $T^{e}_{V^{e}}$ associated to $V^{e} \in \Rep_{\Lambda}(\widehat{G}^{e})^{\chi}$ for $\chi \in X^{*}(Z_{\widehat{G}}^{\Gal_{F},+})$ via Theorem \ref{thm:diagHckact} sends $A \in \mc{D}_{\tn{lis}}(\Bun_{G}^{e,\kappa=x}, \Lambda)^{\omega}$ to $T^{e}_{V^{e}}(A)\in \mc{D}_{\tn{lis}}(\Bun_{G}^{e,\kappa= x-\chi}, \Lambda)^{\omega}$.
\end{lem}

\begin{proof}
    By taking $Z$ sufficiently large and finite we can assume that $V^{e} \in \Rep_{\Lambda}(\widehat{G/Z})^{\chi}$ and that $A \in \mc{D}_{\tn{lis}}(\Bun_{Z \subset G}^{e,\kappa=x}, \Lambda)^{\omega}$. Denote by $\lambda_{x},\lambda_{\chi} \in \Hom_{F}(u,Z)$ the inertial morphisms associated to $x,\chi \in \pi_{1}(G)^{e}_{\Gal_{F}}$. By construction of $T^{e}_{V^{e}}$ (cf. the proof of Theorem \ref{thm:extended-hecke-action}) it suffices to prove this result for the Hecke correspondence 
    \[
\begin{tikzcd}
  &  \Hck_{Z \subset G}^{e,1,(\lambda_{x},\lambda_{x}/\lambda_{\chi})} \arrow["\overleftarrow{h}"]{ld} \arrow["\overrightarrow{h}"]{rd} & \\
  \Bun_{G}^{e,\lambda_{x}} & & \Bun_{G,\Spd(C)}^{e,\lambda_{x}/\lambda_{\chi}}.
\end{tikzcd}
\]

From here, the same argument as in the proof of \cite[Lemma 5.3.2]{Zou24} holds, replacing the $\tn{Bl}_{b}$-pullback square loc. cit. (for $b \in B_{e}(G)$ with $\lambda_{b}=\lambda_{x}/\lambda_{\chi}$, which in the setting of \cite{Zou24} is an element of $B(G)$) with
\[
   \begin{tikzcd}
    \tn{Gr}_{Z \subset G}^{e,(\lambda_{x},\lambda_{x}/\lambda_{\chi})} \arrow["\tn{BL}^{e}_{b}"]{r} \arrow{d} & \Hck_{Z \subset G}^{e,1,(\lambda_{x},\lambda_{x}/\lambda_{\chi})} \arrow["\overrightarrow{h}"]{d} \\
    \{b\} \arrow{r} & \Bun_{G,\Spd(C)}^{e,\lambda_{x}/\lambda_{\chi}},
\end{tikzcd}
\]
where $\tn{BL}_{b}^{e}$ is defined using the variant of Beauville--Laszlo from Lemma \ref{lem:GerbeBL} (cf. also the proof of Proposition \ref{prop:Heckeproper}, identifying the pullback of $b$ to the disk with $\widetilde{\Q}_{\lambda_{x}/\lambda_{\chi},C}$) and then using Proposition \ref{prop:mainGrisom} to identify $\tn{Gr}_{Z \subset G}^{e,(\lambda_{x},\lambda_{x}/\lambda_{\chi})}$ with $\tn{Gr}_{G/Z,\Spd(C)}^{\lambda_{\chi}}$.
\end{proof}

\subsection{Extended categorical conjectures}

With the spectral action in place, we can now formulate the extended variant of the Fargues--Scholze categorical conjecture. Let $(U,\psi)$ be a Whittaker datum for $G$, let us denote by $\Wcal_{\psi} \in \mc{D}_{\tn{lis}}(\Bun_{G},\Lambda)$ the Whittaker sheaf, i.e. 
$$\Wcal_{\psi} = i_{1,!}(c-\ind_{U(F)}^{G(F)}(\psi))$$
where $i_1 : \Bun_G^1 \hookrightarrow \Bun_G$ is the inclusion of the stratum corresponding to $b = 1$. Let us denote by $\Wcal_{\psi}^e$ its extension to $\Bun_G^e$ along the closed and open inclusion $\Bun_G \subset \Bun_G^e$.

Let us denote by 
$$a_{\psi}^e : \Ind\Perf(\Par_G^e) \to \mc{D}_{\tn{lis}}(\Bun_{G}^e,\Lambda)$$
the functor $\mathcal{F} \mapsto \mathcal{F} * \Wcal_{\psi}^e$ and by $c_{\psi}^e : \mc{D}_{\tn{lis}}(\Bun_{G}^e,\Lambda) \to \Ind\Perf(\Par_G^{e})$ its right adjoint. These are the extended versions of the functors 
$$a_{\psi} : \Ind\Perf(\Par_G) \leftrightarrows \mc{D}_{\tn{lis}}(\Bun_{G},\Lambda) : c_{\psi}$$
defined analogously. 

\begin{conj}[Categorical local Langlands equivalence, \protect{\cite[Conjecture X.3.5]{Geometrization}}]\label{conj:CLLC}
    For $\Lambda = \ov{\Z_{\ell}}[\frac{1}{n}]$ where $n := |\pi_{1}(\widehat{G})_{\tn{tor}}|$, the functor $c_{\psi}$ induces an equivalence of categories 
    $$\mc{D}_{\tn{lis}}(\Bun_{G},\Lambda)^{\omega} \cong \Coh_{\nilp}^{\tn{qc}}(\Par_G),$$
    where the right-hand side is the category of coherent sheaves with quasi-compact support and nilpotent singular support. 
\end{conj}

\begin{conj}[Extended categorical local Langlands equivalence]\label{conj:extended-categorical-equivalence}
    Under the same assumptions as Conjecture \ref{conj:CLLC}, the functor $c_{\psi}^e$ induces an equivalence of categories 
    $$\mc{D}_{\tn{lis}}(\Bun_{G}^e,\Lambda)^{\omega} \cong \Coh_{\nilp}^{\tn{qc}}(\Par_G^e).$$
\end{conj}

As a sanity check for Conjecture \ref{conj:extended-categorical-equivalence}, we establish the case of tori in Section \ref{sec:case-of-tori}.

\begin{thm}\label{thm:case-of-tori}
    Conjecture \ref{conj:extended-categorical-equivalence} holds when $G$ is a torus.
\end{thm}

\begin{rque}
    The case of tori for the ``classical'' Fargues--Scholze equivalence was established by \cite{Zou24}; the proof of \Cref{thm:case-of-tori} works in the same way. 
\end{rque}

One of the main tools for the proof of \Cref{thm:case-of-tori} is the following  analogue of \cite[Lemma 5.3.3]{Zou24}. Recall that by Lemma \ref{lem:Konradlem1} (and \cite[Proposition 4.3.1]{Zou24}), the category $\tn{Ind} \Perf(\Par_{G}^{e})$ has a natural $X^{*}(Z_{\widehat{G}}^{\Gal_{F},+})$-grading.

\begin{lem}\label{lem:gradedaction}
The functor
\begin{equation*}
a_{\psi}^e : \tn{Ind} \Perf(\Par_{G}^{e}) \to \mc{D}_{\tn{lis}}(\Bun_{G}^{e},\Lambda)
\end{equation*}
preserves the decomposition into $X^{*}(Z_{\widehat{G}}^{\Gal_{F},+}) \cong \pi_{1}(G)^{e}_{\Gal_{F}}$-graded pieces after twisting the isomorphism $X^{*}(Z_{\widehat{G}}^{\Gal_{F},+}) \xrightarrow{\sim} \pi_{1}(G)^{e}_{\Gal_{F}}$ by $-1$.
\end{lem}

\begin{proof}
    Using Theorem \ref{thm:relative-tensor-product} to write 
    \begin{equation*}
        \tn{IndPerf}(\Par_{G}^{e}) = \Ind\Perf(\Par_G) \otimes_{\Ind\Perf(\pt/\widehat{G})} \Ind\Perf(\pt/\widehat{G}^e),
    \end{equation*}
    we deduce from the universal property of the relative tensor product that it suffices to prove the result for the two separate action maps
    \begin{equation*}
        \tn{IndPerf}(\tn{pt}/\widehat{G}^{e}) \to \mc{D}_{\tn{lis}}(\Bun_{G}^{e},\Lambda), \hspace{1mm} \tn{IndPerf}(\Par_{G}) \to \mc{D}_{\tn{lis}}(\Bun_{G}^{e},\Lambda)
    \end{equation*}
    from Theorem \ref{thm:extended-spectral-action-on-bun_G}, where we use the $X^{*}(Z_{\widehat{G}}^{\Gal_{F},+})$-grading on $\tn{IndPerf}(\tn{pt}/\widehat{G}^{e})$ and the usual $X^{*}(Z_{\widehat{G}}^{\Gal_{F}})$-grading on $\tn{IndPerf}(\Par_{G})$. The former case follows from Lemma \ref{lem:Konradlem1} and the latter case from \cite[Lemma 5.3.2]{Zou24}.
\end{proof}

Recall that there is a well-defined map $\Par_G^e \to \pt/\widehat{G}^e$. 

\begin{lem}\label{lem:reduction-on-the-coherent-side}
    Under the hypotheses of Theorem \ref{thm:relative-tensor-product}, the functor 
    $$\IndCoh_{\tn{nilp}}^{\tn{qc}}(\Par_G) \otimes_{\Ind\Perf(\pt/\widehat{G})} \Ind\Perf(\pt/\widehat{G}^e) \to \IndCoh_{\nilp}^{\tn{qc}}(\Par_G^e)$$
    is an equivalence. 
\end{lem}

\begin{proof}
    It follows from Theorem \ref{thm:relative-tensor-product} that the natural functor 
    $$\IndCoh^{\tn{qc}}(\Par_G) \otimes_{\Ind\Perf(\pt/\widehat{G})} \Ind\Perf(\pt/\widehat{G}^e) \to \IndCoh^{\tn{qc}}(\Par_G^e)$$
    is an equivalence. As coherent singular support is detected after pullback to $\Par_G^{\square}$, we get the desired equivalence.
\end{proof}

It follows from Lemma \ref{lem:reduction-on-the-coherent-side} that Conjecture \ref{conj:extended-categorical-equivalence} can be deduced from the ``usual'' Fargues--Scholze categorical conjecture if one proves the following conjecture. 

\begin{conj}\label{conj:reduction-extended-to-classical}
    The functor 
    $$\mc{D}_{\tn{lis}}(\Bun_{G},\Lambda) \otimes_{\Ind\Perf(\pt/\widehat{G})} \Ind\Perf(\pt/\widehat{G}^e) \to\mc{D}_{\tn{lis}}(\Bun_{G}^e,\Lambda)$$
    is an equivalence, where $\tn{IndPerf}(\pt/\widehat{G})$ acts on $\mc{D}_{\tn{lis}}(\Bun_{G},\Lambda)$ through the spectral action and the pullback along $\Par_G \to \pt/\widehat{G}$.
\end{conj}

\begin{thm}\label{thm:reduction-connected-center-case}
    Conjecture \ref{conj:reduction-extended-to-classical} holds when $G$ has connected center and $H^1(F,Z_G) = 0$.
\end{thm}

This is proved in \S \ref{sec:connectedZredpf}.

\begin{rque}
    As suggested by the discussion in Section \ref{sec:reduction-perspectives}, the primary missing input expected for a full reduction theorem is the compatibility of the spectral action with the action of the center. Once this compatibility is established, the general reduction conjecture should reduce to the connected-center case.
\end{rque}

\begin{rque}
    The hypothesis $H^1(F,Z_G) = 0$ in the statement of \Cref{thm:reduction-connected-center-case} comes from the dependency on \cite{Zou26} in our proof of the theorem. It is expected that the main results of \emph{loc. cit.} hold without that assumption. We expect that once the stronger version of \cite{Zou26} is established, we should be able to remove the cohomological vanishing assumption on the center. Note however that for the proof we provide the connectedness of the center is crucial. 
\end{rque}

Since the categorical local Langlands correspondence for $\GL_n$ has been established by \cite{HansenMann26}, we can state the following corollary: 
\begin{corol}
    Let $G = \GL_n$, $F$ a finite extension of $\mathbb{Q}_p$ and $\Lambda = \Qlb$. If \cite[Conjecture 1.6.2]{HansenMann26} holds then Conjecture \ref{conj:extended-categorical-equivalence} holds. 
\end{corol}

\begin{proof}
    It is an immediate consequence of \cite[Theorem 1.1.1]{HansenMann26} and \Cref{thm:reduction-connected-center-case}.
\end{proof}

\begin{rque}
    Conjecture 1.6.2 in \cite{HansenMann26} is the compatibility of the functor $c_{\psi}$ with geometric Eisenstein series functors which is expected to hold in full generality. 
\end{rque}

\subsection{Proof of the conjecture for tori}\label{sec:case-of-tori}

The goal of this subsection is to prove \Cref{thm:case-of-tori}. The version of this theorem for the ``classical'' Fargues--Scholze equivalence was established by \cite{Zou24}. Since the details are essentially the same, we will provide them only when the argument differs in a substantive way, and otherwise just refer to \cite{Zou24}.

Fix $T$ a torus defined over $F$. Observe that $\Bun_{T}^{e}$ has a unique Whittaker sheaf $\mc{W}^{e} \in \mc{D}_{\tn{lis}}(\Bun_{T}^{e},\Lambda)^{\omega}$, defined by pushing forward the Whittaker sheaf $\mc{W}$ (\cite[Definition 6.2.1]{Zou24}) on the closed and open $\Bun_{T} = \Bun_{T}^{e,1} \hookrightarrow \Bun_{T}^{e}$.

\begin{proof}[Proof of Theorem \ref{thm:case-of-tori}]
The identical argument in \cite[Lemma 6.2.5]{Zou24} applied to the $\widehat{T}^{\Gal_{F},+}$-gerbe $\Par_{T}^{e}$ over the algebraic space $\Hom(T(F),\mathbb{G}_{m})$ (as defined in \cite[Lemma 4.1.2]{Zou24}), replacing Lemma 5.3.3 loc. cit. with our Lemma \ref{lem:gradedaction}, implies that it suffices to prove the result for the induced map
\begin{equation}\label{eq:zeropieceSA}
    \Perf^{\tn{qc}}(\Par_{T}^{e})_{0} \to \mc{D}(T(F)\tn{-$\tn{Mod}_{\Z_{\ell}}$})^{\omega},
\end{equation}
where $\Perf^{\tn{qc}}(\Par_{T}^{e})_{0}$ is the $0$-isotypic summand of $\Perf^{\tn{qc}}(\Par_{T}^{e})$ under the $X^{*}(\widehat{T}^{\Gal_{F},+})$-grading from \cite[Lemma 4.3.2]{Zou24}.

By construction (cf. Theorem \ref{thm:extended-spectral-action-on-bun_G} and the proof of Theorem \ref{thm:extended-hecke-action}), for any $G$ the action of $\Perf(\Par_{G})$ via the pullback map $\Perf(\Par_{G}) \to \Perf(\Par_{G}^{e})$ is the usual spectral action from \cite{Geometrization}. Moreover, any object of $\Perf^{\tn{qc}}(\Par_{T}^{e})_{0}$ has trivial action by $\widehat{T}^{\Gal_{F},+}$ by definition and thus descends to $\Par_{T}$ (recall that $\Par_{T}^{e} = \pt/\widehat{T}^{e} \times_{\pt/\widehat{T}} \Par_{T}$). We conclude that the map \eqref{eq:zeropieceSA} is the same one from the proof of \cite[Theorem 6.4.1]{Zou24}. The result then follows from the proof loc. cit.
\end{proof}

\section{Around the reduction conjecture}\label{sec:reduction-theorem}
Throughout the next two subsections section we assume that $Z_{G}$ is connected.

\subsection{The Picard groupoid $\Bun_{Z_{G}}^{e}$}
There is a natural action map 
\begin{equation*}
    \Bun_{G} \times \Bun_{Z_{G}}^{e} \xrightarrow{a} \Bun_{G}^{e}
\end{equation*}
given by sending $(\mc{P},\mc{U})$ to $\mc{U} \cdot \mc{P}$.
\begin{prop}\label{prop:quotient-groupoid-by-center}
    The map $a$ induces an isomorphism of v-stacks
    \begin{equation*}
        \Bun_{G} \times^{\Bun_{Z_{G}}} \Bun_{Z_{G}}^{e} \xrightarrow{\sim} \Bun_{G}^{e}.
    \end{equation*}
\end{prop}

\begin{proof}
   Evidently the map factors through the quotient by the anti-diagonal $\Bun_{Z_{G}}$-action, and the non-empty fibers of $a$ are $\Bun_{Z_{G}}$-torsors. Surjectivity on geometric points is the same, by Proposition \ref{prop:BunGestructure}, as the surjectivity of
   \begin{equation*}
       B(G) \times B_{e}(Z_{G}) \to B_{e}(G),
   \end{equation*}
   which follows immediately from the surjectivity of $B_{e}(Z_{G}) \to \Hom_{F}(u,Z_{G})$ (this is implicit in the short exact sequence \eqref{eq:pi1eSES} which, as recalled loc. cit., is proved in \cite[Section 9]{Fargues22}). Note that this is where we use the assumption that $G$ has connected center.
   
   Translating by $b_{\lambda}$ in $\Bun_{Z_{G}}^{e}(C)$ lifting a given $\lambda \in \Hom_{F}(u,Z_{G})$ then reduces the claimed result to the same claim for the action $\Bun_{G} \times \Bun_{Z_{G}} \to \Bun_{G}$, where it is clear.
\end{proof}

For $T$ an $F$-rational torus, denote by $(\mc{D}_{\tn{lis}}(\Bun_{T}^{e},\Lambda),\star)$ the symmetric monoidal category with the monoidal structure induced by the multiplication map on the Picard groupoid $\Bun_{T}^{e}$. There is a canonical action of $\Bun_{Z_G}^e$ on $\Bun_G^e$.

\begin{corol}\label{corol:relative-tensor-over-center}
    There is a canonical equivalence 
    \begin{equation*}
     \mc{D}_{\tn{lis}}(\Bun_{G}^{e},\Lambda) \cong 
     \mc{D}_{\tn{lis}}(\Bun_{G},\Lambda) \otimes_{
     \mc{D}_{\tn{lis}}(\Bun_{Z_{G}},\Lambda)} 
     \mc{D}_{\tn{lis}}(\Bun_{Z_{G}}^{e},\Lambda).
    \end{equation*}
\end{corol}

\begin{proof}
    By grouping together components on both sides and translating by $b_{\lambda} \in \Bun_{Z_{G}}^{e}(C)$ for $\lambda \in \Hom_F(u, Z_G)$ if necessary as in Proposition \ref{prop:quotient-groupoid-by-center}, we are reduced to 
    $$\mc{D}_{\tn{lis}}(\Bun_{G},\Lambda) \cong 
     \mc{D}_{\tn{lis}}(\Bun_{G},\Lambda) \otimes_{
     \mc{D}_{\tn{lis}}(\Bun_{Z_{G}},\Lambda)} 
     \mc{D}_{\tn{lis}}(\Bun_{Z_{G}},\Lambda).$$
\end{proof}

\begin{prop}[Extended  analogue of \protect{\cite[Theorem 3.4]{Zou26}}]\label{prop:equivalence-monoidal-categories}
    If $H^{1}(F,T)=0$, then the spectral action of Theorem \ref{thm:extended-spectral-action-on-bun_G} induces an equivalence of symmetric monoidal categories
    \begin{equation*}
        (\mc{D}_{\tn{lis}}(\Bun_{T}^{e},\Lambda),\star) \to (\QCoh(\Par_{T}^{e}), \otimes).
    \end{equation*}
\end{prop}

\begin{proof}
    First, observe that there is (by Proposition \ref{prop:BunGestructure}) an exact sequence of Picard stacks
    \begin{equation*}
        1 \to \pt/T(F) \to \Bun_{T}^{e} \to \underline{B_{e}(T)} \to 1
    \end{equation*}
    which (regardless of the assumptions on $T$) splits after base-changing to $\Perfd_{C}$ via pulling back along the canonical map $[\mc{T}_{C}/\tilde{t}_{C}] \to \Kott \times \Kal$. 

    On the dual side, \cite[Lemma 2.5]{Zou26} constructs an isomorphism $\Par_{T} \xrightarrow{\sim} \Hom(T(F), \mathbb{G}_{m}) \times \pt/\widehat{T}^{\Gal_{F}}$ compatible with the projections to $\pt/\widehat{T}$, and since $\Par_{T}^{e} = \pt/\widehat{T}^{e} \times_{\pt/\widehat{T}} \Par_{T}$ this induces an isomorphism 
    \begin{equation*}
         \Par_{T}^{e} \xrightarrow{\sim} \Hom(T(F), \mathbb{G}_{m}) \times \pt/\widehat{T}^{\Gal_{F},+}.
    \end{equation*}
    From here the proof of \cite[Theorem 3.4]{Zou26} holds verbatim.
\end{proof}

\subsection{Proof of Theorem \ref{thm:reduction-connected-center-case}} \label{sec:connectedZredpf}

The goal of this subsection is to prove \Cref{thm:reduction-connected-center-case}, that is, we want to prove that the following map (cf. Conjecture \ref{conj:reduction-extended-to-classical})
\begin{equation}\label{eq:reduction-conneced-center}
    \mc{D}_{\tn{lis}}(\Bun_{G},\Lambda) \otimes_{\Ind\Perf(\pt/\widehat{G})} \Ind\Perf(\pt/\widehat{G}^e) \to\mc{D}_{\tn{lis}}(\Bun_{G}^e,\Lambda)
\end{equation}
is an equivalence assuming $G$ has connected center and $H^1(F, Z_G) = 0$. 

\begin{lem}\label{lem:reduction-to-center}
If $G$ has connected center, then there exists a canonical isomorphism of stacks 
    \begin{equation*}\label{eq:reduction-to-center}
        \pt/\widehat{G} \times_{\pt/\widehat{Z_G}} \pt/\widehat{Z_{G}}^e = \pt/\widehat{G}^e.
    \end{equation*}
\end{lem}

\begin{proof}
    This follows from the fact that $\widehat{G}^e = \widehat{G}_{\tn{sc}} \times \varprojlim_n \widehat{Z_G}$ where the transitions are given by $z \mapsto z^n$. 
\end{proof}

It follows from Lemma \ref{lem:reduction-to-center} that the left-hand side of \eqref{eq:reduction-conneced-center} is then equivalent to 
\begin{equation}\label{eq:reduction-connected-center-part-2}
     \mc{D}_{\tn{lis}}(\Bun_{G},\Lambda) \otimes_{\Ind\Perf(\pt/\widehat{G})} \Ind\Perf(\pt/\widehat{G}^e) \cong \mc{D}_{\tn{lis}}(\Bun_{G},\Lambda) \otimes_{\Ind\Perf(\pt/\widehat{Z_G})} \Ind\Perf(\pt/\widehat{Z_G}^e).
\end{equation}

\begin{proof}[Proof of \Cref{thm:reduction-connected-center-case}]
    We have the following chain of equivalences 
    \begin{align*}
        &\mc{D}_{\tn{lis}}(\Bun_{G},\Lambda) \otimes_{\Ind\Perf(\pt/\widehat{G})} \Ind\Perf(\pt/\widehat{G}^e) \\
        &\cong \mc{D}_{\tn{lis}}(\Bun_{G},\Lambda) \otimes_{\Ind\Perf(\pt/\widehat{Z_G})} \Ind\Perf(\pt/\widehat{Z_G}^e) \\
        &\cong \mc{D}_{\tn{lis}}(\Bun_{G},\Lambda) \otimes_{\Ind\Perf(\Par_{Z_G})} \Ind\Perf(\Par_{Z_G}^e) \\
        &\cong \mc{D}_{\tn{lis}}(\Bun_{G},\Lambda) \otimes_{\mc{D}_{\tn{lis}}(\Bun_{Z_G},\Lambda)} \mc{D}_{\tn{lis}}(\Bun_{Z_G}^e,\Lambda) \\
        &\cong \mc{D}_{\tn{lis}}(\Bun_{G}^e,\Lambda).
    \end{align*}
    The maps are as follows:
    \begin{enumerate}
        \item the first equivalence is \eqref{eq:reduction-connected-center-part-2},
        \item the second map is \eqref{eq:relative-tensor-product} applied to $Z_G$, 
        \item the third map follows from the equivalence of Proposition \ref{prop:equivalence-monoidal-categories} and the fact that in the non-extended setting, this action is compatible with the action of $\mc{D}_{\tn{lis}}(\Bun_{Z_{G}},\Lambda)$ on $\mc{D}_{\tn{lis}}(\Bun_{G},\Lambda)$, which is \cite[Theorem 4.1]{Zou26}, 
        \item the last equivalence follows from Corollary \ref{corol:relative-tensor-over-center};
    \end{enumerate}
    we have thus proven the result.
\end{proof}

\subsection{A $2$-categorical Fourier--Mukai heuristic}\label{sec:reduction-perspectives}

Conjecture \ref{conj:reduction-extended-to-classical} predicts that the automorphic category for the extended stack can be recovered from the usual Fargues--Scholze category by extension of scalars from $\Ind\Perf(\pt/\widehat{G})$ to $\Ind\Perf(\pt/\widehat{G}^{e})$. The purpose of this subsection is to give an informal discussion about how a variant of the $2$-categorical Fourier--Mukai methods of \cite{GLCV} suggests such an equivalence. 

In the final steps of the proof of the Geometric Langlands conjectures, the authors of \cite{GLCV} reduce from a general reductive group $G$ to $G$ semisimple and then to $G$ semisimple and simply connected. The second reduction uses a $2$-categorical Fourier--Mukai transform and produces variants of the Geometric Langlands equivalence for unramified inner forms \cite[Corollary 8.10.2]{GLCV}. Since the formalism of rigid inner forms is designed to organize the inner forms of $G$, it is natural to expect an analogue of this construction in the Fargues--Scholze setting. We assume that $\Lambda$ is torsion, that $G$ is semisimple, and that $F$ has characteristic zero. We temporarily set $Z:=Z_G$.

\begin{defi}
    Let $\Gamma$ be a finite flat abelian group scheme over $F$. We define $\Ge(\Gamma)$ to be the stack of gerbes banded by $\Gamma$, given by
    \begin{align*}
        S \in \Perfd \mapsto \{\tn{\'{e}tale gerbes over} \ X_S\}.
    \end{align*}
\end{defi}

Dually, we denote by $\Ge^{\alg}(\pi_1(\widehat{G})(1))$ the condensed algebraic stack over $\Lambda$ given by
$$\Ge^{\alg}(\pi_1(\widehat{G})(1)) := \Maps(B\Weil_F, B^2(\pi_1(\widehat{G})(1))).$$
We make the following assumptions about sheaves of categories on these stacks:
\begin{enumerate}
    \item The stack $\Ge^{\alg}(\pi_1(\widehat{G})(1))$ has trivial condensed structure and is determined by a single $\infty$-groupoid.
    \item The stack $\Ge(Z)$ is discrete, meaning that as a stack over $\Perfd$ it is pulled back from an $\infty$-groupoid, and there is an equivalence between sheaves of categories on $\Ge(Z)$ and $\Ge^{\alg}(Z)$.
    \item There is a well-defined $2$-categorical Fourier--Mukai duality of algebraic stacks over $\Lambda$ (see \cite[Section 8.1]{GLCV})
    \begin{equation}\label{eq:two-fm-transform}
        \Ge^{\alg}(\pi_1(\widehat{G})(1)) \times \Ge^{\alg}(Z) \to B^2\Gm.
    \end{equation}
\end{enumerate}
These assumptions are reasonable. The last two are analogues of the descriptions in \cite[Sections 8.3 and 8.4]{GLCV}, while the first should follow from the finiteness of $\pi_1(\widehat{G})(1)$.

Following \cite{GLCV}, we introduce two sheaves of categories:
\begin{enumerate}
    \item Consider the map $\pi^{\autom} : \Bun_{G_{\ad}} \to \Ge(Z)$. As in \cite[Section 8.6.3]{GLCV}, define
    $$\mc{D}^{\autom} := \pi^{\autom}_*\mc{D}_{\et,\Bun_{G_{\ad}}},$$
    where $\mc{D}_{\et,\Bun_{G_{\ad}}}$ is the sheaf of categories whose global sections are $\mc{D}_{\et}(\Bun_{G_{\tn{ad}}})$. By the assumption above, we view it as a sheaf of categories on $\Ge^{\alg}(Z)$.
    \item Since $\Par_G$ is $1$-affine, the spectral action of $\Perf(\Par_G)$ on $\mc{D}(\Bun_G)$ determines a sheaf of categories $\mc{D}_{\Par_G}$ on $\Par_G$\footnote{For this informal discussion, we ignore the differences between $\QCoh$ and $\Ind\Perf$.}. Consider the map $\pi^{\spec} : \Par_G \to \Ge^{\alg}(\pi_1(\widehat{G})(1))$. As in \cite[Section 8.6.6]{GLCV}, define
    $$\mc{D}^{\spec} := \pi^{\spec}_*\mc{D}_{\Par_G}.$$
\end{enumerate}

\begin{conj}[Analogue of \protect{\cite[Theorem 8.6.8]{GLCV}}]\label{conj:two-categorical-fourier-mukai-transform}
    Under the $2$-categorical Fourier--Mukai transform \eqref{eq:two-fm-transform}, we have
    \begin{equation}
        2\tn{-FM}(\mc{D}^{\spec}) = \mc{D}^{\autom},
    \end{equation}
    up to a sign twist.
\end{conj}

Let us relate Conjecture \ref{conj:two-categorical-fourier-mukai-transform} to Conjecture \ref{conj:reduction-extended-to-classical}. Recall that $u$ denotes the group banding the gerbe $\Kal$ and that there is a canonical isomorphism
$$\Hom_F(u, Z) = X^*(\pi_1(\widehat{G})(1)).$$
We also have
$$\Bun_G^e = \bigsqcup_{\lambda \in \Hom_F(u, Z)} \Bun_G^{e, \lambda}.$$
In particular, there is a Cartesian diagram
\[\begin{tikzcd}
	{\Bun_{G_{\ad}}} & {\Bun_{G}^e} \\
	{\Ge(Z)} & {\Hom_F(u,Z)}
	\arrow["{\pi^{\autom}}"', from=1-1, to=2-1]
	\arrow[from=1-2, to=1-1]
	\arrow["{\pi^e}", from=1-2, to=2-2]
	\arrow["{\ev_{\Kal}}", from=2-2, to=2-1]
\end{tikzcd}\]
Dually, there is a diagram
\[\begin{tikzcd}
	{\Par_G^e} & {\Par_G} & \\
	{B\widehat{G}^e} & {B\widehat{G}} & {\Ge^{\alg}(\pi_1(\widehat{G})(1))} \\
	\pt & {B^2\pi_1(\widehat{G})(1)}
	\arrow[from=1-1, to=1-2]
	\arrow[from=1-1, to=2-1]
	\arrow[from=1-2, to=2-2]
	\arrow[from=1-2, to=2-3]
	\arrow[from=2-1, to=2-2]
	\arrow[from=2-1, to=3-1]
	\arrow[from=2-2, to=3-2]
	\arrow["\ev"{description}, from=2-3, to=3-2]
	\arrow["q"', from=3-1, to=3-2]
\end{tikzcd}\]
where $\ev : \Ge^{\alg}(\pi_1(\widehat{G})(1)) \to B^2\pi_1(\widehat{G})(1)$ is induced by the map $\pt \to B\Weil_F$. This map is $1$-affine. Assuming Conjecture \ref{conj:two-categorical-fourier-mukai-transform} and the compatibilities of $2$-categorical Fourier--Mukai transforms established in \cite[Section 8.4]{GLCV}, we obtain an equivalence
\begin{equation}\label{eq:reduction-fm}
    q^*\ev_{\pt,*}\mc{D}^{\spec} \cong \ev_{\Kal}^*\mc{D}^{\autom}.
\end{equation}

Unwinding the definitions, the left-hand side is
$$q^*\ev_{\pt,*}\mc{D}^{\spec} = \mc{D}(\Bun_G) \otimes_{\Ind\Perf(\pt/\widehat{G})} \Ind\Perf(\pt/\widehat{G}^e),$$
while the right-hand side is
$$\ev_{\Kal}^*\mc{D}^{\autom} = \mc{D}(\Bun_G^e).$$
Thus \eqref{eq:reduction-fm} gives an equivalence of the form predicted by Conjecture \ref{conj:reduction-extended-to-classical}. Even assuming Conjecture \ref{conj:two-categorical-fourier-mukai-transform}, it remains unclear whether this equivalence agrees with the natural functor appearing in Conjecture \ref{conj:reduction-extended-to-classical}.

\printbibliography

\end{document}